\documentclass[reqno]{amsart}
\usepackage{amsthm, amssymb, amsfonts}
\usepackage{lmodern}
\usepackage{color}
\usepackage{hyperref}
\usepackage[T5]{fontenc}
\usepackage{amscd,amssymb}
\usepackage[v2,cmtip]{xy}
\usepackage{mathrsfs}
\theoremstyle{plain}
\newtheorem{thm}{Theorem}[section]

\newtheorem{con}[thm]{Conjecture}
\newtheorem{corl}[thm]{Corollary}
\theoremstyle{definition}

\theoremstyle{plain}
\newtheorem{thms}{Theorem}[subsection]
\newtheorem{props}[thms]{Proposition}
\newtheorem{lems}[thms]{Lemma}
\newtheorem{corls}[thms]{Corollary}
\theoremstyle{definition}

\newtheorem{defns}[thms]{Definition}

\newtheorem{notas}[thms]{Notation}

\allowdisplaybreaks
\begin{document} 
	
\title[The squaring operation and the hit problem]
{The squaring operation and the hit problem\\ for the polynomial algebra in a type\\ of generic degree}
\author{Nguy\~\ecircumflex n Sum}

\address{Department of Mathematics and Applications, S\`ai G\`on University, 273 An D\uhorn \ohorn ng V\uhorn \ohorn ng, District 5, H\`\ocircumflex\ Ch\'i Minh city, Viet Nam}

\email{nguyensum@sgu.edu.vn}



\thanks{The author was supported in part by the National Foundation for Science and Technology Development (NAFOSTED) of Viet Nam under the grant number 101.04-2017.05}

\maketitle

\begin{abstract}
Let $P_k$ be the graded polynomial algebra $\mathbb F_2[x_1,x_2,\ldots ,x_k]$ with the degree of each generator $x_i$ being 1, where $\mathbb F_2$ denote the prime field with two elements. 
	
The \textit{hit problem} of Frank Peterson asks for a minimal generating set for the polynomial algebra $P_k$ as a module over the mod-2 Steenrod algebra $\mathcal{A}$. Equivalently, we want to find a vector space basis for $\mathbb F_2 \otimes_{\mathcal A} P_k$ in each degree.  
	
In this paper, we study a generating set for the kernel of Kameko's squaring operation $\widetilde{Sq}^0_*: \mathbb F_2 \otimes_{\mathcal A} P_k \longrightarrow \mathbb F_2 \otimes_{\mathcal A} P_k$ in a so-called \textit{generic degree}. By using this result, we explicitly compute the hit problem for $k=5$ in the respective generic degree.
\end{abstract}

\bigskip
\begin{center}
	\textit{(In memory of Professor Reginald Wood)}
\end{center}

\medskip
\section{Introduction}\label{s1} 
\setcounter{equation}{0}

Denote by $P_k:= \mathbb F_2[x_1,x_2,\ldots ,x_k]$ the polynomial algebra over the field with two elements, $\mathbb F_2$, in $k$ generators $x_1, x_2, \ldots , x_k$, each of degree 1. This algebra arises as the cohomology with coefficients in $\mathbb F_2$ of a classifying space of an elementary abelian 2-group $V_k$ of rank $k$.  Therefore, $P_k$ is a module over the mod-2 Steenrod algebra, $\mathcal A$.   
The action of $\mathcal A$ on $P_k$ is determined by the elementary properties of the Steenrod squares $Sq^i$ and subject to the Cartan formula
$Sq^n(fg) = \sum_{i=0}^nSq^i(f)Sq^{n-i}(g),$
for $f, g \in P_k$ (see Steenrod and Epstein~\cite{st}).

A polynomial $f$ in $P_k$ is called \textit{hit} if it can be written as a finite sum $f = \sum_{i> 0}Sq^{i}(h_i)$ for suitable polynomials $h_i\in P_k$. That means $f$  belongs to  $\mathcal{A}^+P_k$, where $\mathcal{A}^+$ denotes the augmentation ideal in $\mathcal A$. 

We study the \textit{Peterson hit problem} of determining a minimal set of generators for the polynomial algebra $P_k$ as a module over the Steenrod algebra. Equivalently, we want to find a vector space basis for the quotient 
\[QP_k := P_k/\mathcal A^+P_k = \mathbb F_2 \otimes_{\mathcal A} P_k.\] 

The Peterson hit problem is an open problem in Algebraic Topology. It was first studied by Peterson~\cite{pe}, Priddy~\cite{pr}, Singer~\cite {si1} and  Wood~\cite{wo}, who showed its relation to several classical problems respectively in cobordism theory, modular representation theory of general linear groups, Adams spectral sequence for the stable homotopy of spheres, and stable homotopy type of classifying spaces of finite groups. Then, this problem  was studied by Carlisle and Wood~\cite{cw}, Crabb and Hubbuck~\cite{ch}, Kameko~\cite{ka,ka2}, Mothebe \cite{mo2}, Nam~\cite{na}, Peterson~\cite{pe1}, Repka and Selick~\cite{res}, Silverman~\cite{sl}, Silverman and Singer~\cite{ss}, Singer~\cite{si2}, Walker and Wood~\cite{wa1,wa2}, Wood~\cite{wo2} and others.

\smallskip
Let $GL_k$ be the general linear group over the field $\mathbb F_2$. Since $V_k$ is an $\mathbb F_2$-vector space of dimension $k$, this group acts naturally on $V_k$ and therefore on the cohomology $P_k$ of $BV_k$. The two actions of $\mathcal A$ and $GL_k$ upon $P_k$ commute with each other. Hence, there is an inherited action of $GL_k$ on $QP_k$. 

The vector space $QP_k$ was explicitly calculated by Peterson~\cite{pe} for $k=1, 2,$ by Kameko~\cite{ka} for $k=3$ and by Kameko \cite{ka2} and the present author \cite{su2} for $k = 4$, unknown in general. Recently, the hit problem and its applications to representations of general linear groups have been presented in the monographs of Walker and Wood \cite{wa3, wa4}. 

\smallskip
For a positive integer $n$, by $\mu(n)$ one means the smallest number $r$ for which it is possible to write $n = \sum_{1\leqslant i\leqslant r}(2^{u_i}-1)$ with $u_i >0$. By a simple computation, we can see that  $\mu(n) = s$ if and only if there exists a unique sequence of integers $d_1 > d_2 >\ldots > d_{s-1}\geqslant d_s>0$ such that 
\begin{equation} \label{ct1.1}n =  2^{d_1} + 2^{d_2}+ \ldots + 2^{d_{s-1}}+ 2^{d_{s}} - s = \sum_{1\leqslant i \leqslant s}(2^{d_i}-1), 
\end{equation}
(see e.g. \cite[Lemma 2]{su3} for a proof). From this it implies that $n-s$ is even and that $\mu(\frac{n-s}2) \leqslant s$, where $s = \mu(n)$. 

Based on the results of Wood \cite{wo} and Kameko \cite[Theorem 4.2]{ka}, the hit problem is reduced to the case of degree $n$ of the form \eqref{ct1.1} with $\mu(n) = s < k$. 

The hit problem in the case of degree $n$ of the form \eqref{ct1.1} with $s=k-1$,  was studied by Crabb and Hubbuck~\cite{ch}, Nam~\cite{na}, Repka and Selick~\cite{res}, Walker and Wood \cite{wa2} and the present author \cite{su,su2}. 

For $s = k - 2$,   in \cite{su}, we studied the kernel of Kameko's squaring operation $\widetilde{Sq}^0_*: QP_k \to QP_k$. This operation is induced by the $\mathbb F_2$-linear map $\varphi:P_k\to P_k$, given by
\[\varphi(x) =  \left\{
\begin{array}{l@{\enskip}l}
y, &\mbox{ if }x=x_1x_2\ldots x_ky^2,\\
0,  &\mbox{ otherwise,}
\end{array}\right.\]
for any monomial $x \in P_k$. Note that $\varphi$ is a homomorphism of $GL_k$-modules but it is not an $\mathcal A$-homomorphism. However, 
$\varphi Sq^{2t} = Sq^{t}\varphi$ and $\varphi Sq^{2t+1} = 0$ for any non-negative integer $t$. So, for each positive integer $n$ such that $n-k$ is even, $\varphi$ induced a homomorphism of $GL_k$-modules:
\[(\widetilde{Sq}^0_*)_{(k,n)}:= \widetilde{Sq}^0_*: (QP_k)_{n}\to (QP_k)_{\frac{n-k}2}.\]
Here and in what follows, we denote by $(P_k)_n$ the subspace of $P_k$ consisting of the homogeneous polynomials of degree $n$ in $P_k$ and  $(QP_k)_n$ the subspace of $QP_k$ consisting of all the classes represented by the elements in $(P_k)_n$.

Since $(\widetilde{Sq}^0_*)_{(k,n)}$ is a homomorphism of $GL_k$-modules, $\mbox{Ker}(\widetilde{Sq}^0_*)_{(k,n)}$ gives a representation of $GL_k$. We have gave a prediction for the dimension of $\mbox{Ker}(\widetilde{Sq}^0_*)_{(k,n)}$ in this case.

\begin{con}[See \cite{su}] \label{ker} Let  $n= \sum_{i=1}^{k-2}(2^{d_i}-1)$ with $d_i$ positive integers. If $d_{i-2} - d_{i-1} > i$ for  $3 \leqslant i \leqslant k - 1$ and $d_{k-2}  > k \geqslant 3$, then 
\[\dim\mbox{\rm Ker}(\widetilde{Sq}^0_*)_{(k,n)} = \prod_{3\leqslant i \leqslant k}(2^i-1).\]
\end{con}

This conjecture is true for $k \leqslant 4$ and unknown for $k \geqslant 5$. 

In \cite{su,su2}, we have studied the hit problem in case of the degree $n$ with $\mu(n)= s = k-1$ by using the strictly inadmissible monomials and Singer's criterion in \cite{si2} on hit monomials. However, these tools are not enough to study this problem in the case of the degree $n$ with $\mu(n)=s = k-2$.

In this paper, based on Silverman's criterion in \cite{sl2} on hit polynomials, we introduce the notion of strongly inadmissible monomial to construct a generating set for the kernel of Kameko's squaring operation.   One of our main results is Theorem \ref{dlck2} in Section~\ref{s3} which provides the upper bound on the dimension of $\mbox{Ker}(\widetilde{Sq}^0_*)_{(k,n)}$. By using this result, we verify Conjecture \ref{ker} for $k=5$. We prove the following. 

\begin{thm}\label{dl20} Let $n = 2^{d+s+t} + 2^{d + s} + 2^d -3$ with $d,\, s,\, t$ non-negative integers. If  $d \geqslant 6$ and $t,\, s \geqslant 4$, then
\begin{equation}\label{ct130}\dim\mbox{\rm Ker}(\widetilde{Sq}^0_*)_{(5,n)} = (2^3-1)(2^4-1)(2^5-1) = 3255.
\end{equation}
\end{thm}

Thus, Conjecture \ref{ker} is true for $k =5$. Based on Theorem \ref{dl20} and our result in \cite[Theorem 1.4]{su2}, one gets the following.

\begin{corl}\label{cor4240} Let $n$ be as in Theorem $\ref{dl20}$. If $d \geqslant 6$ and $s,\, t \geqslant 4$, then
$$\dim (QP_5)_n = 4(2^3-1)(2^4-1)(2^5-1) = 13020.$$
\end{corl}

We show in Section \ref{s4} that many of the difficulties encountered when using strictly inadmissible monomials are overcome by using strongly inadmissible monomials. Thus, the notion of strongly inadmissible monomial should be a useful new tool in studying the Peterson hit problem.

\medskip
This paper is organized as follows.
In Section \ref{s2}, we recall  some needed information on the admissible monomials in $P_k$ and criteria of Singer \cite{si2} and Silverman \cite{sl2} on hit monomials. In Section \ref{s3}, we present the results for a generating set of the kernel of Kameko's squaring operation. As an application of the results of Section \ref{s3}, in Section \ref{s4},  we prove that Conjecture \ref{ker} is true for $k=5$. Finally, in Section \ref{s5} we list the needed admissible monomials of degree $3(2^d-1)$ in $P_5$.

\section{Preliminaries 
}\label{s2}
\setcounter{equation}{0}

In this section, we recall some results from Kameko ~\cite{ka},  Singer ~\cite{si2}, Silverman~\cite{sl2} and our work \cite{su2} which will be used in the next sections.

\subsection{The weight vector and the admissible monomials}\

\begin{notas} 
In the paper, we use the following notations.
\begin{align*}
\mathbb N_k &= \{1,2, \ldots , k\},\\
X_{\mathbb J} &= X_{j_1,j_2,\ldots , j_s} =
\prod_{j\in \mathbb N_k\setminus \mathbb J}x_j , \ \ \mathbb J = \{j_1,j_2,\ldots , j_s\}\subset \mathbb N_k.
\end{align*}
In particular, we have
$X_\emptyset = x_1x_2\ldots x_k,\, X_j = x_1\ldots\hat x_j\ldots x_{k}, 1\leqslant j \leqslant k,\, X_{\mathbb N_k} =1.$
		
Denote by $\alpha_i(a)$ the $i$-th coefficient in the dyadic expansion of a non-negative integer $a$. That means
\[a= \alpha_0(a)2^0+\alpha_1(a)2^1+\alpha_2(a)2^2+ \ldots ,\]
for $ \alpha_i(a) =0$ or 1 and $i\geqslant 0.$ Denote by $\alpha(a)$ the number of 1's in the dyadic expansion of $a$.
		
Let $x=x_1^{a_1}x_2^{a_2}\ldots x_k^{a_k} \in P_k$. We denote $\nu_j(x) = a_j, 1 \leqslant j \leqslant k$ and $\nu(x) = \max\{\nu_j(x):1 \leqslant j \leqslant k\}$.  
We set 
\[\mathbb J_i(x) = \{j \in \mathbb N_k :\alpha_i(\nu_j(x)) =0\},\] 
for $i\geqslant 0$. Then, we have
\[x = \prod_{i\geqslant 0}X_{\mathbb J_i(x)}^{2^i}.\] 
\end{notas}
\begin{defns} A weight vector $\omega$ is a sequence of non-negative integers $(\omega_1,\omega_2$, $\ldots , \omega_i, \ldots)$ such that $\omega_i = 0$ for $i \gg 0$.
For any monomial  $x$ in $P_k$, we define two sequences associated with $x$ by
\begin{align*} 
\omega(x)&=(\omega_1(x),\omega_2(x),\ldots , \omega_i(x), \ldots),\\
\sigma(x) &= (\nu_1(x),\nu_2(x),\ldots ,\nu_k(x)),
\end{align*}
where
$\omega_i(x) = \sum_{1\leqslant j \leqslant k} \alpha_{i-1}(\nu_j(x))= \deg X_{\mathbb J_{i-1}(x)},\ i \geqslant 1.$
The sequences $\omega(x)$ and $\sigma(x)$ are respectively called the weight vector and the exponent vector of $x$. 

The set of weight vectors (respectively exponent vectors) is given the left lexicographical order.  

For a weight vector $\omega = (\omega_1,\omega_2,\ldots)$, define $\deg \omega = \sum_{i > 0}2^{i-1}\omega_i$ and the length $\ell(\omega) = \max\{i : \omega_i >0\}$. Then, we write $\omega = (\omega_1,\omega_2,\ldots, \omega_r)$ if $\ell(\omega) = r$.  For a weight vector $\eta = (\eta_1,\eta_2, \ldots)$, we define the concatenation of weight vectors 
\[\omega|\eta = (\omega_1,\ldots,\omega_r,\eta_1,\eta_2,\ldots)\] 
if $\ell(\omega) = r$ and $(a)|^b = (a)|(a)|\ldots|(a)$, ($b$ times of $(a)$'s), where $a,\, b$ are positive integers.  
We denote $P_k(\omega)$ the subspace of $P_k$ spanned by monomials $y$ such that
$\deg y = \deg \omega$ and $\omega(y) \leqslant \omega$, and by $P_k^-(\omega)$ the subspace of $P_k(\omega)$ spanned by monomials $y$  such that $\omega(y) < \omega$. 
\end{defns}
Denote by $\mathcal A(s-1)$ the sub-Hopf algebra of $\mathcal A$ generated by  $Sq^i$ with $0\leqslant i < 2^s$, and $\mathcal A(s-1)^+ = \mathcal A^+\cap\mathcal A(s-1)$.

\begin{defns}\label{dfn2}  For $\omega$ a weight vector and $f, g$ two polynomials of the same degree in $P_k$, we define 
		
i) $f \equiv g$ if and only if $f + g \in \mathcal A^+P_k$. If $f \equiv 0$, then $f$ is said to be \textit{hit}.
		
ii) $f \equiv_{\omega} g$ if and only if $f + g \in \mathcal A^+P_k+P_k^-(\omega).$
		
iii) $f \simeq_{(s,\omega)} g$ if and only if $f + g \in  \mathcal A(s-1)^+P_k+P_k^-(\omega)$. 
\end{defns}

Obviously, the relations $\equiv$, $\equiv_{\omega}$ and $\simeq_{(s,\omega)}$ are equivalence relations. For $s=0$, we have $f \simeq_{(0,\omega)} g$ if and only if $f + g \in P_k^-(\omega)$. 

For $x$ a monomial in $P_k$ and $\omega = \omega(x)$, we denote $x \simeq_{s}g$ if and only if $x \simeq_{(s,\omega(x))}g$. 

Denote by $QP_k(\omega)$ the quotient of $P_k(\omega)$ by the equivalence relation $\equiv_\omega$. Following \cite{su3}, we have 
\begin{equation}\label{ct21} (QP_k)_n \cong \bigoplus_{\deg \omega = n}QP_k(\omega).\end{equation}

For any polynomial $f$ in $P_k$, we denote $[f]$ the class in $QP_k$ represented by $f$. For a subset $S \subset P_k$, we denote 
$[S] = \{[f] : f \in S\} \subset QP_k.$ If $f \in P_k(\omega)$ and $S \subset P_k(\omega)$, then  we denote by $[f]_\omega$ the class in $QP_k(\omega)$ represented by $f$ and 
\[ [S]_\omega = \{[f]_\omega : f \in S\} \subset QP_k(\omega).\]

We recall some elementary properties on the action of the Steenrod squares on $P_k$.

\begin{props}\label{mdcb1} Let $f$ be a homogeneous polynomial in $P_k$. 
	
{\rm i)} If $i > \deg f$, then $Sq^i(f) =0$. If $i = \deg f$, then $Sq^i(f) =f^2$.
	
{\rm ii)} If $i$ is not divisible by $2^s$, then $Sq^i(f^{2^s}) = 0$ while $Sq^{r2^s}(f^{2^s}) = (Sq^r(f))^{2^s}$.
\end{props}

\begin{props}[Kameko {\cite[Lemma 3.1]{ka}}]\label{bdkbs} Let $x$ be a monomial in $P_k$ and $n,\, s$ be positive integers such that $0<n<2^s$. If $v$ is a monomial in $P_k$ which appears as a term in the polynomial $Sq^n(x)$, then there is an index $i \leqslant s$ such that $\omega_i(v) < \omega_i(x)$ and $\omega(v) < \omega(x)$.	
\end{props}

\begin{defns}\label{defn3} 
Let $x, y$ be monomials of the same degree in $P_k$. We define $x <y$ if and only if one of the following holds:  
		
i) $\omega (x) < \omega(y)$;
		
ii) $\omega (x) = \omega(y)$ and $\sigma(x) < \sigma(y).$
\end{defns}

\begin{defns}
A monomial $x$ is said to be inadmissible if there exist monomials $y_1,y_2,\ldots, y_r$ such that $y_j<x$ for $j=1,2,\ldots , r$ and $x \equiv \sum_{j=1}^ry_j.$ 
		
A monomial $x$ is said to be admissible if it is not inadmissible.
\end{defns}

Obviously, the set of all the admissible monomials of degree $n$ in $P_k$ is a minimal set of $\mathcal{A}$-generators for $P_k$ in degree $n$. 

\begin{defns} 
A monomial $x$ is said to be strictly inadmissible if and only if there exist monomials $y_1,y_2,\ldots, y_r$ such that $y_j<x,$ for $j=1,2,\ldots , r$ and 
$x \simeq_s \sum_{j=1}^r y_j $ with $s = \max\{i \, :\, \omega_i(x) > 0\}$.
\end{defns}

It is easy to see that if $x$ is strictly inadmissible, then it is inadmissible. The following theorem is a modification of a result in \cite{ka}.

\begin{thms}[Kameko \cite{ka}, Sum \cite{su}]\label{dlcb1}  
For any monomials $x, y, w$  in $P_k$ such that $\omega_i(x) = 0$ for $i > r>0$, $\omega_s(w) \ne 0$ and $\omega_i(w) = 0$ for $i > s>0$, we have
	
{\rm i)}  If  $w$ is inadmissible, then so is $xw^{2^r}$.
	
{\rm ii)}  If $w$ is strictly inadmissible, then so is $xw^{2^r}y^{2^{r+s}}$.
\end{thms}

\begin{props}[See \cite{su}]\label{mdcb3} Let $x$ be an admissible monomial in $P_k$ and let $i_0$ be a positive integer. Then we have
	
{\rm i)} If $\omega_{i_0}(x)=0$, then $\omega_{i}(x)=0$ for all $i > i_0$.
	
{\rm ii)} If $\omega_{i_0}(x)<k$, then $\omega_{i}(x)<k$ for all $i > i_0$.
\end{props}

For $1 \leqslant i \leqslant k$, define a homomorphism $f_i: P_{k-1} \to P_k$ of $\mathcal A$-algebras by substituting
\begin{equation}\label{ct22}
f_i(x_u) = \begin{cases}x_u, &\mbox{ if } 1 \leqslant u <i,\\ x_{u+1},&\mbox{ if } i \leqslant u <k. \end{cases}
\end{equation}
\begin{props}[See Mothebe and Uys \cite{mo2}]\label{mdmo} Let $i, d$ be positive integers such that $1 \leqslant i \leqslant k$. If $w$ is an admissible monomial in $P_{k-1}$, then $x_i^{2^d-1}f_i(w)$ is also an admissible monomial in $P_{k}$.
\end{props}

\subsection{Some criteria for hit monomials.}\

\medskip
Firstly, we recall Singer's criterion on hit monomials in $P_k$. 
\begin{defns} A monomial $z$ in $P_k$ is called a spike if $\nu_j(z)=2^{s_j}-1$ for $s_j$ a non-negative integer and $j=1,2, \ldots , k$. 
If $z$ is a spike with $s_1>s_2>\ldots >s_{r-1}\geqslant s_r>0$ and $s_j=0$ for $j>r,$ then it is called the minimal spike.
\end{defns}

Note that if $\mu(n) = s$, then $n$ is of the form \eqref{ct1.1} and $z = \prod_{i=1}^sx_i^{2^{d_i}-1}$ is the minimal spike of degree $n$. It is easy to show that a spike of degree $n$ is the minimal spike if its weight vector order is minimal with respect to other spikes of degree $n$. The following is a criterion for hit monomials in $P_k$.

\begin{thms}[See Singer~\cite{si2}]\label{dlsig} Suppose $x \in P_k$ is a monomial of degree $n$, where $\mu(n) \leqslant k$. Let $z$ be the minimal spike of degree $n$. If $\omega(x) < \omega(z)$, then $x$ is hit.
\end{thms}

We remark that this criterion is not enough to determine all hit monomials. For example, it can be shown that $x_1^{15}x_2^3x_3^3$ is the minimal spike of degree $21$ and $x_1x_2^5x_3^5x_4^5x_5^5$ is hit, but $\omega(x_1x_2^5x_3^5x_4^5x_5^5) = (5,0,4,0) > (3,3,1,1) = \omega(x_1^{15}x_2^3x_3^3)$.
So, we need Silverman's criterion for hit polynomials in $P_k$.
\begin{thms}[{See Silverman~\cite[Theorem 1.2]{sl2}}]\label{dlsil} Let $p$ be a polynomial of the form $fg^{2^m}$ for some homogeneous  polynomials $f$ and $g$. If $\deg f < (2^m - 1)\mu(\deg g)$, then $p$ is hit.
\end{thms}

This result leads to a criterion in terms of the minimal spike which strengthens Theorem \ref{dlsig}.

\begin{thms}[{See Walker and Wood~\cite[Theorem 14.1.3]{wa3}}]\label{dlww} Let $x \in P_k$ be a monomial of degree $n$, where $\mu(n)\leqslant k$ and let $z$ be the minimal spike of degree $n$. If there is an index $h$ such that $\sum_{i=1}^{h}2^{i-1}\omega_i(x) < \sum_{i=1}^{h}2^{i-1}\omega_i(z),$ then $x$ is hit.
\end{thms}

For $1 \leqslant r \leqslant k$, we set 
\begin{align*} 
P_s^0 &=\langle\{x=x_1^{a_1}x_2^{a_2}\ldots x_s^{a_s} \ : \ a_1a_2\ldots a_s=0\}\rangle,\\ 
P_s^+ &= \langle\{x=x_1^{a_1}x_2^{a_2}\ldots x_s^{a_s} \ : \ a_1a_2\ldots a_s>0\}\rangle. 
\end{align*}

It is easy to see that $P_s^0$ and $P_s^+$ are the $\mathcal{A}$-submodules of $P_k$, $P_s = P_s^0\oplus P_s^+$ and $QP_s = QP_s^0\oplus QP_s^+$, where $QP_s^0 = P_s^0/\mathcal A^+P_s^0$ and $QP_s^+ = P_s^+/\mathcal A^+P_s^+$. 

For $J= (j_1, j_2, \ldots, j_s) : 1 \leqslant j_1 <\ldots < j_s \leqslant k$, we define a monomorphism $\theta_J: P_s \to P_k$ of $\mathcal A$-algebras by substituting $\theta_J(x_u) = x_{j_u}$ for $1 \leqslant u \leqslant s$. It is easy to see that, for any weight vector $\omega$ of degree $n$, 
\[Q\theta_J(P_s^+)(\omega) \cong  QP_s^+(\omega)\mbox{ and } (Q\theta_J(P_s^+))_n \cong (QP_s^+)_n\] 
for $1 \leqslant s \leqslant k$, where $Q\theta_J(P_s^+) = \theta_J(P_s^+)/\mathcal A^+\theta_J(P_s^+)$. So, by a simple computation using Theorem~ \ref{dlsig} and (\ref{ct21}), we get the following.
\begin{props}[{See Walker and Wood~\cite[Proposition 6.2.9]{wa3}}]\label{mdbs} For a weight vector $\omega$ of degree $n$, we have direct summand decompositions of the $\mathbb F_2$-vector spaces
\begin{align*} QP_k(\omega)  &= \bigoplus_{\mu(n) \leqslant s\leqslant k}\bigoplus_{\ell(J) =s}Q\theta_J(P_s^+)(\omega), 
\end{align*}
where $\ell(J)$ is the length of $J$. Consequently, 
\begin{align*}
\dim QP_k(\omega) &= \sum_{\mu(n) \leqslant s\leqslant k}{k\choose s}\dim QP_s^+(\omega),\\
\dim (QP_k)_n &= \sum_{\mu(n) \leqslant s\leqslant k}{k\choose s}\dim (QP_s^+)_n.
\end{align*}
\end{props}

\begin{notas}
We denote by $B_{k}(n)$ the set of all admissible monomials of degree $n$  in $P_k$, $B_{k}^0(n) = B_{k}(n)\cap P_k^0$, $B_{k}^+(n) = B_{k}(n)\cap P_k^+$. For a weight vector $\omega$ of degree $n$, we set $B_k(\omega) = B_{k}(n)\cap P_k(\omega)$, $B_k^+(\omega) = B_{k}^+(n)\cap P_k(\omega)$. 
Then, $[B_{k}^0(n)]$, $[B_{k}^+(n)]$, $[B_k(\omega)]_\omega$ and $[B_k^+(\omega)]_\omega$ are respectively bases of the $\mathbb F_2$-vector spaces $(QP_k^0)_n$, $(QP_k^+)_n$, $QP_k(\omega)$ and $QP_k^+(\omega) := QP_k(\omega)\cap QP_k^+$.

For any $(i;I)$ with $I= (i_1,i_2,\ldots,i_r)$, $1\leqslant i < i_1< i_2<\ldots<i_r\leqslant k$, $0\leqslant r < k$, define a homomorphism $p_{(i;I)}: P_k \to P_{k-1}$ of algebras by substituting
\begin{equation}\label{ct23}
p_{(i;I)}(x_j) =\begin{cases} x_j, &\text{ if } 1 \leqslant j < i,\\
\sum_{t=1}^rx_{i_t-1}, &\text{ if }  j = i,\\  
x_{j-1},&\text{ if } i< j \leqslant k.
\end{cases}
\end{equation}
Then $p_{(i;I)}$ is a homomorphism of $\mathcal A$-modules.
These homomorphisms will be used in the proof of Theorem \ref{dl51}.
\end{notas}

\section{On the kernel of Kameko's squaring operation}\label{s3}
\setcounter{equation}{0}

In this section, we consider $n =\sum_{i=1}^{k-2}(2^{d_i}-1)$ with $d_i$ positive integers such that $d_1 > d_2 > \ldots > d_{k-3} \geqslant d:=d_{k-2}>0$, $k \geqslant 4$, $m = \sum_{i=1}^{k-3}(2^{d_i-d}-1)=\beta_k^d(n)$, where the function $\beta_k:\mathbb Z \to \mathbb Z$ is defined by $\beta_k(t) = \frac{t-k+2}2$ if $t-k+2$ is even and $\beta_k(t) = 0$ if $t-k+2$ is odd. Note that $\mu(n) = k-2$ and this degree is used in Conjecture \ref{ker} on the dimension of the kernel of Kameko's squaring operation
\[(\widetilde{Sq}^0_*)_{(k,n)}: (QP_k)_n \to (QP_k)_{\frac{n-k}2}.\]

The main result of the section is Theorem \ref{dlck2} that provides an upper bound for the dimension of $\mbox{\rm Ker}((\widetilde{Sq}^0_*)_{(k,n)})$.

Firstly, we prove some properties of monomials in $\mbox{\rm Ker}((\widetilde{Sq}^0_*)_{(k,n)})$ from which we can reduce the computations to the case of weight vector $(k-2)|^d$. In Subsection \ref{s32}, we present the notion of strongly inadmissible monomial and prove Proposition \ref{mdcb51} that is used in place of Theorem \ref{dlcb1}. Note that the notion of strongly inadmissible monomial is weaker than that of strictly inadmissible monomial, so using this notion can overcome many difficulties encountered when using the notion of strictly inadmissible monomial. In Subsection \ref{s33}, we prove our main result by using the results in the previous subsections.

\subsection{Some properties of monomials in $\mbox{\rm Ker}((\widetilde{Sq}^0_*)_{(k,n)})$}\label{s31}\

\medskip
In this subsection we present some properties of the admissible monomials in the kernel of Kameko's squaring operation that allow us to reduce the study of this subspace to the case of weight vector $(k-2)|^d$. 
\begin{lems}\label{bdbt1} 
If $x$ is an admissible monomial of degree $n$ in $P_k$ such that $[x] \in \mbox{\rm Ker}((\widetilde{Sq}^0_*)_{(k,n)})$, then $\omega_i(x) = k-2$ for $1 \leqslant i \leqslant d_{k-2}$.
\end{lems}
\begin{proof} Note that $z = \prod_{t=1}^{k-2}x_t^{2^{d_t}-1}$ is the minimal spike of degree $n$ and $\omega_i(z) = k-2$ for $1 \leqslant i \leqslant d_{k-2}$. Since $x$ is admissible, $[x] \ne 0$. If $\omega_1(x) = k-1$, then $x = X_jy^2$ with $y$ a monomial of degree $\frac{n-k+1}2= \beta_k(n)+\frac 12$, however this is not an integer. So, by Theorem \ref{dlsig}, we have either $\omega_1(x) = k-2$ or $\omega_1(x) = k$. If $\omega_1(x) = k$, then $x = X_{\emptyset}y^2$ with $y$ a monomial in $P_k$. Since $x$ is admissible, by Theorem \ref{dlcb1}, $y$ is also admissible. Hence, $(\widetilde{Sq}^0_*)_{(k,n)}([x]) = [y] \ne 0$. This contradicts the hypothesis that $x \in \mbox{\rm Ker}((\widetilde{Sq}^0_*)_{(k,n)})$, hence $\omega_1(x) = k-2$. Then, we have $x = X_{j,\ell}y^2$ with $1 \leqslant j < \ell \leqslant k$ and $y$ an admissible monomial of degree $\beta_k(n) =\sum_{i=1}^{k-2}(2^{d_i-1}-1)$. Since $\omega_1(y)\ne k-1$, using Theorem \ref{dlsig} and Proposition \ref{mdcb3} we get $\omega_2(x) = \omega_1(y) = k-2$. By repeating the above argument we obtain $\omega_i(x) = k-2$ for $1 \leqslant i \leqslant d_{k-2}$. The lemma is proved.
\end{proof}

\begin{lems}\label{bdbt2} If $x$ is a monomial of degree $n$ in $P_k$ such that $[x] \in \mbox{\rm Ker}((\widetilde{Sq}^0_*)_{(k,n)})$, then $x \equiv \sum \bar x$ with $\bar x$ monomials in $P_k$ such that $\omega_i(\bar x) = k - 2$, for $1 \leqslant i \leqslant d_{k-2}$.
\end{lems}
\begin{proof} If $\omega_1(x) < k-2$, then by Theorem \ref{dlsig}, $x$ is hit, hence the lemma holds. Suppose $\omega_1(x) = k-2$ and let $s>1$ be the smallest index such that $\omega_s(x) \ne k - 2$. If  $\omega_s(x) < k - 2$, then by Theorem \ref{dlsig}, $x$ is hit, hence the lemma holds. Since $\omega_1(x) \ne k-1$, we obtain $\omega_s(x) = k$. Then we have
$x = wy^{2^{s-2}}\prod_{t\geqslant s}X_{\mathbb J_t(x)}^{2^t}$, where $w = \prod_{t=0}^{s-3}X_{\mathbb J_t(x)}^{2^t}$, $y = X_{\mathbb J_{s-2}(x)}X_{\mathbb J_{s-1}(x)}^2 = X_{\mathbb J_{s-2}(x)}X_{\emptyset}^2 = X_{u,v}^3x_u^2x_v^2$ with $1 \leqslant u < v \leqslant k$. It is easy to see that
\begin{equation}\label{ctbs1}
y = \sum_{i\ne u,v}X_{i,u,v}^3x_ux_v^2x_i^4 + Sq^1(X_{u,v}^3x_ux_v^2).
\end{equation} 
By combining Proposition \ref{mdcb1} and the Cartan formula, we have \begin{align*}&w(Sq^1(X_{u,v}^3x_ux_v^2))^{2^{s-2}} = wSq^{2^{s-2}}\left((X_{u,v}^3x_ux_v^2)^{2^{s-2}}\right)\\
&\quad = Sq^{2^{s-2}}\left(w(X_{u,v}^3x_ux_v^2)^{2^{s-2}}\right) + Sq^{2^{s-2}}(w)(X_{u,v}^3x_ux_v^2)^{2^{s-2}}.
\end{align*} 
Let $\bar w$ be a monomial which appears as a term in $Sq^{2^{s-2}}(w)$. By Proposition \ref{bdkbs}, $\omega(\bar w) < \omega (w)= (k-2)|^{s-2}$. Hence, using Theorem \ref{dlsig} we see that the polynomial $w(Sq^1(X_{u,v}^3x_ux_v^2))^{2^{s-2}}\prod_{t\geqslant s}X_{\mathbb J_t(x)}^{2^t}$ is hit. So, from the relation \eqref{ctbs1} we obtain $x \equiv \sum_{i\ne u,v}x_{(i,u,v)}$, where 
\[x_{(i,u,v)} = \prod_{0\leqslant t\leqslant s-3}X_{\mathbb J_t(x)}^{2^t}(X_{i,u,v}^3x_ux_v^2x_i^4)^{2^{s-2}}\prod_{t\geqslant s}X_{\mathbb J_t(x)}^{2^t}.\]
A simple computation shows that $\omega_t(x_{(i,u,v)}) = k -2$ for $1 \leqslant t \leqslant s$. By repeating this argument we see that the lemma is true in this case.
	
If $\omega_1(x) = k$, then $x = X_\emptyset y^2$ with $y$ a monomial in $P_k$. Then, we have  $(\widetilde{Sq}^0_*)_{(k,n)}([x]) = [y] = 0$. Hence, 
$y = \sum_{r > 0} Sq^r(g_r)$ with suitable polynomial $g_r$ in $P_k$. Then, by using Proposition \ref{mdcb1} and the Cartan formula, we get 
\begin{align*}
x &= X_\emptyset y^2 = \sum_{r > 0} X_\emptyset Sq^{2r}(g_r^2)\\
&=\sum_{r>0}Sq^{2r}(X_\emptyset g_r^2) + \sum_{r>0}\sum_{t=1}^{r}Sq^{2t}(X_\emptyset)(Sq^{r-t}(g_r))^2.
\end{align*}
Since $\deg (X_\emptyset) = k$, $Sq^{2t}(X_\emptyset) = 0$ for $2t > k$. If $2t \leqslant k$ and $w$ is a monomial which appears as a term of $Sq^{2t}(X_\emptyset)$, then $\omega_1(w) = k-2t\leqslant k-2$. Hence, from the above equality and Theorem \ref{dlsig}, we get $x \equiv \sum x'$  with $x'$ monomials in $(P_k)_n$ such that $\omega_1(x') = k - 2$. The lemma is proved.
\end{proof}

From this lemma, it suffices to consider monomials $x$ such that $\omega_i(x) = k-2$ for $1\leqslant i \leqslant d=d_{k-2}$. Then
\[ x = \prod_{t=1}^dX_{i_t,j_t}^{2^{d-t}}y^{2^d},\]
where $1 \leqslant i_t < j_t \leqslant k$, $1 \leqslant t \leqslant d$ and $y \in (P_k)_m$. Note that $\omega(x) = (k-2)|^d|\omega(y)$.

\begin{lems}\label{bdlh} Let $x$ be a monomial of degree $(k-2)(2^d-1)$. If $\omega_1(x) < k$ and there is $r > d$ such that $\omega_r(x) > 0$, then $x \in P_k^-((k-2)|^d) + \mathcal A(d-1)^+P_k$.
\end{lems}
\begin{proof} If $\omega_1(x) < k-2$, then $x \in P_k^-((k-2)|^d)$, hence the lemma holds. Since $\omega_1(x)\ne k-1$, if $\omega_1(x) \geqslant k-2$, then $\omega_1(x) = k-2$. Let $s$ be the smallest index such that $\omega_s(x) > k-2$. Since $\omega_s(x)\ne k-1$, we have $\omega_s(x) = k$. If $s \geqslant d$, then $(k-2)(2^d-1)= \deg x \geqslant (k-2)(2^d-1) + 2^{t-1}\omega_t(x)> (k-2)(2^d-1)$. This is a contradiction, so $s < d$. If there is $1 <r < s$ such that $\omega_r(x) < k-2$, then $x \in P_k^-((k-2)|^d)$, so the lemma holds. Suppose $\omega_r(x) = k-2$ for $1 \leqslant r < s$ and $X_{\mathbb J_{s-2}(x)} = X_{u,v}$ for $1 \leqslant u < v \leqslant k$. Then, we have 
\[X_{\mathbb J_{s-1}(x)}^2X_{\mathbb J_{s-2}(x)} = X_{u,v}^3x_u^2x_v^2 = \sum_{i \ne u,v}X_{i,u,v}^3x_i^4x_ux_v^2 + Sq^1(X_{u,v}^3x_ux_v^2).\]
By an argument analogous to the one in the proof of Lemma \ref{bdbt2}, 
 we get 
 $$x = \sum_{i\ne u,v}x_{(i,u,v)}\quad \mbox{mod}(P_k^-((k-2)|^d) + \mathcal A(d-1)^+P_k),$$ where 
\[x_{(i,u,v)} =
\prod_{t=0}^{s-3}X_{\mathbb J_t(x)}^{2^t}(X_{i,u,v}^3x_ux_v^2x_i^4)^{2^{s-2}}\prod_{t\geqslant s}X_{\mathbb J_t(x)}^{2^t}.\]
It is easy to see that $\omega_t(x_{(i,u,v)}) = k -2$ for $1 \leqslant t \leqslant s$ and $\omega_r(x_{(i,u,v)}) > 0$ for suitable $r > d$. By repeating this argument we obtain 
$$x =\sum \bar x \quad \mbox{mod}(P_k^-((k-2)|^d) + \mathcal A(d-1)^+P_k) $$ with $\bar x$ monomials such that $\omega_t(\bar x) \leqslant k -2$ for $1 \leqslant t \leqslant d$ and $\omega_r(\bar x) > 0$ for suitable $r > d$. Then we have $\sum_{i=1}^d2^{i-1}\omega_i(\bar x) < \deg \bar x = (k-2)(2^d-1)$. Hence, there is an index $u \leqslant d$ such that $\omega_t(\bar x) = k-2$ for $1\leqslant t < u$, $\omega_u(\bar x) < k-2$, therefore $\bar x \in P_k^-((k-2)|^d)$. The lemma is proved. 
\end{proof}

\subsection{Strongly inadmissible monomials}\label{s32}\

\medskip
In this subsection, we introduce the notion of strongly inadmissible monomial in $P_k$ and use it to study the kernel of Kemeko's squaring operation. 
 
\begin{defns} 
Let $n$ be a positive integer and $z$ be the minimal spike of degree $n$. Denote by $\mathcal P_{(k,n)}$ the subspace of $P_k$ spanned by all monomials $x$ of degree $n$ such that
\[\sum_{1 \leqslant j\leqslant h}2^{j-1}\omega_j(x) < \sum_{1 \leqslant j\leqslant h}2^{j-1}\omega_j(z),\]
for some  index $h \geqslant 1$.
\end{defns}

\begin{defns}\label{dnstin}
A monomial $x$ of degree $n$ in $P_k$ is said to be strongly inadmissible if there exist  monomials $y_1, y_2,\ldots , y_t$ of the same weight vector $\omega(x)$ such that $y_u < x,\, 1\leqslant u \leqslant t$ and
\[ x \simeq_{s} y_1 + y_2 + \ldots + y_t\ \mbox{ mod}(\mathcal P_{(k,n)}),\]
where $s = \max\{i:\omega_i(x)>0\}$. 
\end{defns}

Obviously, if $x$ is strictly inadmissible, then it is strongly inadmissible. By using Theorem \ref{dlww}, we see that if $g \in \mathcal P_{(k,n)}$, then $g \in\mathcal A^+P_k$. Hence, if $x$ is strongly inadmissible, then it is inadmissible. However, if $g \notin \mathcal A(s-1)^+P_k$, then $x$ is not strictly inadmissible. Therefore, the use of strongly inadmissible monomials is more convenient than that of the strictly inadmissible monomials. It can overcome many difficulties encountered when using the notion of strictly inadmissible monomial.

For example, let $x= x_1x_2^3x_3^6x_4^6x_5^5$ be the monomial of weight vector $(3)|^3$ in $P_5$. We have 
\begin{align*}x &= x_1x_2^{3}x_3^{5}x_4^{6}x_5^{6} + x_1x_2^{3}x_3^{6}x_4^{5}x_5^{6} + x_1x_2^{5}x_3^{5}x_4^{5}x_5^{5}\\ &\qquad  + Sq^1(x_1^{2}x_2^{3}x_3^{5}x_4^{5}x_5^{5}) + Sq^2(x_1x_2^{3}x_3^{5}x_4^{5}x_5^{5})\ \mbox{mod}(P_5^-((3)|^3)),
\end{align*}
where $x_1x_2^{5}x_3^{5}x_4^{5}x_5^{5} \in \mathcal P_{(5,21)}$, hence $x$ is strongly inadmissible. It is easy to see that  
\begin{align*}x_1x_2^{5}x_3^{5}x_4^{5}x_5^{5} &= Sq^2(x_1x_2^{3}x_3^{5}x_4^{5}x_5^{5} + x_1x_2^{5}x_3^{3}x_4^{5}x_5^{5} + x_1x_2^{5}x_3^{5}x_4^{3}x_5^{5} + x_1x_2^{5}x_3^{5}x_4^{5}x_5^{3})\\ &\quad +  Sq^8 (x_1x_2^{3}x_3^{3}x_4^{3}x_5^{3})\ \mbox{mod}(P_5^-((3)|^3)) \in \mathcal A(3)^+P_5 + P_5^-((3)|^3).
\end{align*}
If $x$ is strictly inadmissible, then we must have $x_1x_2^{5}x_3^{5}x_4^{5}x_5^{5} \in \mathcal A(2)^+P_5 + P_5^-((3)|^3)$. However, we have been unable to prove this.

\smallskip
For a positive integer $a$, denote by $\alpha(a)$ the number of ones in the dyadic expansion of $a$ and by $\zeta(a)$ the greatest integer $u$ such that $a$ is divisible by $2^u$. That means $a = 2^{\zeta(a)}b$ with $b$ an odd integer. We set $\delta(a) = a - \alpha(a) - \zeta(a)$.

\begin{props}\label{mdstin} Let $d$ be a positive integer.
If $z^*$ is the minimal spike of degree $n_d := (k-2)(2^d-1)$ and $d > \delta(k-2)$, then $\omega_i(z^*) = k-2$ for $1 \leqslant i \leqslant d-\delta(k-2)$ and $\omega_i(z^*) < k-2$ for $i > d-\delta(k-2)$.
\end{props}
\begin{proof} Set  $s =\alpha(k-2)$. We have
$k-2 = 2^{t_1} + 2^{t_2} + \ldots + 2^{t_{s-1}} + 2^{t_s},$  
where $t_1 > t_2 > \ldots > t_{s-1} > t_s = \zeta(k-2) \geqslant 0$. Then, we obtain
\begin{align*}
(k-2)(2^d-1) &= 2^{d+t_1} + 2^{d+t_2} + \ldots + 2^{d+t_{s-1}} + 2^{d+t_s} - k + 2\\
&= \sum_{1 \leqslant i \leqslant k-2}(2^{e_i}-1),
\end{align*}
where  
\[e_i = \left\{\begin{array}{l@{\enskip}l} d+t_i, &1 \leqslant i < s,\\d+t_s - i+s-1, &s \leqslant i \leqslant k-3,\\ d+t_s - k+s+2, & i = k-2.\end{array}\right.\]
It is easy to see that $e_1 > e_2 > \ldots >e_{k-3} = e_{k-2}=d-\delta(k-2)>0.$ Hence, $z^* = \prod_{i=1}^{k-2}x_i^{2^{e_i-1}}$ is the minimal spike of degree $n_d=(k-2)(2^d-1)$, $\omega_j(z^*) = k-2$ for $1 \leqslant j \leqslant e_{k-2}=d-\delta(k-2)$ and $\omega_i(z^*) < k-2$ for $i > d-\delta(k-2)$. The proposition is proved.
\end{proof}
The following is a refinement of Theorem \ref{dlcb1}.
\begin{props}\label{mdcb51}  Let $c, d, e$ be positive integers and let $u,\, w,\, y \in P_k$ be monomials such that $\omega(u) = (k-2)|^c$,  $\omega(w) = (k-2)|^d$ and $\omega(y) = (k-2)|^e$.
If $w$ is strongly inadmissible, then so is $uw^{2^c}y^{2^{c+d}}$.
\end{props}
\begin{proof} Note that the weight vector of $uw^{2^c}y^{2^{c+d}}$ is $(k-2)|^{c+d+e}$. Since $w$ is strongly inadmissible, there exist monomials $y_1,y_2,\ldots, y_t$ of the same weight vector $(k-2)|^d$, $g_1 \in P_k^-((k-2)|^d)$ and $g_2 \in \mathcal P_{(k,n_d)}$ such that  $y_i<w$ for $i=1,2,\ldots , t$ and 
\[w = y_1 + y_2 + \ldots + y_t + g_1 + g_2 + \sum_{1\leqslant j < 2^d} Sq^j(h_j),\]
where $h_j$ are suitable polynomials in $P_k$ and $n_d = (k-2)(2^d-1)$. Since $j2^c < 2^{c+d}$, by using Proposition \ref{mdcb1} and the Cartan formula we have $$(Sq^j(h_j))^{2^c}y^{2^{c+d}} = Sq^{j2^c}(h_j^{2^c})y^{2^{c+d}} = Sq^{j2^c}\left(h_j^{2^c}y^{2^{c+d}}\right).$$  
Then, combining the Cartan formula and Proposition \ref{mdcb1}, we get
\begin{align*}u(Sq^j(h_j))^{2^c}y^{2^{c+d}} &= Sq^{j2^c}\left(uh_j^{2^c}y^{2^{c+d}}\right)+ \sum_{1 \leqslant r \leqslant j}Sq^{r2^c}(u)\left(Sq^{j-r}(h_jy^{2^{d}})\right)^{2^c}.
\end{align*}
Suppose $v$ is a monomial which appears as a term of $Sq^{r2^c}(u)$. By Proposition \ref{bdkbs}, we have $\omega(v) < \omega(u) = (k-2)|^c$. Hence, 
$$\sum_{1 \leqslant r \leqslant j}Sq^{r2^c}(u)\left(Sq^{j-r}(h_jy^{2^{d}})\right)^{2^c} \in P_k^-((k-2)|^{c+d+e}),$$
for $1\leqslant j < 2^d$. Combining the above equalities gives
\begin{align*}uw^{2^c}y^{2^{c+d}} &= \sum_{1\leqslant i \leqslant t}uy_i^{2^c}y^{2^{c+d}} + ug_1^{2^c}y^{2^{c+d}} + ug_2^{2^c}y^{2^{c+d}}\\
& \quad + \sum_{1\leqslant j < 2^d} Sq^{j2^c}\left(uh_j^{2^c}y^{2^{c+d}}\right)\ \mbox{ mod}(P_k^-((k-2)|^{c+d+e})).
\end{align*}
Since $\omega(u) = (k-2)|^c$, we can easily check that $uy_i^{2^c}y^{2^{c+d}} < uw^{2^c}y^{2^{c+d}}$ for $1 \leqslant i \leqslant t$, $ug_1^{2^c}y^{2^{c+d}} \in P_k^-((k-2)|^{c+d+e})$ and $ug_2^{2^c}y^{2^{c+d}} \in \mathcal P_{(k,n_{c+d+e})}$ with $n_{c+d+e} = (k-2)(2^{c+d+e}-1)$. Hence, the last equality implies that $uw^{2^c}y^{2^{c+d}}$ is strongly inadmissible. 
\end{proof}

\begin{lems}\label{bdtd} Let $f, g \in (P_k)_{n_d}$ be homogeneous polynomials with $n_d = (k-2)(2^d-1)$, and let $y \in (P_k)_m$ be a monomial. If $f \simeq_{(d,(k-2)|^d)} g \mbox{ \rm mod}(\mathcal P_{(k,n_d)})$, then $fy^{2^d} \equiv gy^{2^d}$.
\end{lems}
\begin{proof} Note that $z = \prod_{i=1}^{k-2}x^{2^{d_i}-1}$ is the minimal spike of degree $n$ and $\omega_t(z) = k-2$ for $1 \leqslant t \leqslant d$. Suppose 
\[f = g + g_1 + \sum_{1\leqslant j < 2^d}Sq^j(h_j),\] 
where $g_1\in \mathcal P_{(k,n_d)}$ and suitable polynomials $h_j \in P_k$. By Proposition \ref{mdcb1} and the Cartan formula, 
\[Sq^j(h_j)y^{2^d} = Sq^j(h_jy^{2^d}),\ 1 \leqslant j < 2^d.\] 
By Definition \ref{dnstin}, if a monomial $w$ appears as a term of the polynomial $g_1$, then there is an integer $h\geqslant 1$, such that 
\begin{equation}\label{ctbs}\sum_{1\leqslant i\leqslant h} 2^{i-1}\omega_i(w) < \sum_{1\leqslant i\leqslant h} 2^{i-1}\omega_i(z^*)\leqslant \deg(z^*) = n_d,\end{equation}
where $z^*$ is determined as in Proposition \ref{mdstin}. If $h \leqslant d$, then using Proposition \ref{mdstin} we have 
$$\sum_{1\leqslant i\leqslant h} 2^{i-1}\omega_i(z^*) \leqslant (k-2)(2^h-1) = \sum_{1\leqslant i\leqslant h} 2^{i-1}\omega_i(z).$$ 
Since $\omega_i(w) = \omega_i(wy^{2^d})$ for $1\leqslant i \leqslant h \leqslant d$, using Theorem \ref{dlww} we see that $wy^{2^d}$ is hit.
Suppose that $h > d$. Since $\deg w = n_d$, from \eqref{ctbs} we see that there is an index $r > d$ such that $\omega_r(w)>0$. Then, using Lemma \ref{bdlh} we have $w \in P_k^-((k-2)|^d) + \mathcal A(d-1)^+P_k$. This implies that $wy^{2^d}$ is hit. Hence, the polynomial $g_1y^{2^d}$ is hit and $fy^{2^d} \equiv gy^{2^d}$. The lemma is proved.
\end{proof}

\subsection{A construction for $\mathcal A$-generators of $\mbox{\rm Ker}((\widetilde{Sq}^0_*)_{(k,n)})$}\label{s33}\

\begin{notas}\label{kh2} Let $S$ be a finite sequence of positive integers. Then, there are positive integers $c_0, c_1, \ldots , c_r$ and $s_{0},\, s_{1}, \ldots , s_{r}$ such that $s_{i+1} \ne s_{i}$ and $S = (s_{0})|^{c_0}|(s_{1})|^{c_1}|\ldots|(s_{r})|^{c_r}$. We define  
${\rm rl}(S) = c_1 + c_2 + \ldots + c_r$, the reduced length of $S$. For example, with $S= (2,2,3,1,1,1)= (2)|^2|(3)|^1|(1)|^3$, we have $c_0 =2,\, c_1=1,\, c_2 = 3$, hence ${\rm rl}(S) = c_1+c_2=4$.
		
Denote by ${\sf PSeq}_k^d$ the set of all pairs $(\mathcal I,\mathcal J)$ of sequences $\mathcal I = (i_1, i_2,\ldots,i_d)$, $\mathcal J = (j_1, j_2,\ldots,j_d)$, where $i_t,\, j_t$ are integers such that $1 \leqslant i_t < j_t \leqslant k$, for $1 \leqslant t \leqslant d$, and by ${\sf PInc}_k^d$ the set of all  $(\mathcal I,\mathcal J) \in {\sf PSeq}_k^d$ such that $i_1\leqslant i_2 \leqslant \ldots \leqslant i_d$ and $j_1 \leqslant j_2 \leqslant \ldots \leqslant j_{d}$. By convention, ${\sf PSeq}_k^0 = \emptyset$. For $(\mathcal I,\mathcal J)\in {\sf PSeq}_k^d$, we denote
\[X_{(\mathcal I,\mathcal J)} = \prod_{1 \leqslant t \leqslant d}X_{i_t,j_t}^{2^{d-t}} \in P_k((k-2)|^d) \subset (P_k)_{(k-2)(2^d-1)}.\]		
\end{notas}

\begin{defns}\label{gth} Let $d_0$ be a positive integer, $d_0 > 2$, and $\mathcal B$ be a subset of ${\sf PInc}_k^{d_0}$. The set $\mathcal B$ is said to be compatible with $(k-2)|^{d_0}$ if the following conditions hold:
		
\smallskip
i) For any $(\mathcal I, \mathcal J)\in \mathcal B$,  $\mbox{rl}(\mathcal I) \leqslant d_0-2$ and $\mbox{rl}(\mathcal J) \leqslant d_0-2$, 
		
ii) For any $(\mathcal H,\mathcal K) \in {\sf PSeq}_k^{d_0}$, we have 
\begin{equation}\label{ctbd} X_{(\mathcal H,\mathcal K)} \simeq_{d_0} \sum_{u=\min \mathcal H+1}^{\min \mathcal K}\sum_{(\mathcal I,\mathcal J)\in \mathcal B_u}X_{(\mathcal I,\mathcal J)} \mbox{ \rm mod}(\mathcal P_{(k,n_{d_0})}),
\end{equation}
where $\mathcal B_u$ is a set of some pairs $(\mathcal I,\mathcal J) \in \mathcal B$ such that $\min\mathcal I = \min \mathcal H$, $ \min\mathcal J = u$ and $n_{d_0} = (k-2)(2^{d_0}-1)$.
\end{defns}

From the result in \cite[Proposition 5.2.1]{su} we see that the set 
$$\mathcal B_4 = \{(\mathcal I, \mathcal J) \in {\sf PSeq}_4^5: X_{(\mathcal I, \mathcal J)} \in B_4((2)|^5)\} \subset  {\sf PInc}_4^5$$ is compatible with $(2)|^{5}$.

For $1 \leqslant i < j \leqslant k$, denote 
$f_{(i,j)} = f_if_{j-1}: P_{k-2} \stackrel{{\scriptstyle{f_{j-1}}}}{\longrightarrow} P_{k-1} \stackrel{{\scriptstyle{f_i}}}{\longrightarrow} P_k.$
Here $f_i$ and $f_{j-1}$ are defined by \eqref{ct22}. More precisely,
$$f_{(i,j)}(x_t) = \begin{cases} x_t, &\mbox{ if } 1 \leqslant t < i,\\
x_{t+1}, &\mbox{ if } i \leqslant t < j-2,\\
x_{t+2}, &\mbox{ if } j-2 \leqslant t \leqslant k-2,
\end{cases}$$
for $1 \leqslant t \leqslant k-2$. The main result of this section is the following. 
\begin{thms}\label{dlck2} Let $d_0$ be a positive integer, $d_0 > 2$, and let $k \geqslant 4$, $n = \sum_{i=1}^{k-2}(2^{d_i}-1)$ with $d_i$ positive integers such that $d_1 > d_2> \ldots > d_{k-3} \geqslant d_{k-2} = d \geqslant d_0 $. Denote $m = \sum_{i=1}^{k-3}(2^{d_i-d}-1) = \beta_k^d(n)$ with $\beta_k(n) = \frac{n-k+2}2$. Suppose the set $\mathcal B \subset {\sf PInc}_k^{d_0}$ is compatible with $(k-2)|^{d_0}$. Then, 
$$ \overline{\mathcal B} := \bigcup_{(\mathcal I,\mathcal J)\in \mathcal B}\Big\{X_{(\mathcal I,\mathcal J)}(X_{i,j})^{2^d-2^{d_0}}(f_{(i,j)}(y))^{2^d}: y \in  B_{k-2}(m)  
\Big\}$$
is a set of generators for $\mbox{\rm Ker}(\widetilde {Sq}^0_*)_{(k,n)}$, where  $i = \min\mathcal I=i_1$, $j = \min\mathcal J=j_1$ and $B_{k-2}(m)$ is the set of all the admissible monomials of the degree $m$ in $P_{k-2}$. Consequently,  
$\dim \mbox{\rm Ker}(\widetilde {Sq}^0_*)_{(k,n)} \leqslant |\mathcal B|\dim (QP_{k-2})_m.$
\end{thms}

We need the following lemmas for the proof of the theorem.
\begin{lems}\label{bdcbs}Let $n,\, m,\, d_0$ and $\mathcal B$ be as in Theorem \ref{dlck2}. Let $y_0$ be a monomial in $(P_k)_{m_0-1}$ with $m_0 = \sum_{i=1}^{k-2}(2^{d_i-d_0}-1)= \beta_k^{d_0}(n)$, $y_u = y_0x_u$ for $1\leqslant u \leqslant k$, and $(\mathcal I,\mathcal J) \in \mathcal B$, $i = \min \mathcal I$, $j = \min \mathcal J$. Then we have
\begin{align}\label{cth1}	
X_{(\mathcal I,\mathcal J)}y_i^{2^{d_0}} &\equiv 
\sum _{1\leqslant u < k\atop u \ne i, j}\sum_{(\mathcal U,\mathcal V) \in \mathcal B_u} X_{(\mathcal U,\mathcal V)}y_u^{2^{d_0}},\\
X_{(\mathcal I,\mathcal J)}y_j^{2^{d_0}}&\equiv
\sum _{2\leqslant v \leqslant k\atop v \ne i, j} \sum_{(\mathcal U,\mathcal V) \in \mathcal C_v} X_{(\mathcal U,\mathcal V)}y_v^{2^{d_0}},\label{cth2}
\end{align}
where $\mathcal B_u$ is a set of some $(\mathcal U,\mathcal V) \in B$ such that $\min \mathcal U = u$ for $u < i$ and $\min \mathcal U = i$ for $u > i$; $\mathcal C_v$ is a set of some $(\mathcal U,\mathcal V) \in \mathcal B$ such that $\min \mathcal V = v$ for $v < j$ and $\min \mathcal V = j$ for $v > j$.
\end{lems}
\begin{proof}
By the Cartan formula, we have  $Sq^1(X_j) = \sum_{u\ne j}X_{u,j}x_u^2$ and $Sq^1(X_jy_0^{2}) = Sq^1(X_j)y_0^{2}$. Hence, we obtain
\[X_{i,j}y_i^2 = \sum _{1\leqslant u < j\atop u\ne i} X_{u,j}y_u^{2} + \sum _{j < u \leqslant k} X_{j,u}y_u^{2} + Sq^1(X_jy_0^{2}).\] 
Since $(\mathcal I,\mathcal J) \in \mathcal B \subset {\sf PInc}_k^{d_0}$, we have $X_{(\mathcal I,\mathcal J)} =  X_{(\mathcal I\setminus i,\mathcal J\setminus j)}(X_{i,j})^{2^{d_0-1}}$ with $i=i_1,\, j= j_1$, and 
\[X_{(\mathcal I\setminus i,\mathcal J\setminus j)}(X_{u,j})^{2^{d_0-1}} = \left\{\begin{array}{l@{\enskip}l}
X_{(u|(\mathcal I\setminus i),\mathcal J)}, &\mbox{if } u< j,\\
X_{(j|(\mathcal I\setminus i),u|(\mathcal J\setminus j))}, &\mbox{if } u> j.\end{array}\right.\]
By using the Cartan formula, Proposition \ref{mdcb1} and Theorem \ref{dlsig} we see that the polynomial $X_{(\mathcal I\setminus i,\mathcal J\setminus j)}(Sq^1(X_jy_0^{2}))^{2^{d_0-1}}$ is hit. So, we get
\[X_{(\mathcal I,\mathcal J)}y_i^{2^{d_0}} \equiv \sum _{1\leqslant u < j\atop u \ne i} X_{(u|(\mathcal I\setminus i),\mathcal J)}y_u^{2^{d_0}} + \sum _{j < u \leqslant k} X_{(j|(\mathcal I\setminus i),u|(\mathcal J\setminus j))}y_u^{2^{d_0}}. \]
Since ${\rm rl}(\mathcal I) < d_0-1$ and ${\rm rl}(\mathcal J) < d_0-1$, we have $\min (u|(\mathcal I\setminus i)) = u$ for $u < i$, $ \min (j|(\mathcal I\setminus i)) = \min (u|(\mathcal I\setminus i)) = i$ for $i<u < j$ and $\min (u|(\mathcal J\setminus j)) = j$ for $u > j$. Hence, the relation (\ref{cth1}) follows from the condition (\ref{ctbd}) of $\mathcal B$ in Definition \ref{gth} and Lemma \ref{bdtd}.
	
The relation (\ref{cth2}) is proved by a similar computation.
\end{proof}

\begin{lems}\label{bdbt}
Let $n,\, d_0,\, m_0$ be as in Lemma \ref{bdcbs} and let $\mathcal P_k^1(n)$ denote the subspace of $(P_k)_n$ spanned by all monomials of the form $X_{(\mathcal I,\mathcal J)}(f_i(y))^{2^{d_0}}$ with $(\mathcal I,\mathcal J)\in \mathcal B$, $i = \min \mathcal I$ and $y \in (P_{k-1})_{m_0}$. Then $ \mbox{\rm Ker}(\widetilde {Sq}^0_*)_{(k,n)} \subset [\mathcal P_k^1(n)].$
\end{lems}
\begin{proof}
Let $x$ be a monomial of degree $n$ such that $[x] \in \mbox{\rm Ker}(\widetilde {Sq}^0_*)_{(k,n)}$. By using Lemmas \ref{bdbt1} and \ref{bdbt2}, we can assume that $\omega_{i}(x) = k-2$, for $1 \leqslant i \leqslant d$. Then, $x = \prod_{t=1}^dX_{\alpha_t,\delta_t}^{2^{d-t}}\bar y^{2^d}$, where $\alpha_t,\, \delta_t$ are integers such that $1 \leqslant \alpha_t < \delta_t \leqslant k$, for $1 \leqslant t \leqslant d$, and $\bar y$ is a monomial of degree $m= \beta_k^d(n)$ in $P_{k}$. Since $d \geqslant d_0$, we set $\mathcal H = (\alpha_1,\alpha_2, \ldots, \alpha_{d_0})$, $\mathcal K = (\delta_1, \delta_2, \ldots , \delta_{d_0})$, then we have $x = X_{(\mathcal H, \mathcal K)}\tilde y^{2^{d_0}}$, where $\tilde y = \prod_{t=d_0+1}^dX_{\alpha_t,\delta_t}^{2^{d-t}}\bar y^{2^d}$ is the monomial of degree $m_0$ in $P_{k}$. By the condition of the set $\mathcal B$ in Definition \ref{gth}, the monomial $X_{(\mathcal H, \mathcal K)}$ is of the form (\ref{ctbd}). Hence, by using Lemma \ref{bdtd}, one gets
\[
x = X_{(\mathcal H,\mathcal K)}\tilde y^{2^{d_0}} \equiv \sum_{u=\min\mathcal H+1}^{\min\mathcal K}\sum_{(\mathcal I,\mathcal J)\in \mathcal B_u}X_{(\mathcal I,\mathcal J)}\tilde y^{2^{d_0}}
\]
where $\mathcal B_u$ is as in Definition \ref{dnstin}.
For $i = \min \mathcal H$, we have $\tilde y = x_i^af_{i}(y)$  with $a$ a non-negative integer and $y \in (P_{k-1})_{m_0-a}$. We prove the lemma by proving $[X_{(\mathcal I,\mathcal J)}(x_i^af_{i}(y))^{2^{d_0}}] \in [\mathcal P_k^1(n)]$ for all $(\mathcal I,\mathcal J) \in \mathcal B_u$, $i < u \leqslant \min\mathcal K$. We prove this claim by double induction on $(a,i)$.
	
If $a = 0$, then the claim is true for all $1 \leqslant i < k$.   Suppose $a > 0$ and the claim is true for $(a-1,i)$ with $1 \leqslant i < \min \mathcal J$. 
	
For $i = 1$, by using Lemma \ref{bdcbs} with $y_0 = x_1^{a-1}f_1(y)$, we get 
\begin{equation}\label{ct2bd}
X_{(\mathcal I,\mathcal J)}(x_1^af_{1}(y))^{2^{d_0}} \equiv \sum_{2 \leqslant t \leqslant k\atop t \ne \min\mathcal J}\sum_{(\mathcal U,\mathcal V)\in \mathcal B_{(t,1)}} X_{(\mathcal U,\mathcal V)}(x_1^{a-1}f_1(x_{t-1}y))^{2^{d_0}},
\end{equation}
where $\mathcal B_{(t,1)}$ is a set of some $(\mathcal U,\mathcal V)\in \mathcal B$ such that $\min\mathcal U=1$. 
By the inductive hypothesis, $[X_{(\mathcal U,\mathcal V)}(x_1^{a-1}f_1(x_{t-1}y))^{2^{d_0}}] \in [\mathcal P_k^1(n)]$ for all $(\mathcal U,\mathcal V) \in \mathcal B_{(t,1)}$ with $1 < t \ne \min\mathcal J$. Hence, the claim is true for $(a,1)$.
	
Suppose $i > 1$ and the claim is true for all $(a',t)$, $1\leqslant t < i$, and for $(a-1,i)$. By applying Lemma \ref{bdcbs} for $y_0 = x_i^{a-1}f_i(y)$, we have
\begin{align}
X_{(\mathcal I,\mathcal J)}(x_i^af_{i}(y))^{2^{d_0}}  &\equiv \sum _{1\leqslant t <  i}\sum_{(\mathcal U,\mathcal V) \in \mathcal B_{(t,i)}} X_{(\mathcal U,\mathcal V)}(x_tx_i^{a-1}f_i(y))^{2^{d_0}}\nonumber\\ & + \sum _{i< t \leqslant k\atop t \ne \min\mathcal J}\sum_{(\mathcal U,\mathcal V) \in \mathcal B_{(t,i)}} X_{(\mathcal U,\mathcal V)}(x_i^{a-1}f_i(x_{t-1}y))^{2^{d_0}},\label{ct3bd}
\end{align}
where $\mathcal B_{(t,i)}$ is a set of some $(\mathcal U,\mathcal V) \in \mathcal B$ such that $\min\mathcal U = t$ for $t < i$ and $\min\mathcal U = i$ for $t> i$.
From the relation (\ref{ct3bd}) and the inductive hypothesis, we see that our claim is true for $(a,i)$. This completes the proof.
\end{proof}

We now prove Theorem \ref{dlck2}.
\begin{proof}[Proof of Theorem \ref{dlck2}.] Denote by $\langle [\overline{\mathcal B}]\rangle$ the subspace of $(QP_k)_n$ spanned by the set $[\overline{\mathcal B}]$. We prove that $ \mbox{\rm Ker}(\widetilde {Sq}^0_*)_{(k,n)} \subset \langle [\overline{\mathcal B}]\rangle $. By using Lemma \ref{bdbt}, we need only to prove that $[X_{(\mathcal I,\mathcal J)}(f_{i}(y^*))^{2^d}] \in \langle [\overline{\mathcal B}]\rangle$ for all $(\mathcal I,\mathcal J) \in \mathcal B$ with $\min\mathcal I = i$ and $y^* \in (P_{k-1})_{m_0}$, where $m_0 = \sum_{t=1}^{k-2}(2^{d_t-d_0}-1)=\beta_k^{d_0}(n)$.
	
Set $j = \min \mathcal J$, we have $f_i(y^*) = x_j^bf_{(i,j)}(y)$ with $b$ a non-negative integer and $y \in (P_{k-2})_{m_0-b}$. We prove  $[X_{(\mathcal I,\mathcal J)}(x_j^b f_{(i,j)}( y))^{2^d}] \in \langle [\overline{\mathcal B}]\rangle$ by double induction on $(b,j)$. 
	
If $b = 0$, then $y \in (P_{k-2})_{m_0}$. Since $\omega_u(y) = k-2$ for $1 \leqslant u \leqslant d-d_0$, we get $y = Y^{2^{d-d_0}-1}(\tilde y)^{2^{d-d_0}}$, with $\tilde y \in (P_{k-2})_m$ and $Y = x_1x_2\ldots x_{k-2}$. Note that $f_{(i,j)}(Y) = X_{i,j}$, hence $f_{(i,j)}(y) = X_{i,j}^{2^{d-d_0}-1}(f_{(i,j)}(\tilde y))^{2^{d-d_0}}$. Since $B_{k-2}(m)$ is a set of $\mathcal A$-generators for $(P_{k-2})_m$, there are $z_1,z_2,\ldots, z_r \in B_{k-2}(m)$ such that 
\[ \tilde y \equiv z_1 + z_2  + \ldots + z_r + \sum_{t>0}Sq^t(h_t),\]
where $h_t$ are suitable polynomials in $P_{k-2}$. Set $\textsf{p} = X_{(\mathcal I,\mathcal J)}(X_{i,j})^{2^d-2^{d_0}}$. By using Proposition \ref{mdcb1} and the Cartan formula, we have
\begin{align}
&\textsf{p}(Sq^t(f_{(i,j)}(h_t))^{2^d}= \textsf{p}Sq^{t2^d}\left((f_{(i,j)}(h_t))^{2^d}\right)\notag\\
&\quad = Sq^{t2^d}\left(\textsf{p}(f_{(i,j)}(h_t))^{2^d}\right) + \sum_{1\leqslant \ell \leqslant t}Sq^{\ell 2^d}(\textsf{p})\left(Sq^{t-\ell}(f_{(i,j)}(h_t)\right)^{2^d}\label{ctbs2}.
\end{align} 
Suppose $w$ is a monomial which appears as a term in the polynomial $Sq^{\ell 2^d}(\textsf{p})$. By Proposition \ref{bdkbs} we have $\omega(w) < \omega(\textsf{p}) = (k-2)|^d$. Hence, using Theorem \ref{dlsig} and \eqref{ctbs2}, we see that the polynomial $\textsf{p}(Sq^t(f_{(i,j)}(h_t))^{2^d}$ is hit. Since $f_{(i,j)}:P_{k-2} \to P_k$ is a homomorphism of $\mathcal A$-algebras, we get
\[ \left[X_{(\mathcal I,\mathcal J)}(f_{(i,j)}(y))^{2^d}\right] = \sum_{1 \leqslant u \leqslant r} \left[X_{(\mathcal I,\mathcal J)}(X_{i,j})^{2^d-2^{d_0}}(f_{(i,j)}(z_u))^{2^d}\right] \in \langle [\overline{\mathcal B}]\rangle.\]
Hence, our claim is true for $(0,j),\ i<j\leqslant k$. We assume $b > 0$ and our claim holds for  $(b-1,j)$ with $i < j \leqslant k$. 
	
For $j = 2$, we have $ i = 1$. By applying Lemma \ref{bdcbs} for $y_0 = x_2^{b-1}f_{(1,2)}(y)$ we obtain
\[X_{(\mathcal I,\mathcal J)}(x_2^b f_{(1,2)}(y))^{2^{d_0}} \equiv \sum_{3\leqslant t \leqslant k}\sum_{(\mathcal U,\mathcal V)\in \mathcal B_t}X_{(\mathcal U,\mathcal V)}(x_2^{b-1}f_{(1,2)}(x_{t-2}y))^{2^{d_0}}, \]
where $\mathcal B_t$ is a set of some $(\mathcal U,\mathcal V) \in \mathcal B$ such that $\min \mathcal V = 2$. The last equality and the inductive hypothesis imply our claim for $(b,2)$. 
	
Suppose $j > 2$ and the claim holds for all $(b',t)$ with $1\leqslant i < t <j$ and for $(b-1,j)$. By using Lemma \ref{bdcbs} with $y_0 = x_j^{b-1}f_{(i,j)}(y)$, we have 
\begin{align*}X_{(\mathcal I,\mathcal J)}(x_j^b f_{(i,j)}( y))^{2^{d_0}} &\equiv \sum_{1 \leqslant t < i}\sum_{(\mathcal U,\mathcal V)\in \mathcal B_t^*}X_{(\mathcal U,\mathcal V)}(f_{i}(x_{t}x_{j-1}^{b-1}y))^{2^{d_0}}\\ 
& \qquad + \sum_{i < t < j}\sum_{(\mathcal U,\mathcal V)\in \mathcal B_t^*}X_{(\mathcal U,\mathcal V)}(f_{i}(x_{t-1}x_{j-1}^{b-1}y))^{2^{d_0}}\\
& \qquad + \sum_{j < t \leqslant k}\sum_{(\mathcal U,\mathcal V)\in \mathcal B_t^*}X_{(\mathcal U,\mathcal V)}(x_j^{b-1}f_{(i,j)}(x_{t-2}y))^{2^{d_0}},
\end{align*}
where $\mathcal B_t^*$ is a set of some $(\mathcal U,\mathcal V) \in \mathcal B$ such that $\min \mathcal V = i$ for $t<i$, $\min \mathcal V = t$ for $i<t<j$ and $\min \mathcal V = j$ for $t > j$. From the last equality and the inductive hypothesis, our claim is true for $(b,j)$. The theorem is proved.
\end{proof}

\section{An application to the case $k = 5$}\label{s4}
\setcounter{equation}{0}

In this section, we prove one of our main results, Theorem \ref{dl20}, that gives an affirmative answer to Conjecture \ref{ker} for $k = 5$. 
To do this, we explicitly determine the set $B_5((3)|^d)$ of all the admissible monomials of weight vector $(3)|^d$ for $d \geqslant 5$. By combining this result and Theorem \ref{dlck2} one gets an upper bound for the dimension of the kernel of Kameko's squaring operation in the degree $n = 2^{d+t+u} + 2^{d+t} + 2^d -3$ with $d > 5$ and $t,\, u \geqslant 4$.  By using Theorem \ref{dlwa} below, we show that this upper bound is also a lower bound.

\begin{thm}[See Walker and Wood {\cite[Proposition 24.5.1]{wa4}}]\label{dlwa} Let $k \geqslant 3$ and $n= \sum_{i=1}^{k-2}(2^{d_i}-1)$ with $d_i$ positive integers. If $d_{i} - d_{i+1} \geqslant 4$ for  $1 \leqslant i \leqslant k - 3$ and $d_{k-2} \geqslant 5$, then 
\begin{equation}\label{ct12}\dim(QP_k)_n \geqslant (k-1)\prod_{3\leqslant i \leqslant k}(2^i-1).\end{equation}
\end{thm}

From the results of Kameko \cite[Theorem 8.1]{ka} and our work \cite[Proposition 2.5.1]{su2} we see that if $d \geqslant 4$, then $QP_3^+((3)|^d) = \langle [(x_1x_2x_3)^{2^d-1}]_{(3)|^d}\rangle$ and 
$QP_4^+((3)|^d) = \langle \{[w_{d,u}]_{(3)|^d} : 1\leqslant u \leqslant 11\}\rangle,$ 
where $w_{d,u}$ are determined as in Section \ref{s5}.

\smallskip
By applying Proposition \ref{mdbs}, we get 
$\dim QP_5^0((3)|^d) = {5\choose 3} + 11{5\choose 4} = 65.$
So, we need only to determine $QP_5^+((3)|^d)$. 

\begin{thm}\label{dl51} Let $d$ be an integer. If $d \geqslant 5$, then $QP_5^+((3)|^d)$ is an $\mathbb F_2$-vector space of dimension $90$ with a basis consisting the classes represented by the admissible monomials $a_{d,t},\, 1 \leqslant t \leqslant 90$, which are determined as in Section $\ref{s5}$.
Consequently, $\dim QP_5((3)|^d) = 155$ for any $d \geqslant 5$.
\end{thm}
In \cite[Proposition 1]{su3}, we have proved that for any weight vector $\omega$, $QP_k(\omega)$ is an $GL_k$-module. Hence, Theorem \ref{dl51} gives a representation of dimension 155 of the general group $GL_5$.

The theorem is proved by induction on $d$. The proof is based on Proposition \ref{mdcb51} and suitable strongly inadmissible monomials of weight vector $(3)|^d$ with $2 \leqslant d \leqslant 5$. Moreover, to prove the theorem for $d = 5$, we need to use suitable sets of generators for $QP_5((3)|^d)$ with $1 \leqslant d \leqslant 4$.

\subsection{Generating sets for $QP_5((3)|^d)$ with $d \leqslant 4$}\label{s511}\

\begin{props}\label{md52} We have
	
\smallskip
{\rm i)}  $B_5((3)|^1) = \{X_{\alpha, \beta}: 1 \leqslant \alpha < \beta \leqslant 5\}$. Hence, $\dim QP_5((3)|^1) = 10$.
	
{\rm ii)} $B_5^+((3)|^2)$ is the set of the monomials $a_{2,t}, \, 1 \leqslant t \leqslant 15$, which are determined as follows:
	
\medskip
\centerline{\begin{tabular}{llll}
$1.\ \  x_1x_2x_3^{2}x_4^{2}x_5^{3} $ & $2.\ \  x_1x_2x_3^{2}x_4^{3}x_5^{2} $ & $3.\ \ x_1x_2x_3^{3}x_4^{2}x_5^{2} $ & $4.\ \ x_1x_2^{2}x_3x_4^{2}x_5^{3} $\cr  
$5.\ \ x_1x_2^{2}x_3x_4^{3}x_5^{2} $ & $6.\ \ x_1x_2^{2}x_3^{2}x_4x_5^{3} $ & $7.\ \ x_1x_2^{2}x_3^{2}x_4^{3}x_5 $ & $8.\ \ x_1x_2^{2}x_3^{3}x_4x_5^{2} $\cr  
$9.\ \ x_1x_2^{2}x_3^{3}x_4^{2}x_5 $ & $10.\  x_1x_2^{3}x_3x_4^{2}x_5^{2} $ & $11.\  x_1x_2^{3}x_3^{2}x_4x_5^{2} $ & $12.\  x_1x_2^{3}x_3^{2}x_4^{2}x_5 $\cr 
$13.\  x_1^{3}x_2x_3x_4^{2}x_5^{2} $ & $14.\  x_1^{3}x_2x_3^{2}x_4x_5^{2} $ & $15.\  x_1^{3}x_2x_3^{2}x_4^{2}x_5 $.&
\end{tabular}}
	
\smallskip
Consequently, $\dim QP_5((3)|^2) = 55$.
\end{props}
\begin{lems}\label{bdk1} Let $i_1, i_2, j_1, j_2\in \mathbb N_k$ such that $i_1 < j_1,\, i_2 < j_2$.
	
\smallskip	
{\rm i)} If either $i_1 > i_2$ or $i_1 = i_2$ and $j_1 > j_2$, then $X_{i_1,j_1}^2X_{i_2,j_2}$ is strictly inadmissible.
	
{\rm ii)} If $j_1 > j_2$ and $i,\, j \in \mathbb N_k,\, i < j$, then the monomial $X_{i_1,j_1}^4X_{i_2,j_2}^2X_{i,j}$ is strictly inadmissible.
	
{\rm iii)} If either $i_1 < i_2 \leqslant j_1$ or $i_1 = i_2, j_1 \ne j_2$, then $X_{i_1,j_1}^4X_{i_2,j_2}^3$ is strictly  inadmissible.
	
{\rm iv)} If either $i_1 < i_2$ or $i_1 = i_2$ and $j_1 \leqslant j_2$, then $X_{i_1,j_1}^8X_{i_2,j_2}^7$ is strictly inadmissible. 
\end{lems}
\begin{proof} For simplicity, we prove Part ii). The others can be proved by a similar computation. 

If $i_1 = i_2 = i$, then $x = X_{i_1,j_1}^2X_{i_2,j_2} = x_{j_1}x_{j_2}^2X_{i, j_1, j_2}^3$. We have
\[x = x_{j_1}^2x_{j_2}X_{i,j_1,j_2}^3 + \sum_{t\ne i, j_1,j_2} x_{j_1}x_{j_2}x_t^4X_{i,j_1,j_2,j}^3 + Sq^1(x_{j_1}x_{j_2}X_{i,j_1,j_2,j}^3).\]
This equality shows that $x$ is strictly inadmissible. By Theorem \ref{dlcb1}, $x^2X_{i,j}$ is also strictly inadmissible.
	
Suppose $i_1 < i_2$. Then $x = x_{i_1}x_{i_2}^2x_{j_2}^2x_{j_1}X_{i_1,i_2,j_1,j_2}^3$. We have 
\begin{align*} x_{i_1}x_{i_2}^2x_{j_2}^2x_{j_1} &= x_{i_1}x_{i_2}x_{j_2}^2x_{j_1}^2 + x_{i_1}x_{i_2}^2x_{j_2}x_{j_1}^2 \\
& \quad + Sq^1(x_{i_1}^2x_{i_2}x_{j_2}x_{j_1}) + Sq^2(x_{i_1}x_{i_2}x_{j_2}x_{j_1}).
\end{align*}
So, by using the Cartan formula, we get
\[x = x_{i_1}x_{i_2}x_{j_2}^2 x_{j_1}^2X_{i_1,i_2,j_1,j_2}^3 + x_{i_1}x_{i_2}^2x_{j_2}x_{j_1}^2 X_{i_1,i_2,j_1,j_2}^3 + A + B + C,\]
where 
\begin{align*}
A &= x_{i_1}^2x_{i_2}x_{j_1}x_{j_2}Sq^1(X_{i_1,i_2,j_1,j_2}^3) + Sq^1(x_{i_1}x_{i_2}x_{j_1}x_{j_2})Sq^1(X_{i_1,i_2,j_1,j_2}^3),\\
B &= x_{i_1}x_{i_2}x_{j_1}x_{j_2}Sq^2(X_{i_1,i_2,j_1,j_2}^3),\\
C &=  Sq^1(x_{i_1}^2x_{i_2}x_{j_1}x_{j_2} X_{i_1,i_2,j_1,j_2}^3) + Sq^2(x_{i_1}x_{i_2}x_{j_1}x_{j_2} X_{i_1,i_2,j_1,j_2}^3).
\end{align*}
A simple computation shows that $X_{i,j}A^2 \in P_k^-((k-2)|^3)$, $X_{i,j}B^2 \in P_k^-((k-2)|^3)+\mathcal A(0)^+P_5$ and $X_{i,j}C^2 \in P_k^-((k-2)|^3)+\mathcal A(2)^+P_5$. Hence, the monomial  $x^2X_{i,j}$ is strictly inadmissible. 
\end{proof}

\begin{proof}[Proof of Proposition \ref{md52}.]
For $d = 1$, if $x \in P_5((3)|^1)$, then  $\omega(x) = (3)|^1$ if and only if $x = X_{\alpha, \beta}$ with $1 \leqslant \alpha < \beta \leqslant 5$. Since $X_{\alpha, \beta}$ is admissible, we see that the first of Proposition \ref{md52} is true.
	
From the results in Kameko \cite[Theorem 8.1]{ka} and our work \cite[Proposition 5.2.1]{su2}, we have $|B_3^+((3)|^2)| = 1$ and $|B_4^+((3)|^2)| = 6$.
Hence, by applying Proposition \ref{mdbs}, we get $\dim QP_5^0((3)|^2) = {5\choose 3} + 6{5\choose 4} = 40.$ So, we need only to determine $QP_5^+((3)|^2)$. We can check that if $x \in P_5^+((3)|^2)$ and $x \ne a_{2,t}$ for all $t,\, 1 \leqslant t \leqslant 15$, then $x = X_{i_1,j_1}^2X_{i_2,j_2}$ with $i_1 > i_2$. By Lemma \ref{bdk1}(i), $x$ is inadmissible. 
	
We observe that for $1 \leqslant t \leqslant 15$, $a_{2,t} = x_if_i(b_{2,t})$ with $b_{2,t}$ an admissible monomial of degree 8 in $P_4$ and $1\leqslant i \leqslant 5$. By Proposition \ref{mdmo}, $a_{2,t}$ is admissible. The proposition is proved.
\end{proof} 

Consider the case $d=3$. From the results in Kameko \cite[Theorem 8.1]{ka} and our work \cite[Proposition 5.4.2]{su2}, we have $|B_3^+((3)|^3)| = 1$ and $|B_4^+((3)|^3)| = 10$.
So, by using Proposition \ref{mdbs}, we get $\dim QP_5^0((3)|^3) = {5\choose 3} + 10{5\choose 4} =60$. We need to compute $QP_5^+((3)|^3)$.

We denote by $A(3)$ the set of the monomials $a_{3,t},\, 1 \leqslant t \leqslant 50$, which are given in Section \ref{s5} for $d=3$ and five monomials:

\medskip
\centerline{\begin{tabular}{lll}
$a_{3,51} = x_1^{3}x_2^{3}x_3^{4}x_4^{4}x_5^{7}$& 
$a_{3,52} =  x_1^{3}x_2^{3}x_3^{4}x_4^{7}x_5^{4}$& 
$a_{3,53} =  x_1^{3}x_2^{3}x_3^{7}x_4^{4}x_5^{4}$\cr 
$a_{3,54} =  x_1^{3}x_2^{7}x_3^{3}x_4^{4}x_5^{4} $&  
$a_{3,55} =  x_1^{7}x_2^{3}x_3^{3}x_4^{4}x_5^{4}$.&\cr
\end{tabular}}

\begin{props}\label{mdd532} $B_5^+((3)|^3) \subset A(3)\cup C(3)$, where $C(3)$ is the set of the monomials $a_{3,t},\, 56 \leqslant t \leqslant 70$, which are determined as follows:
	
\medskip
\centerline{\begin{tabular}{llll} 
$56. \ x_1x_2^{3}x_3^{5}x_4^{6}x_5^{6}$& 
$57. \ x_1x_2^{3}x_3^{6}x_4^{5}x_5^{6}$& 
$58. \ x_1x_2^{6}x_3^{3}x_4^{5}x_5^{6}$& 
$59. \ x_1^{3}x_2x_3^{5}x_4^{6}x_5^{6} $\cr  
$60. \ x_1^{3}x_2x_3^{6}x_4^{5}x_5^{6}$& 
$61. \ x_1^{3}x_2^{5}x_3x_4^{6}x_5^{6}$& 
$62. \ x_1^{3}x_2^{5}x_3^{6}x_4x_5^{6}$& 
$63. \ x_1^{3}x_2^{5}x_3^{2}x_4^{5}x_5^{6}$\cr  
$64. \ x_1^{3}x_2^{3}x_3^{4}x_4^{5}x_5^{6}$& 
$65. \ x_1^{3}x_2^{3}x_3^{5}x_4^{4}x_5^{6}$& 
$66. \ x_1^{3}x_2^{3}x_3^{5}x_4^{6}x_5^{4}$& 
$67. \ x_1^{3}x_2^{4}x_3^{3}x_4^{5}x_5^{6} $\cr  
$68. \ x_1^{3}x_2^{5}x_3^{3}x_4^{4}x_5^{6}$& 
$69. \ x_1^{3}x_2^{5}x_3^{3}x_4^{6}x_5^{4}$& 
$70. \ x_1^{3}x_2^{5}x_3^{6}x_4^{3}x_5^{4}$.& \cr
\end{tabular}}		
\end{props}

The proof of this proposition uses the following lemmas.

\begin{lems}\label{bdd30}
Let $w$ be one of the monomials: $x_1x_2^{6}x_3^{6}x_4$,  $x_1x_2^{2}x_3^{6}x_4^{5}$, $x_1x_2^{6}x_3^{2}x_4^{5}$,  $x_1x_2^{6}x_3^{3}x_4^{4}$, $x_1^{3}x_2^{4}x_3x_4^{6}$, $x_1^{3}x_2^{4}x_3^{5}x_4^{2}$, $x_1^{3}x_2^{5}x_3^{4}x_4^{2}$, $x_1^{3}x_2^{4}x_3^{3}x_4^{4}$, $x_1^{3}x_2^{4}x_3^{4}x_4^{3}$. 
Then, the monomial $x_i^7f_i(w)$, $1\leqslant i \leqslant 5$, is strictly inadmissible.
\end{lems}
\begin{proof} By using the Cartan formula, we have	
\begin{align*} x_1x_2^{6}x_3^{6}x_4 &= x_1x_2^{5}x_3^{6}x_4^{2} + x_1x_2^{6}x_3^{5}x_4^{2} + Sq^1(x_1^{2}x_2^{5}x_3^{5}x_4)\\ &\quad +  Sq^2(x_1x_2^{5}x_3^{5}x_4)\ \mbox{ mod}(P_4^-((2)|^3)),\\
x_1x_2^{2}x_3^{6}x_4^{5} &=  x_1x_2x_3^{6}x_4^{6} + x_1x_2^{2}x_3^{5}x_4^{6} +  Sq^1(x_1^{2}x_2x_3^{5}x_4^{5})\\ &\quad +  Sq^2(x_1x_2x_3^{5}x_4^{5}) \ \mbox{ mod}(P_4^-((2)|^3)),\\
x_1x_2^{6}x_3^{2}x_4^{5} &= x_1x_2^{5}x_3^{2}x_4^{6} + x_1x_2^{6}x_3x_4^{6} + Sq^1(x_1^{2}x_2^{5}x_3x_4^{5})\\ &\quad + Sq^2(x_1x_2^{5}x_3x_4^{5}) \ \mbox{ mod}(P_4^-((2)|^3)),\\
x_1x_2^{6}x_3^{3}x_4^{4} &= x_1x_2^{3}x_3^{4}x_4^{6} + x_1x_2^{3}x_3^{6}x_4^{4} + x_1x_2^{4}x_3^{3}x_4^{6} + x_1x_2^{4}x_3^{6}x_4^{3} + x_1x_2^{6}x_3^{2}x_4^{5}\\ &\quad +  Sq^1(x_1^{2}x_2^{5}x_3^{3}x_4^{3} + x_1^{2}x_2^{3}x_3^{5}x_4^{3} + x_1^{2}x_2^{3}x_3^{3}x_4^{5}) +  Sq^2(x_1x_2^{5}x_3^{3}x_4^{3}\\ &\quad + x_1x_2^{3}x_3^{5}x_4^{3} + x_1x_2^{3}x_3^{3}x_4^{5} + x_1x_2^{6}x_3^{2}x_4^{3})\ \mbox{ mod}(P_4^-((2)|^3)),\\
x_1^{3}x_2^{4}x_3x_4^{6} &= x_1^{2}x_2^{3}x_3^{4}x_4^{5} + x_1^{2}x_2^{5}x_3^{4}x_4^{3} + x_1^{3}x_2^{2}x_3^{4}x_4^{5} + x_1^{3}x_2^{3}x_3^{4}x_4^{4}\\ &\quad + Sq^1(x_1^{3}x_2^{3}x_3^{4}x_4^{3} + x_1^{3}x_2^{4}x_3x_4^{5}) + Sq^2(x_1^{2}x_2^{3}x_3^{4}x_4^{3} + x_1^{5}x_2^{2}x_3^{2}x_4^{3})\\ &\quad +  Sq^4(x_1^{3}x_2^{2}x_3^{2}x_4^{3}) \ \mbox{ mod}(P_4^-((2)|^3)),\\
x_1^{3}x_2^{4}x_3^{5}x_4^{2} &= x_1^{3}x_2^{2}x_3^{5}x_4^{4} + x_1^{3}x_2^{4}x_3^{3}x_4^{4}\\ &\quad +  Sq^2(x_1^{5}x_2^{2}x_3^{3}x_4^{2}) +  Sq^4(x_1^{3}x_2^{2}x_3^{3}x_4^{2}) \ \mbox{ mod}(P_4^-((2)|^3)),\\
x_1^{3}x_2^{5}x_3^{4}x_4^{2} &= x_1^{3}x_2^{3}x_3^{4}x_4^{4} + x_1^{3}x_2^{5}x_3^{2}x_4^{4}\\ &\quad + Sq^2(x_1^{5}x_2^{3}x_3^{2}x_4^{2}) +  Sq^4(x_1^{3}x_2^{3}x_3^{2}x_4^{2}) \ \mbox{ mod}(P_4^-((2)|^3)),\\
x_1^{3}x_2^{4}x_3^{3}x_4^{4} &= x_1^{2}x_2^{3}x_3^{5}x_4^{4} + x_1^{2}x_2^{5}x_3^{3}x_4^{4} + x_1^{3}x_2^{3}x_3^{4}x_4^{4}\\ &\quad +  Sq^1(x_1^{3}x_2^{3}x_3^{3}x_4^{4}) + Sq^2(x_1^{2}x_2^{3}x_3^{3}x_4^{4}) \ \mbox{ mod}(P_4^-((2)|^3)),\\
x_1^{3}x_2^{4}x_3^{4}x_4^{3} &= x_1^{2}x_2^{3}x_3^{4}x_4^{5} + x_1^{2}x_2^{5}x_3^{4}x_4^{3} + x_1^{3}x_2^{3}x_3^{4}x_4^{4}\\ &\quad +  Sq^1(x_1^{3}x_2^{3}x_3^{4}x_4^{3}) + Sq^2(x_1^{2}x_2^{3}x_3^{4}x_4^{3}) \ \mbox{ mod}(P_4^-((2)|^3)).
\end{align*}
From the above equalities we see that there is a positive integer $r$ such that
$$w = y_1 + y_2 + \ldots + y_r + Sq^1(g_1)+Sq^2(g_2) + Sq^4(g_3) \ \mbox{ mod}(P_4^-((2)|^3)),$$
where $y_t$ are monomials of weight vector $(2)|^3$ in $P_4$, $y_t < w$ with $1 \leqslant t \leqslant r$ and $g_1,\, g_2,\, g_3$ are suitable polynomials in $P_4$. Using the Cartan formula and Lemma \ref{bdlh} we get
\begin{align*} x_i^7f_i(w) &= x_i^7f_i(y_1)+ x_i^7f_i(y_2)+\ldots + x_i^7f_i(y_r)\ \mbox{ mod}(P_5^-((3)|^3) +\mathcal A(2)^+P_5).
\end{align*} 
Since $x_i^7f_i(y_t) < x_i^7f_i(w)$ for $1 \leqslant t \leqslant r$, the monomial $x_i^7f_i(w)$ is strictly inadmissible. The lemma is proved.
\end{proof}
\begin{lems}\label{bdd31}\

\medskip
{\rm i)} The following monomials are strictly inadmissible:
	
\medskip
\centerline{\begin{tabular}{lllll}
$x_1x_2^{6}x_3^{3}x_4^{6}x_5^{5}$& $x_1x_2^{6}x_3^{6}x_4^{3}x_5^{5}$& $x_1^{3}x_2^{5}x_3^{5}x_4^{2}x_5^{6}$& $x_1^{3}x_2^{5}x_3^{5}x_4^{6}x_5^{2}$& $x_1^{3}x_2^{5}x_3^{6}x_4^{5}x_5^{2} $\cr  $x_1^{3}x_2^{4}x_3^{5}x_4^{3}x_5^{6}$& $x_1^{3}x_2^{4}x_3^{5}x_4^{6}x_5^{3}$& $x_1^{3}x_2^{5}x_3^{4}x_4^{3}x_5^{6}$& $x_1^{3}x_2^{5}x_3^{4}x_4^{6}x_5^{3}$& $x_1^{3}x_2^{5}x_3^{6}x_4^{4}x_5^{3}$.\cr 
\end{tabular}} 

\medskip 
{\rm ii)} The following monomials are strongly inadmissible:

\medskip
\centerline{\begin{tabular}{lllll}
$x_1x_2^{3}x_3^{6}x_4^{6}x_5^{5}$& $x_1^{3}x_2x_3^{6}x_4^{6}x_5^{5}$& $x_1^{3}x_2^{5}x_3^{6}x_4^{6}x_5$& $x_1^{3}x_2^{5}x_3^{2}x_4^{6}x_5^{5}$& $x_1^{3}x_2^{5}x_3^{6}x_4^{2}x_5^{5}$.\cr 
\end{tabular}} 

\end{lems}
\begin{proof} 
Based on the Cartan formula we have
\begin{align*}
x_1x_2^{6}x_3^{3}x_4^{6}x_5^{5} &= x_1x_2^{3}x_3^{5}x_4^{6}x_5^{6} + x_1x_2^{3}x_3^{6}x_4^{5}x_5^{6} + x_1x_2^{3}x_3^{6}x_4^{6}x_5^{5} + x_1x_2^{5}x_3^{3}x_4^{6}x_5^{6}\\ &\quad + x_1x_2^{6}x_3^{3}x_4^{5}x_5^{6} +  Sq^1(x_1^{2}x_2^{3}x_3^{5}x_4^{5}x_5^{5} + x_1^{2}x_2^{5}x_3^{3}x_4^{5}x_5^{5})\\ &\quad + Sq^2(x_1x_2^{3}x_3^{5}x_4^{5}x_5^{5} + x_1x_2^{5}x_3^{3}x_4^{5}x_5^{5})\
\mbox{ mod}(P_5^-((3)|^3)),\\
x_1x_2^{6}x_3^{6}x_4^{3}x_5^{5} &= x_1x_2^{3}x_3^{5}x_4^{6}x_5^{6} + x_1x_2^{3}x_3^{6}x_4^{5}x_5^{6} + x_1x_2^{3}x_3^{6}x_4^{6}x_5^{5} + x_1x_2^{5}x_3^{6}x_4^{3}x_5^{6}\\ &\quad + x_1x_2^{6}x_3^{5}x_4^{3}x_5^{6} +  Sq^1(x_1^{2}x_2^{3}x_3^{5}x_4^{5}x_5^{5} + x_1^{2}x_2^{5}x_3^{5}x_4^{3}x_5^{5})\\ &\quad +  Sq^2(x_1x_2^{3}x_3^{5}x_4^{5}x_5^{5} + x_1x_2^{5}x_3^{5}x_4^{3}x_5^{5})\ \mbox{ mod}(P_5^-((3)|^3)),\\
x_1^{3}x_2^{5}x_3^{5}x_4^{2}x_5^{6} &= x_1^{3}x_2^{3}x_3^{5}x_4^{4}x_5^{6} + x_1^{3}x_2^{5}x_3^{3}x_4^{4}x_5^{6} +  Sq^1(x_1^{3}x_2^{3}x_3^{3}x_4x_5^{10})\\ &\quad +  Sq^2(x_1^{5}x_2^{3}x_3^{3}x_4^{2}x_5^{6}) + Sq^4(x_1^{3}x_2^{3}x_3^{3}x_4^{2}x_5^{6}) \ \mbox{ mod}(P_5^-((3)|^3)),\\
x_1^{3}x_2^{5}x_3^{5}x_4^{6}x_5^{2} &= x_1^{3}x_2^{3}x_3^{5}x_4^{6}x_5^{4} + x_1^{3}x_2^{5}x_3^{3}x_4^{6}x_5^{4} + Sq^1(x_1^{3}x_2^{3}x_3^{3}x_4^{9}x_5^{2})\\ &\quad +  Sq^2(x_1^{5}x_2^{3}x_3^{3}x_4^{6}x_5^{2}) + Sq^4(x_1^{3}x_2^{3}x_3^{3}x_4^{6}x_5^{2}) \ \mbox{ mod}(P_5^-((3)|^3)),\\
x_1^{3}x_2^{5}x_3^{6}x_4^{5}x_5^{2} &= x_1^{3}x_2^{3}x_3^{6}x_4^{5}x_5^{4} + x_1^{3}x_2^{5}x_3^{6}x_4^{3}x_5^{4} + Sq^1(x_1^{3}x_2^{3}x_3^{9}x_4^{3}x_5^{2})\\ &\quad +  Sq^2(x_1^{5}x_2^{3}x_3^{6}x_4^{3}x_5^{2}) + Sq^4(x_1^{3}x_2^{3}x_3^{6}x_4^{3}x_5^{2}) \ \mbox{ mod}(P_5^-((3)|^3)),\\
x_1^{3}x_2^{4}x_3^{5}x_4^{3}x_5^{6} &= x_1^{3}x_2^{2}x_3^{5}x_4^{5}x_5^{6} + x_1^{3}x_2^{4}x_3^{3}x_4^{5}x_5^{6} + Sq^1(x_1^{3}x_2x_3^{3}x_4^{3}x_5^{10}) \\ &\quad +  Sq^2(x_1^{5}x_2^{2}x_3^{3}x_4^{3}x_5^{6}) + Sq^4(x_1^{3}x_2^{2}x_3^{3}x_4^{3}x_5^{6}) \ \mbox{ mod}(P_5^-((3)|^3)),\\
x_1^{3}x_2^{4}x_3^{5}x_4^{6}x_5^{3} &= x_1^{3}x_2^{2}x_3^{5}x_4^{6}x_5^{5} + x_1^{3}x_2^{4}x_3^{3}x_4^{6}x_5^{5} + Sq^1(x_1^{3}x_2x_3^{3}x_4^{10}x_5^{3}) \\ &\quad +  Sq^2(x_1^{5}x_2^{2}x_3^{3}x_4^{6}x_5^{3}) +  Sq^4(x_1^{3}x_2^{2}x_3^{3}x_4^{6}x_5^{3}) \ \mbox{ mod}(P_5^-((3)|^3)),\\
x_1^{3}x_2^{5}x_3^{4}x_4^{3}x_5^{6} &= x_1^{3}x_2^{3}x_3^{4}x_4^{5}x_5^{6} + x_1^{3}x_2^{5}x_3^{2}x_4^{5}x_5^{6} +  Sq^1(x_1^{3}x_2^{3}x_3x_4^{3}x_5^{10})\\ &\quad +  Sq^2(x_1^{5}x_2^{3}x_3^{2}x_4^{3}x_5^{6}) +  Sq^4(x_1^{3}x_2^{3}x_3^{2}x_4^{3}x_5^{6}) \ \mbox{ mod}(P_5^-((3)|^3)),\\
x_1^{3}x_2^{5}x_3^{4}x_4^{6}x_5^{3} &= x_1^{3}x_2^{3}x_3^{4}x_4^{6}x_5^{5} + x_1^{3}x_2^{5}x_3^{2}x_4^{6}x_5^{5} +  Sq^1(x_1^{3}x_2^{3}x_3x_4^{10}x_5^{3})\\ &\quad +  Sq^2(x_1^{5}x_2^{3}x_3^{2}x_4^{6}x_5^{3}) +  Sq^4(x_1^{3}x_2^{3}x_3^{2}x_4^{6}x_5^{3}) \ \mbox{ mod}(P_5^-((3)|^3)),\\
x_1^{3}x_2^{5}x_3^{6}x_4^{4}x_5^{3} &= x_1^{3}x_2^{3}x_3^{6}x_4^{4}x_5^{5} + x_1^{3}x_2^{5}x_3^{6}x_4^{2}x_5^{5} + Sq^1(x_1^{3}x_2^{3}x_3^{9}x_4^{2}x_5^{3})\\ &\quad +  Sq^2(x_1^{5}x_2^{3}x_3^{6}x_4^{2}x_5^{3}) + Sq^4(x_1^{3}x_2^{3}x_3^{6}x_4^{2}x_5^{3})  \ \mbox{ mod}(P_5^-((3)|^3)).
\end{align*}
Part i) follows from the above equalities.
We prove Part ii). For $w = x_1x_2^{3}x_3^{6}x_4^{6}x_5^{5}$, we have
\begin{align*} 
w &= x_1x_2^{3}x_3^{5}x_4^{6}x_5^{6} + x_1x_2^{3}x_3^{6}x_4^{5}x_5^{6} + x_1x_2^{5}x_3^{5}x_4^{5}x_5^{5}\\ &\quad +  Sq^1(x_1^{2}x_2^{3}x_3^{5}x_4^{5}x_5^{5}) +  Sq^2(x_1x_2^{3}x_3^{5}x_4^{5}x_5^{5}) \ \mbox{ mod}(P_5^-((3)|^3)),
\end{align*}
where $x_1x_2^{3}x_3^{5}x_4^{6}x_5^{6},\, x_1x_2^{3}x_3^{6}x_4^{5}x_5^{6} < w$ and $x_1x_2^{5}x_3^{5}x_4^{5}x_5^{5} \in \mathcal P_{(5,21)}$. Hence, the monomial $w$ is strongly inadmissible. By a similar computation we have
\begin{align*} 
x_1^{3}x_2x_3^{6}x_4^{6}x_5^{5} &= x_1^{3}x_2x_3^{5}x_4^{6}x_5^{6} + x_1^{3}x_2x_3^{6}x_4^{5}x_5^{6} + x_1^{5}x_2x_3^{5}x_4^{5}x_5^{5}\\ &\quad +  Sq^1(x_1^{3}x_2^{2}x_3^{5}x_4^{5}x_5^{5}) +  Sq^2(x_1^{3}x_2x_3^{5}x_4^{5}x_5^{5}) \ \mbox{ mod}(P_5^-((3)|^3)),\\
x_1^{3}x_2^{5}x_3^{6}x_4^{6}x_5 &= x_1^{3}x_2^{5}x_3^{5}x_4^{6}x_5^{2} + x_1^{3}x_2^{5}x_3^{6}x_4^{5}x_5^{2} + x_1^{5}x_2^{5}x_3^{5}x_4^{5}x_5\\ &\quad + Sq^1(x_1^{3}x_2^{6}x_3^{5}x_4^{5}x_5) +  Sq^2(x_1^{3}x_2^{5}x_3^{5}x_4^{5}x_5)  \ \mbox{ mod}(P_5^-((3)|^3)),\\
x_1^{3}x_2^{5}x_3^{2}x_4^{6}x_5^{5} &= x_1^{3}x_2^{5}x_3x_4^{6}x_5^{6} + x_1^{3}x_2^{5}x_3^{2}x_4^{5}x_5^{6} + x_1^{5}x_2^{5}x_3x_4^{5}x_5^{5}\\ &\quad +  Sq^1(x_1^{3}x_2^{6}x_3x_4^{5}x_5^{5}) +  Sq^2(x_1^{3}x_2^{5}x_3x_4^{5}x_5^{5}) \ \mbox{ mod}(P_5^-((3)|^3)),\\
x_1^{3}x_2^{5}x_3^{6}x_4^{2}x_5^{5} &= x_1^{3}x_2^{5}x_3^{5}x_4^{2}x_5^{6} + x_1^{3}x_2^{5}x_3^{6}x_4x_5^{6} + x_1^{5}x_2^{5}x_3^{5}x_4x_5^{5}\\ &\quad + Sq^1(x_1^{3}x_2^{6}x_3^{5}x_4x_5^{5}) +  Sq^2(x_1^{3}x_2^{5}x_3^{5}x_4x_5^{5}) \ \mbox{ mod}(P_5^-((3)|^3)).
\end{align*}
Since $x_1^{5}x_2x_3^{5}x_4^{5}x_5^{5},\, x_1^{5}x_2^{5}x_3^{5}x_4^{5}x_5, \, x_1^{5}x_2^{5}x_3x_4^{5}x_5^{5},\,  x_1^{5}x_2^{5}x_3^{5}x_4x_5^{5} \in \mathcal P_{(5,21)}$, Part ii) follows from the above equalities. The lemma is completely proved.
\end{proof}

\begin{proof}[Proof of Proposition \ref{mdd532}.] We can see that if $x \in P_5^+((3)|^3)$ and $x \ne a_{3,t}$ for all $t,\, 1 \leqslant t \leqslant 70$, then either $x$ is one of the monomials as given in Lemmas \ref{bdd30}, \ref{bdd31}, or $x$ is of the form $X_{i,j}^4X_{i_1,j_1}^2X_{i_2,j_2}$ with $i_1>i_2$. Hence, by Lemma \ref{bdk1}(i) and Theorem \ref{dlcb1}, $x$ is inadmissible. The proposition is proved.
\end{proof}

Consider the case $d = 4$. From the results in Kameko \cite[Theorem 8.1]{ka} and our work \cite[Proposition 5.4.2]{su2}, we get $|B_3^+((3)|^4)| = 1$ and $|B_4^+((3)|^4)| = 11$. By Proposition \ref{mdbs}, $\dim QP_5((3)|^4) = {5\choose 3} + 11{5\choose 4} = 65$. We need to determine the set $B_5^+((3)|^4)$.

\medskip
Denote by $A(4)$ the set of the monomials $a_{4,t}$, $1 \leqslant t \leqslant 55$, which are determined as in Section \ref{s5} for $d=4$.

\begin{props}\label{mdd41} $B_5^+((3)|^4) \subset A(4)\cup C(4)$, where $C(4)$ is the set of the monomials $a_{4,t},\, 56 \leqslant t \leqslant 89,$ which are determined as in Section $\ref{s5}$ for $d=4$, and the following monomials:
	
\smallskip
\centerline{\begin{tabular}{lll}
$a_{4,90} = x_1^{3}x_2^{7}x_3^{8}x_4^{13}x_5^{14}$& 
$a_{4,91} = x_1^{7}x_2^{3}x_3^{8}x_4^{13}x_5^{14}$& 
$a_{4,92} =  x_1^{7}x_2^{7}x_3^{8}x_4^{9}x_5^{14} $\cr  
$a_{4,93} = x_1^{7}x_2^{7}x_3^{9}x_4^{8}x_5^{14}$& 
$a_{4,94} =  x_1^{7}x_2^{7}x_3^{9}x_4^{10}x_5^{12}$.&\cr
\end{tabular}}
\end{props}

We need the following lemmas for the proof of this proposition.

\begin{lems}\label{bdd41}
If $v$ is one of the monomials: $x_1x_2^{7}x_3^{10}x_4^{12}$, $x_1^{7}x_2x_3^{10}x_4^{12}$, $x_1^{3}x_2^{3}x_3^{12}x_4^{12}$, $x_1^{3}x_2^{5}x_3^{8}x_4^{14}$, $x_1^{3}x_2^{5}x_3^{14}x_4^{8}$, $x_1^{7}x_2^{7}x_3^{8}x_4^{8}$,
then the monomial $x_i^{15}f_i(x)$, $1\leqslant i \leqslant 5$, is strictly inadmissible.
\end{lems}
\begin{proof} Based on the Cartan formula we have
\begin{align*}
x_1x_2^{7}x_3^{10}x_4^{12} &= x_1x_2^{4}x_3^{11}x_4^{14} + x_1x_2^{6}x_3^{11}x_4^{12} + x_1x_2^{7}x_3^{8}x_4^{14}\\ &\quad +  Sq^1(x_1^{2}x_2^{7}x_3^{7}x_4^{13} + x_1^{2}x_2^{7}x_3^{9}x_4^{11} + x_1^{2}x_2^{9}x_3^{7}x_4^{11})\\ &\quad +  Sq^2(x_1x_2^{7}x_3^{7}x_4^{13} + x_1x_2^{7}x_3^{9}x_4^{11} + x_1x_2^{9}x_3^{7}x_4^{11})\\ &\quad +  Sq^4(x_1x_2^{4}x_3^{7}x_4^{14} + x_1x_2^{6}x_3^{7}x_4^{12})\ \mbox{ mod}(P_4^-((2)|^4)),\\
x_1^{7}x_2x_3^{10}x_4^{12} &= x_1^{4}x_2x_3^{11}x_4^{14} + x_1^{6}x_2x_3^{11}x_4^{12} + x_1^{7}x_2x_3^{8}x_4^{14}\\ &\quad +  Sq^1(x_1^{7}x_2^{2}x_3^{7}x_4^{13} + x_1^{7}x_2^{2}x_3^{9}x_4^{11} + x_1^{9}x_2^{2}x_3^{7}x_4^{11})\\ &\quad + Sq^2(x_1^{7}x_2x_3^{7}x_4^{13} + x_1^{7}x_2x_3^{9}x_4^{11} + x_1^{9}x_2x_3^{7}x_4^{11})\\ &\quad +  Sq^4(x_1^{4}x_2x_3^{7}x_4^{14} + x_1^{6}x_2x_3^{7}x_4^{12}) \ \mbox{ mod}(P_4^-((2)|^4)),\\
x_1^{3}x_2^{3}x_3^{12}x_4^{12} &= x_1^{2}x_2^{3}x_3^{12}x_4^{13} + x_1^{2}x_2^{5}x_3^{12}x_4^{11} + x_1^{2}x_2^{8}x_3^{13}x_4^{7}\\ &\quad +  Sq^1(x_1^{3}x_2^{3}x_3^{12}x_4^{11} + x_1^{3}x_2^{8}x_3^{11}x_4^{7}) +  Sq^2(x_1^{2}x_2^{3}x_3^{12}x_4^{11}\\ &\quad + x_1^{2}x_2^{8}x_3^{11}x_4^{7}) +  Sq^4(x_1^{3}x_2^{4}x_3^{12}x_4^{7}) \ \mbox{ mod}(P_4^-((2)|^4)),\\
x_1^{3}x_2^{5}x_3^{8}x_4^{14} &= x_1^{2}x_2^{5}x_3^{9}x_4^{14} + x_1^{3}x_2^{4}x_3^{9}x_4^{14} + Sq^1(x_1^{3}x_2^{3}x_3^{5}x_4^{18}\\ &\quad  + x_1^{3}x_2^{3}x_3^{9}x_4^{14}) +  Sq^2(x_1^{2}x_2^{3}x_3^{9}x_4^{14} + x_1^{5}x_2^{3}x_3^{6}x_4^{14})\\ &\quad  +  Sq^4(x_1^{3}x_2^{3}x_3^{6}x_4^{14}) \ \mbox{ mod}(P_4^-((2)|^4)),\\
x_1^{3}x_2^{5}x_3^{14}x_4^{8} &= x_1^{2}x_2^{5}x_3^{14}x_4^{9} + x_1^{3}x_2^{4}x_3^{14}x_4^{9} +  Sq^1(x_1^{3}x_2^{3}x_3^{14}x_4^{9}\\ &\quad + x_1^{3}x_2^{3}x_3^{17}x_4^{6}) +  Sq^2(x_1^{2}x_2^{3}x_3^{14}x_4^{9} + x_1^{5}x_2^{3}x_3^{14}x_4^{6})\\ &\quad + Sq^4(x_1^{3}x_2^{3}x_3^{14}x_4^{6})\ \mbox{ mod}(P_4^-((2)|^4)),\\
x_1^{7}x_2^{7}x_3^{8}x_4^{8} &= x_1^{4}x_2^{7}x_3^{8}x_4^{11} + x_1^{4}x_2^{11}x_3^{8}x_4^{7} + x_1^{5}x_2^{6}x_3^{8}x_4^{11} + x_1^{5}x_2^{10}x_3^{8}x_4^{7}\\ &\quad + x_1^{7}x_2^{6}x_3^{8}x_4^{9} +  Sq^1(x_1^{7}x_2^{7}x_3^{8}x_4^{7}) +  Sq^2(x_1^{7}x_2^{6}x_3^{8}x_4^{7})\\ &\quad +  Sq^4(x_1^{4}x_2^{7}x_3^{8}x_4^{7} + x_1^{5}x_2^{6}x_3^{8}x_4^{7})\ \mbox{ mod}(P_4^-((2)|^4)).
\end{align*}	
From the above equalities we see that there is a positive integer $s$ such that
$$v = u_1 + u_2 + \ldots + u_s + Sq^1(h_1)+Sq^2(h_2) + Sq^4(h_3) \ \mbox{ mod}(P_4^-((2)|^4)),$$
where $u_t$ are monomials of weight vector $(2)|^4$ in $P_4$, $u_t < v$ with $1 \leqslant t \leqslant s$ and $h_1, h_2, h_3$ are suitable polynomials in $P_4$. Using the Cartan formula and Lemma \ref{bdlh} we get
\begin{align*} x_i^{15}f_i(v) &= x_i^{15}f_i(u_1)+ x_i^{15}f_i(u_2)+\ldots + x_i^{15}f_i(u_s)\ \mbox{ mod}(P_5^-((3)|^4) +\mathcal A(3)^+P_5).
\end{align*} 
Since $x_i^{15}f_i(u_t) < x_i^{15}f_i(v)$ for $1 \leqslant t \leqslant s$, the monomial $x_i^{15}f_i(v)$ is strictly inadmissible.	The lemma is proved.
\end{proof}

\begin{lems}\label{bdd42}\
	
\medskip
{\rm i)} The following monomials are strictly inadmissible:
	
\smallskip
\centerline{\begin{tabular}{llll}
$x_1^{3}x_2^{5}x_3^{9}x_4^{14}x_5^{14}$& $x_1^{3}x_2^{5}x_3^{14}x_4^{9}x_5^{14}$& $x_1^{3}x_2^{7}x_3^{11}x_4^{12}x_5^{12}$& $x_1^{3}x_2^{7}x_3^{13}x_4^{8}x_5^{14} $\cr  $x_1^{3}x_2^{7}x_3^{13}x_4^{14}x_5^{8}$& $x_1^{3}x_2^{12}x_3^{3}x_4^{13}x_5^{14}$& $x_1^{7}x_2^{3}x_3^{11}x_4^{12}x_5^{12}$& $x_1^{7}x_2^{3}x_3^{13}x_4^{8}x_5^{14} $\cr  $x_1^{7}x_2^{3}x_3^{13}x_4^{14}x_5^{8}$& $x_1^{7}x_2^{7}x_3^{9}x_4^{14}x_5^{8}$& $x_1^{7}x_2^{9}x_3^{7}x_4^{10}x_5^{12}$& $x_1^{7}x_2^{11}x_3^{3}x_4^{12}x_5^{12} $\cr  $x_1^{7}x_2^{11}x_3^{5}x_4^{8}x_5^{14}$& $x_1^{7}x_2^{11}x_3^{5}x_4^{14}x_5^{8}$& $x_1^{7}x_2^{11}x_3^{13}x_4^{6}x_5^{8}$.&\cr  
\end{tabular}}
	
\medskip
{\rm ii)} The following monomials are strongly inadmissible:
	
\smallskip
\centerline{\begin{tabular}{llll}
$x_1^{3}x_2^{5}x_3^{14}x_4^{11}x_5^{12}$& $x_1^{3}x_2^{13}x_3^{6}x_4^{11}x_5^{12}$& $x_1^{3}x_2^{13}x_3^{7}x_4^{10}x_5^{12}$& $x_1^{3}x_2^{13}x_3^{14}x_4^{3}x_5^{12}$.\cr
\end{tabular}}
\end{lems}
\begin{proof} By a direct computation using the Cartan formula, we have
\begin{align*}
x_1^{3}x_2^{5}x_3^{9}x_4^{14}x_5^{14} &= x_1^{2}x_2^{3}x_3^{13}x_4^{13}x_5^{14} + x_1^{2}x_2^{5}x_3^{11}x_4^{13}x_5^{14} + x_1^{3}x_2^{3}x_3^{12}x_4^{13}x_5^{14}\\ &\quad + x_1^{3}x_2^{4}x_3^{11}x_4^{13}x_5^{14} +  Sq^1(x_1^{3}x_2^{3}x_3^{7}x_4^{13}x_5^{18} + x_1^{3}x_2^{3}x_3^{7}x_4^{17}x_5^{14}\\ &\quad + x_1^{3}x_2^{3}x_3^{11}x_4^{13}x_5^{14}) + Sq^2(x_1^{2}x_2^{3}x_3^{11}x_4^{13}x_5^{14} + x_1^{5}x_2^{3}x_3^{7}x_4^{14}x_5^{14})\\ &\quad +  Sq^4(x_1^{3}x_2^{3}x_3^{7}x_4^{14}x_5^{14}) \  \mbox{ mod}(P_5^-((3)|^4)),\\
x_1^{3}x_2^{5}x_3^{14}x_4^{9}x_5^{14} &= x_1^{2}x_2^{3}x_3^{13}x_4^{13}x_5^{14} + x_1^{2}x_2^{5}x_3^{13}x_4^{11}x_5^{14} + x_1^{3}x_2^{3}x_3^{13}x_4^{12}x_5^{14}\\ &\quad + x_1^{3}x_2^{4}x_3^{13}x_4^{11}x_5^{14} +  Sq^1(x_1^{3}x_2^{3}x_3^{13}x_4^{7}x_5^{18} + x_1^{3}x_2^{3}x_3^{13}x_4^{11}x_5^{14}\\ &\quad + x_1^{3}x_2^{3}x_3^{17}x_4^{7}x_5^{14}) +  Sq^2(x_1^{2}x_2^{3}x_3^{13}x_4^{11}x_5^{14} + x_1^{5}x_2^{3}x_3^{14}x_4^{7}x_5^{14}) \\ &\quad +  Sq^4(x_1^{3}x_2^{3}x_3^{14}x_4^{7}x_5^{14})  \  \mbox{ mod}(P_5^-((3)|^4)),\\
x_1^{3}x_2^{7}x_3^{11}x_4^{12}x_5^{12} &= x_1^{2}x_2^{7}x_3^{11}x_4^{12}x_5^{13} + x_1^{2}x_2^{7}x_3^{13}x_4^{12}x_5^{11} + x_1^{3}x_2^{7}x_3^{12}x_4^{10}x_5^{13}\\ &\quad + x_1^{3}x_2^{7}x_3^{10}x_4^{12}x_5^{13} +  Sq^1(x_1^{3}x_2^{7}x_3^{11}x_4^{12}x_5^{12} + x_1^{3}x_2^{11}x_3^{9}x_4^{10}x_5^{11}) \\ &\quad+  Sq^2(x_1^{2}x_2^{7}x_3^{11}x_4^{12}x_5^{11} + x_1^{5}x_2^{7}x_3^{10}x_4^{10}x_5^{11})\\ &\quad +  Sq^4(x_1^{3}x_2^{7}x_3^{10}x_4^{10}x_5^{11}) \  \mbox{ mod}(P_5^-((3)|^4)),\\
x_1^{3}x_2^{7}x_3^{13}x_4^{8}x_5^{14} &= x_1^{2}x_2^{7}x_3^{13}x_4^{9}x_5^{14} + x_1^{3}x_2^{5}x_3^{11}x_4^{12}x_5^{14} + x_1^{3}x_2^{5}x_3^{13}x_4^{10}x_5^{14}\\ &\quad + x_1^{3}x_2^{7}x_3^{9}x_4^{12}x_5^{14} + x_1^{3}x_2^{7}x_3^{12}x_4^{9}x_5^{14} +  Sq^1(x_1^{3}x_2^{7}x_3^{7}x_4^{5}x_5^{22}\\ &\quad + x_1^{3}x_2^{7}x_3^{7}x_4^{9}x_5^{18} + x_1^{3}x_2^{7}x_3^{11}x_4^{5}x_5^{18} + x_1^{3}x_2^{7}x_3^{11}x_4^{9}x_5^{14})\\ &\quad +  Sq^2(x_1^{2}x_2^{7}x_3^{7}x_4^{5}x_5^{22} + x_1^{2}x_2^{7}x_3^{11}x_4^{9}x_5^{14} + x_1^{5}x_2^{7}x_3^{7}x_4^{6}x_5^{18}\\ &\quad + x_1^{5}x_2^{7}x_3^{7}x_4^{10}x_5^{14} + x_1^{5}x_2^{7}x_3^{11}x_4^{6}x_5^{14})\\ &\quad +  Sq^4(x_1^{3}x_2^{5}x_3^{7}x_4^{12}x_5^{14} + x_1^{3}x_2^{5}x_3^{13}x_4^{6}x_5^{14} + x_1^{3}x_2^{11}x_3^{7}x_4^{6}x_5^{14})\\ &\quad +  Sq^8(x_1^{3}x_2^{7}x_3^{7}x_4^{6}x_5^{14}) \  \mbox{ mod}(P_5^-((3)|^4)),\\
x_1^{3}x_2^{7}x_3^{13}x_4^{14}x_5^{8} &= x_1^{2}x_2^{7}x_3^{13}x_4^{14}x_5^{9} + x_1^{3}x_2^{5}x_3^{11}x_4^{14}x_5^{12} + x_1^{3}x_2^{5}x_3^{13}x_4^{14}x_5^{10}\\ &\quad + x_1^{3}x_2^{7}x_3^{9}x_4^{14}x_5^{12} + x_1^{3}x_2^{7}x_3^{12}x_4^{14}x_5^{9} +  Sq^1(x_1^{3}x_2^{7}x_3^{7}x_4^{17}x_5^{10}\\ &\quad + x_1^{3}x_2^{7}x_3^{7}x_4^{21}x_5^{6} + x_1^{3}x_2^{7}x_3^{11}x_4^{14}x_5^{9} + x_1^{3}x_2^{7}x_3^{11}x_4^{17}x_5^{6})\\ &\quad + Sq^2(x_1^{2}x_2^{7}x_3^{7}x_4^{21}x_5^{6} + x_1^{2}x_2^{7}x_3^{11}x_4^{14}x_5^{9} + x_1^{5}x_2^{7}x_3^{7}x_4^{14}x_5^{10}\\ &\quad + x_1^{5}x_2^{7}x_3^{7}x_4^{18}x_5^{6} + x_1^{5}x_2^{7}x_3^{11}x_4^{14}x_5^{6})\\ &\quad +  Sq^4(x_1^{3}x_2^{5}x_3^{7}x_4^{14}x_5^{12} + x_1^{3}x_2^{5}x_3^{13}x_4^{14}x_5^{6} + x_1^{3}x_2^{11}x_3^{7}x_4^{14}x_5^{6})\\ &\quad +  Sq^8(x_1^{3}x_2^{7}x_3^{7}x_4^{14}x_5^{6}) \  \mbox{ mod}(P_5^-((3)|^4)),\\
x_1^{3}x_2^{12}x_3^{3}x_4^{13}x_5^{14} &= x_1^{2}x_2^{11}x_3^{5}x_4^{13}x_5^{14} + x_1^{2}x_2^{13}x_3^{3}x_4^{13}x_5^{14} + x_1^{3}x_2^{9}x_3^{5}x_4^{14}x_5^{14}\\ &\quad + x_1^{3}x_2^{11}x_3^{4}x_4^{13}x_5^{14} + Sq^1(x_1^{3}x_2^{7}x_3^{3}x_4^{13}x_5^{18} + x_1^{3}x_2^{7}x_3^{3}x_4^{17}x_5^{14}\\ &\quad + x_1^{3}x_2^{11}x_3^{3}x_4^{13}x_5^{14}) +  Sq^2(x_1^{2}x_2^{11}x_3^{3}x_4^{13}x_5^{14} + x_1^{5}x_2^{7}x_3^{3}x_4^{14}x_5^{14})\\ &\quad +  Sq^4(x_1^{3}x_2^{7}x_3^{3}x_4^{14}x_5^{14})\  \mbox{ mod}(P_5^-((3)|^4)),\\
x_1^{7}x_2^{3}x_3^{11}x_4^{12}x_5^{12} &= x_1^{7}x_2^{2}x_3^{11}x_4^{12}x_5^{13} + x_1^{7}x_2^{2}x_3^{13}x_4^{12}x_5^{11} + x_1^{7}x_2^{3}x_3^{12}x_4^{10}x_5^{13}\\ &\quad + x_1^{7}x_2^{3}x_3^{10}x_4^{12}x_5^{13} +  Sq^1(x_1^{7}x_2^{3}x_3^{11}x_4^{12}x_5^{12} + x_1^{11}x_2^{3}x_3^{9}x_4^{10}x_5^{11})\\ &\quad +  Sq^2(x_1^{7}x_2^{2}x_3^{11}x_4^{12}x_5^{11} + x_1^{7}x_2^{5}x_3^{10}x_4^{10}x_5^{11})\\ &\quad + Sq^4(x_1^{7}x_2^{3}x_3^{10}x_4^{10}x_5^{11})\  \mbox{ mod}(P_5^-((3)|^4)),\\
x_1^{7}x_2^{3}x_3^{13}x_4^{8}x_5^{14} &= x_1^{5}x_2^{3}x_3^{11}x_4^{12}x_5^{14} + x_1^{5}x_2^{3}x_3^{13}x_4^{10}x_5^{14} + x_1^{7}x_2^{2}x_3^{13}x_4^{9}x_5^{14}\\ &\quad + x_1^{7}x_2^{3}x_3^{9}x_4^{12}x_5^{14} + x_1^{7}x_2^{3}x_3^{12}x_4^{9}x_5^{14} + Sq^1(x_1^{7}x_2^{3}x_3^{7}x_4^{5}x_5^{22}\\ &\quad + x_1^{7}x_2^{3}x_3^{7}x_4^{9}x_5^{18} + x_1^{7}x_2^{3}x_3^{11}x_4^{5}x_5^{18} + x_1^{7}x_2^{3}x_3^{11}x_4^{9}x_5^{14})\\ &\quad + Sq^2(x_1^{7}x_2^{2}x_3^{7}x_4^{5}x_5^{22} + x_1^{7}x_2^{2}x_3^{11}x_4^{9}x_5^{14} + x_1^{7}x_2^{5}x_3^{7}x_4^{6}x_5^{18}\\ &\quad + x_1^{7}x_2^{5}x_3^{7}x_4^{10}x_5^{14} + x_1^{7}x_2^{5}x_3^{11}x_4^{6}x_5^{14}) + Sq^4(x_1^{5}x_2^{3}x_3^{7}x_4^{12}x_5^{14}\\ &\quad + x_1^{5}x_2^{3}x_3^{13}x_4^{6}x_5^{14} + x_1^{11}x_2^{3}x_3^{7}x_4^{6}x_5^{14})\\ &\quad +  Sq^8(x_1^{7}x_2^{3}x_3^{7}x_4^{6}x_5^{14})\  \mbox{ mod}(P_5^-((3)|^4)),\\
x_1^{7}x_2^{3}x_3^{13}x_4^{14}x_5^{8} &= x_1^{7}x_2^{2}x_3^{13}x_4^{14}x_5^{9} + x_1^{5}x_2^{3}x_3^{11}x_4^{14}x_5^{12} + x_1^{5}x_2^{3}x_3^{13}x_4^{14}x_5^{10}\\ &\quad + x_1^{7}x_2^{3}x_3^{9}x_4^{14}x_5^{12} + x_1^{7}x_2^{3}x_3^{12}x_4^{14}x_5^{9} +  Sq^1(x_1^{7}x_2^{3}x_3^{7}x_4^{17}x_5^{10}\\ &\quad + x_1^{7}x_2^{3}x_3^{7}x_4^{21}x_5^{6} + x_1^{7}x_2^{3}x_3^{11}x_4^{14}x_5^{9} + x_1^{7}x_2^{3}x_3^{11}x_4^{17}x_5^{6})\\ &\quad +  Sq^2(x_1^{7}x_2^{2}x_3^{7}x_4^{21}x_5^{6} + x_1^{7}x_2^{2}x_3^{11}x_4^{14}x_5^{9} + x_1^{7}x_2^{5}x_3^{7}x_4^{14}x_5^{10}\\ &\quad + x_1^{7}x_2^{5}x_3^{7}x_4^{18}x_5^{6} + x_1^{7}x_2^{5}x_3^{11}x_4^{14}x_5^{6}) +  Sq^4(x_1^{5}x_2^{3}x_3^{7}x_4^{14}x_5^{12}\\ &\quad + x_1^{5}x_2^{3}x_3^{13}x_4^{14}x_5^{6} + x_1^{11}x_2^{3}x_3^{7}x_4^{14}x_5^{6})\\ &\quad +  Sq^8(x_1^{7}x_2^{3}x_3^{7}x_4^{14}x_5^{6}) \  \mbox{ mod}(P_5^-((3)|^4)),\\
x_1^{7}x_2^{7}x_3^{9}x_4^{14}x_5^{8} &= x_1^{4}x_2^{7}x_3^{12}x_4^{11}x_5^{11} + x_1^{4}x_2^{11}x_3^{12}x_4^{7}x_5^{11} + x_1^{5}x_2^{3}x_3^{12}x_4^{11}x_5^{14}\\ &\quad + x_1^{5}x_2^{6}x_3^{12}x_4^{11}x_5^{11} + x_1^{5}x_2^{7}x_3^{9}x_4^{14}x_5^{10} + x_1^{5}x_2^{7}x_3^{10}x_4^{11}x_5^{12}\\ &\quad + x_1^{5}x_2^{7}x_3^{10}x_4^{14}x_5^{9} + x_1^{5}x_2^{10}x_3^{12}x_4^{7}x_5^{11} + x_1^{5}x_2^{11}x_3^{6}x_4^{14}x_5^{9}\\ &\quad + x_1^{5}x_2^{11}x_3^{9}x_4^{14}x_5^{6} + x_1^{5}x_2^{11}x_3^{10}x_4^{7}x_5^{12} + x_1^{7}x_2^{3}x_3^{12}x_4^{9}x_5^{14}\\ &\quad + x_1^{7}x_2^{5}x_3^{9}x_4^{14}x_5^{10} + x_1^{7}x_2^{5}x_3^{10}x_4^{11}x_5^{12} + x_1^{7}x_2^{6}x_3^{12}x_4^{9}x_5^{11}\\ &\quad + x_1^{7}x_2^{7}x_3^{8}x_4^{14}x_5^{9} + x_1^{7}x_2^{7}x_3^{9}x_4^{8}x_5^{14} + x_1^{7}x_2^{7}x_3^{9}x_4^{10}x_5^{12}\\ &\quad + Sq^1(x_1^{7}x_2^{5}x_3^{12}x_4^{7}x_5^{13} + x_1^{7}x_2^{7}x_3^{9}x_4^{8}x_5^{13} + x_1^{7}x_2^{7}x_3^{9}x_4^{9}x_5^{12}\\ &\quad  + x_1^{7}x_2^{7}x_3^{12}x_4^{7}x_5^{11} + x_1^{7}x_2^{10}x_3^{5}x_4^{13}x_5^{9}) +  Sq^2(x_1^{7}x_2^{3}x_3^{5}x_4^{22}x_5^{6}\\ &\quad + x_1^{7}x_2^{3}x_3^{6}x_4^{7}x_5^{20} + x_1^{7}x_2^{3}x_3^{12}x_4^{7}x_5^{14} + x_1^{7}x_2^{6}x_3^{12}x_4^{7}x_5^{11}\\ &\quad + x_1^{7}x_2^{7}x_3^{6}x_4^{14}x_5^{9} + x_1^{7}x_2^{7}x_3^{9}x_4^{14}x_5^{6} + x_1^{7}x_2^{7}x_3^{10}x_4^{7}x_5^{12}\\ &\quad + x_1^{7}x_2^{7}x_3^{10}x_4^{8}x_5^{11} + x_1^{7}x_2^{9}x_3^{5}x_4^{13}x_5^{9} + x_1^{7}x_2^{9}x_3^{6}x_4^{11}x_5^{10} \\ &\quad + x_1^{9}x_2^{9}x_3^{3}x_4^{13}x_5^{9}) +  Sq^4(x_1^{4}x_2^{7}x_3^{12}x_4^{7}x_5^{11} + x_1^{5}x_2^{3}x_3^{12}x_4^{7}x_5^{14}\\ &\quad + x_1^{5}x_2^{6}x_3^{12}x_4^{7}x_5^{11} + x_1^{5}x_2^{7}x_3^{6}x_4^{14}x_5^{9} + x_1^{5}x_2^{7}x_3^{9}x_4^{14}x_5^{6} \\ &\quad  + x_1^{5}x_2^{7}x_3^{10}x_4^{7}x_5^{12} + x_1^{11}x_2^{5}x_3^{5}x_4^{14}x_5^{6} + x_1^{11}x_2^{5}x_3^{6}x_4^{7}x_5^{12})\\ &\quad +  Sq^8(x_1^{7}x_2^{5}x_3^{5}x_4^{14}x_5^{6} + x_1^{7}x_2^{5}x_3^{6}x_4^{7}x_5^{12})  \  \mbox{ mod}(P_5^-((3)|^4)),\\
x_1^{7}x_2^{9}x_3^{7}x_4^{10}x_5^{12} &= x_1^{4}x_2^{7}x_3^{11}x_4^{11}x_5^{12} + x_1^{4}x_2^{11}x_3^{7}x_4^{11}x_5^{12} + x_1^{5}x_2^{7}x_3^{11}x_4^{10}x_5^{12}\\ &\quad + x_1^{5}x_2^{11}x_3^{7}x_4^{10}x_5^{12} + x_1^{7}x_2^{7}x_3^{8}x_4^{11}x_5^{12} + x_1^{7}x_2^{7}x_3^{9}x_4^{10}x_5^{12}\\ &\quad + x_1^{7}x_2^{8}x_3^{7}x_4^{11}x_5^{12} +  Sq^1(x_1^{7}x_2^{7}x_3^{7}x_4^{11}x_5^{12}) + Sq^2(x_1^{7}x_2^{7}x_3^{7}x_4^{10}x_5^{12}) \\ &\quad+  Sq^4(x_1^{4}x_2^{7}x_3^{7}x_4^{11}x_5^{12} + x_1^{5}x_2^{7}x_3^{7}x_4^{10}x_5^{12}) \  \mbox{ mod}(P_5^-((3)|^4)),\\
x_1^{7}x_2^{11}x_3^{3}x_4^{12}x_5^{12} &= x_1^{5}x_2^{7}x_3^{8}x_4^{11}x_5^{14} + x_1^{5}x_2^{11}x_3^{8}x_4^{7}x_5^{14} + x_1^{7}x_2^{9}x_3^{4}x_4^{11}x_5^{14}\\ &\quad + x_1^{7}x_2^{10}x_3x_4^{13}x_5^{14}  + x_1^{7}x_2^{10}x_3^{3}x_4^{12}x_5^{13} + x_1^{7}x_2^{10}x_3^{5}x_4^{12}x_5^{11}\\ &\quad + x_1^{7}x_2^{11}x_3x_4^{12}x_5^{14} + x_1^{7}x_2^{11}x_3^{2}x_4^{12}x_5^{13} +  Sq^1(x_1^{7}x_2^{7}x_3x_4^{7}x_5^{22}\\ &\quad + x_1^{7}x_2^{7}x_3x_4^{11}x_5^{18} + x_1^{7}x_2^{7}x_3x_4^{10}x_5^{19} + x_1^{7}x_2^{7}x_3x_4^{18}x_5^{11}\\ &\quad + x_1^{7}x_2^{11}x_3x_4^{7}x_5^{18} + x_1^{7}x_2^{11}x_3x_4^{11}x_5^{14} + x_1^{7}x_2^{11}x_3^{3}x_4^{12}x_5^{11}\\ &\quad + x_1^{7}x_2^{11}x_3^{4}x_4^{9}x_5^{13}) +  Sq^2(x_1^{7}x_2^{7}x_3^{2}x_4^{13}x_5^{14} + x_1^{7}x_2^{10}x_3x_4^{11}x_5^{14}\\ &\quad + x_1^{7}x_2^{10}x_3^{3}x_4^{12}x_5^{11} + x_1^{7}x_2^{13}x_3^{2}x_4^{7}x_5^{14} + x_1^{7}x_2^{13}x_3^{2}x_4^{10}x_5^{11})\\ &\quad +  Sq^4(x_1^{5}x_2^{7}x_3^{4}x_4^{11}x_5^{14} + x_1^{5}x_2^{11}x_3^{4}x_4^{7}x_5^{14}\\ &\quad + x_1^{11}x_2^{7}x_3^{2}x_4^{7}x_5^{14} + x_1^{11}x_2^{7}x_3^{2}x_4^{10}x_5^{11})\\ &\quad +  Sq^8(x_1^{7}x_2^{7}x_3^{2}x_4^{7}x_5^{14} + x_1^{7}x_2^{7}x_3^{2}x_4^{10}x_5^{11}) \  \mbox{ mod}(P_5^-((3)|^4)),\\
x_1^{7}x_2^{11}x_3^{5}x_4^{8}x_5^{14} &= x_1^{5}x_2^{7}x_3^{9}x_4^{10}x_5^{14} + x_1^{5}x_2^{11}x_3^{9}x_4^{6}x_5^{14} + x_1^{7}x_2^{9}x_3^{5}x_4^{10}x_5^{14}\\ &\quad + x_1^{7}x_2^{10}x_3^{5}x_4^{9}x_5^{14} + x_1^{7}x_2^{11}x_3^{4}x_4^{9}x_5^{14} + Sq^1(x_1^{7}x_2^{7}x_3^{3}x_4^{5}x_5^{22}\\ &\quad + x_1^{7}x_2^{7}x_3^{3}x_4^{9}x_5^{18} + x_1^{7}x_2^{11}x_3^{3}x_4^{5}x_5^{18} + x_1^{7}x_2^{11}x_3^{3}x_4^{9}x_5^{14})\\ &\quad + Sq^2(x_1^{7}x_2^{7}x_3^{2}x_4^{5}x_5^{22} + x_1^{7}x_2^{7}x_3^{3}x_4^{12}x_5^{14} + x_1^{7}x_2^{10}x_3^{3}x_4^{9}x_5^{14}\\ &\quad + x_1^{7}x_2^{13}x_3^{3}x_4^{6}x_5^{14}) +  Sq^4(x_1^{5}x_2^{7}x_3^{5}x_4^{10}x_5^{14} + x_1^{5}x_2^{11}x_3^{5}x_4^{6}x_5^{14}\\ &\quad + x_1^{11}x_2^{7}x_3^{3}x_4^{6}x_5^{14}) + Sq^8(x_1^{7}x_2^{7}x_3^{3}x_4^{6}x_5^{14})\  \mbox{ mod}(P_5^-((3)|^4)),\\
x_1^{7}x_2^{11}x_3^{5}x_4^{14}x_5^{8} &= x_1^{5}x_2^{7}x_3^{9}x_4^{14}x_5^{10} + x_1^{5}x_2^{11}x_3^{9}x_4^{14}x_5^{6} + x_1^{7}x_2^{9}x_3^{5}x_4^{14}x_5^{10}\\ &\quad + x_1^{7}x_2^{10}x_3^{5}x_4^{14}x_5^{9} + x_1^{7}x_2^{11}x_3^{4}x_4^{14}x_5^{9} +  Sq^1(x_1^{7}x_2^{7}x_3^{3}x_4^{18}x_5^{9}\\ &\quad + x_1^{7}x_2^{7}x_3^{3}x_4^{22}x_5^{5} + x_1^{7}x_2^{11}x_3^{3}x_4^{14}x_5^{9} + x_1^{7}x_2^{11}x_3^{3}x_4^{18}x_5^{5})\\ &\quad +  Sq^2(x_1^{7}x_2^{7}x_3^{2}x_4^{22}x_5^{5} + x_1^{7}x_2^{7}x_3^{3}x_4^{14}x_5^{12} + x_1^{7}x_2^{10}x_3^{3}x_4^{14}x_5^{9}\\ &\quad + x_1^{7}x_2^{13}x_3^{3}x_4^{14}x_5^{6}) +  Sq^4(x_1^{5}x_2^{7}x_3^{5}x_4^{14}x_5^{10} + x_1^{5}x_2^{11}x_3^{5}x_4^{14}x_5^{6}\\ &\quad + x_1^{11}x_2^{7}x_3^{3}x_4^{14}x_5^{6}) +  Sq^8(x_1^{7}x_2^{7}x_3^{3}x_4^{14}x_5^{6})\  \mbox{ mod}(P_5^-((3)|^4)),\\
x_1^{7}x_2^{11}x_3^{13}x_4^{6}x_5^{8} &= x_1^{7}x_2^{7}x_3^{12}x_4^{9}x_5^{10} + x_1^{7}x_2^{7}x_3^{13}x_4^{10}x_5^{8} + x_1^{7}x_2^{9}x_3^{13}x_4^{6}x_5^{10}\\ &\quad + x_1^{7}x_2^{9}x_3^{13}x_4^{10}x_5^{6} + x_1^{7}x_2^{10}x_3^{13}x_4^{6}x_5^{9} + x_1^{7}x_2^{10}x_3^{13}x_4^{9}x_5^{6}\\ &\quad + x_1^{7}x_2^{11}x_3^{12}x_4^{6}x_5^{9} + x_1^{7}x_2^{11}x_3^{12}x_4^{9}x_5^{6} + x_1^{7}x_2^{11}x_3^{13}x_4^{4}x_5^{10}\\ &\quad + Sq^1(x_1^{7}x_2^{7}x_3^{11}x_4^{9}x_5^{10} + x_1^{7}x_2^{7}x_3^{19}x_4^{5}x_5^{6} + x_1^{7}x_2^{11}x_3^{11}x_4^{6}x_5^{9}\\ &\quad + x_1^{7}x_2^{11}x_3^{11}x_4^{9}x_5^{6}) +  Sq^2(x_1^{7}x_2^{7}x_3^{11}x_4^{6}x_5^{12} + x_1^{7}x_2^{7}x_3^{11}x_4^{12}x_5^{6}\\ &\quad + x_1^{7}x_2^{7}x_3^{18}x_4^{5}x_5^{6} + x_1^{7}x_2^{7}x_3^{19}x_4^{4}x_5^{6} + x_1^{7}x_2^{10}x_3^{11}x_4^{6}x_5^{9}\\ &\quad + x_1^{7}x_2^{10}x_3^{11}x_4^{9}x_5^{6} + x_1^{7}x_2^{13}x_3^{11}x_4^{6}x_5^{6}) \\ &\quad+  Sq^4(x_1^{5}x_2^{7}x_3^{13}x_4^{6}x_5^{10} + x_1^{5}x_2^{7}x_3^{13}x_4^{10}x_5^{6} + x_1^{5}x_2^{11}x_3^{13}x_4^{6}x_5^{6}\\ &\quad + x_1^{11}x_2^{7}x_3^{11}x_4^{6}x_5^{6} + x_1^{11}x_2^{7}x_3^{13}x_4^{4}x_5^{6})\\ &\quad + Sq^8(x_1^{7}x_2^{7}x_3^{11}x_4^{6}x_5^{6} + x_1^{7}x_2^{7}x_3^{13}x_4^{4}x_5^{6}) \  \mbox{ mod}(P_5^-((3)|^4)).
\end{align*}
Hence, Part i) is proved.
We now prove Part ii). We have
\begin{align*}
x_1^{3}x_2^{5}x_3^{14}x_4^{11}x_5^{12} &= x_1^{2}x_2^{3}x_3^{13}x_4^{14}x_5^{13} + x_1^{2}x_2^{5}x_3^{13}x_4^{14}x_5^{11} + x_1^{3}x_2^{3}x_3^{13}x_4^{14}x_5^{12}\\ &\quad + x_1^{3}x_2^{3}x_3^{14}x_4^{13}x_5^{12} + x_1^{3}x_2^{4}x_3^{13}x_4^{14}x_5^{11} + x_1^{3}x_2^{5}x_3^{13}x_4^{14}x_5^{10}\\ &\quad + x_1^{5}x_2^{5}x_3^{13}x_4^{13}x_5^{9} +  Sq^1(x_1^{3}x_2^{3}x_3^{13}x_4^{14}x_5^{11} + x_1^{3}x_2^{3}x_3^{13}x_4^{18}x_5^{7}\\ &\quad + x_1^{3}x_2^{3}x_3^{17}x_4^{11}x_5^{10} + x_1^{3}x_2^{3}x_3^{17}x_4^{14}x_5^{7} + x_1^{3}x_2^{6}x_3^{13}x_4^{13}x_5^{9}) \\ &\quad+  Sq^2(x_1^{2}x_2^{3}x_3^{13}x_4^{14}x_5^{11} + x_1^{3}x_2^{5}x_3^{13}x_4^{13}x_5^{9} \\ &\quad+ x_1^{5}x_2^{3}x_3^{14}x_4^{11}x_5^{10} + x_1^{5}x_2^{3}x_3^{14}x_4^{14}x_5^{7})\\ &\quad +  Sq^4(x_1^{3}x_2^{3}x_3^{14}x_4^{11}x_5^{10} + x_1^{3}x_2^{3}x_3^{14}x_4^{14}x_5^{7}) \  \mbox{ mod}(P_5^-((3)|^4)),
\end{align*}
where $x_1^{5}x_2^{5}x_3^{13}x_4^{13}x_5^{9} \in  \mathcal P_{(5,45)}$. Hence, this equality shows that $x_1^{3}x_2^{5}x_3^{14}x_4^{11}x_5^{12}$ is strongly inadmissible. 
\begin{align*}
x_1^{3}x_2^{13}x_3^{6}x_4^{11}x_5^{12} &= x_1^{3}x_2^{7}x_3^{8}x_4^{14}x_5^{13} + x_1^{3}x_2^{7}x_3^{12}x_4^{14}x_5^{9} + x_1^{3}x_2^{9}x_3^{6}x_4^{14}x_5^{13}\\ &\quad + x_1^{3}x_2^{9}x_3^{12}x_4^{14}x_5^{7} + x_1^{3}x_2^{11}x_3^{6}x_4^{13}x_5^{12} + x_1^{3}x_2^{13}x_3^{4}x_4^{14}x_5^{11}\\ &\quad + x_1^{3}x_2^{13}x_3^{5}x_4^{14}x_5^{10} + x_1^{5}x_2^{13}x_3^{5}x_4^{13}x_5^{9} +  Sq^1(x_1^{3}x_2^{7}x_3^{5}x_4^{18}x_5^{11}\\ &\quad + x_1^{3}x_2^{7}x_3^{9}x_4^{18}x_5^{7} + x_1^{3}x_2^{11}x_3^{5}x_4^{18}x_5^{7} + x_1^{3}x_2^{11}x_3^{9}x_4^{11}x_5^{10}\\ &\quad + x_1^{3}x_2^{14}x_3^{5}x_4^{13}x_5^{9}) +  Sq^2(x_1^{3}x_2^{13}x_3^{5}x_4^{13}x_5^{9} + x_1^{5}x_2^{7}x_3^{6}x_4^{14}x_5^{11}\\ &\quad + x_1^{5}x_2^{7}x_3^{10}x_4^{14}x_5^{7} + x_1^{5}x_2^{11}x_3^{6}x_4^{11}x_5^{10} + x_1^{5}x_2^{11}x_3^{6}x_4^{14}x_5^{7}) \\ &\quad+  Sq^4(x_1^{3}x_2^{7}x_3^{6}x_4^{14}x_5^{11} + x_1^{3}x_2^{7}x_3^{10}x_4^{14}x_5^{7} + x_1^{3}x_2^{11}x_3^{6}x_4^{11}x_5^{10} \\ &\quad+ x_1^{3}x_2^{11}x_3^{6}x_4^{14}x_5^{7} + x_1^{3}x_2^{13}x_3^{4}x_4^{14}x_5^{7}) \  \mbox{ mod}(P_5^-((3)|^4)).
\end{align*}
Since $x_1^{5}x_2^{13}x_3^{5}x_4^{13}x_5^{9}\in \mathcal P_{(5,45)}$, the monomial $x_1^{3}x_2^{13}x_3^{6}x_4^{11}x_5^{12}$ is strongly inadmissible.

For $w = x_1^{3}x_2^{13}x_3^{7}x_4^{10}x_5^{12}$, we have
\begin{align*}
 w &= x_1^{2}x_2^{7}x_3^{13}x_4^{10}x_5^{13} + x_1^{2}x_2^{7}x_3^{13}x_4^{12}x_5^{11} + x_1^{2}x_2^{7}x_3^{14}x_4^{9}x_5^{13} + x_1^{2}x_2^{13}x_3^{7}x_4^{9}x_5^{14}\\ &\quad + x_1^{2}x_2^{13}x_3^{7}x_4^{10}x_5^{13} + x_1^{2}x_2^{13}x_3^{7}x_4^{12}x_5^{11} + x_1^{2}x_2^{13}x_3^{9}x_4^{7}x_5^{14} + x_1^{2}x_2^{13}x_3^{10}x_4^{7}x_5^{13}\\ &\quad + x_1^{2}x_2^{13}x_3^{12}x_4^{7}x_5^{11} + x_1^{3}x_2^{7}x_3^{8}x_4^{13}x_5^{14} + x_1^{3}x_2^{7}x_3^{10}x_4^{13}x_5^{12} + x_1^{3}x_2^{7}x_3^{13}x_4^{8}x_5^{14}\\ &\quad + x_1^{3}x_2^{7}x_3^{14}x_4^{9}x_5^{12} + x_1^{3}x_2^{8}x_3^{7}x_4^{13}x_5^{14} + x_1^{3}x_2^{8}x_3^{13}x_4^{7}x_5^{14} + x_1^{3}x_2^{10}x_3^{7}x_4^{13}x_5^{12}\\ &\quad + x_1^{3}x_2^{10}x_3^{13}x_4^{7}x_5^{12} + x_1^{3}x_2^{13}x_3^{4}x_4^{11}x_5^{14} + x_1^{3}x_2^{13}x_3^{6}x_4^{11}x_5^{12} + x_1^{3}x_2^{13}x_3^{7}x_4^{8}x_5^{14} +g\\ &\quad +  Sq^1(x_1^{3}x_2^{7}x_3^{7}x_4^{9}x_5^{18} + x_1^{3}x_2^{7}x_3^{7}x_4^{14}x_5^{13} + x_1^{3}x_2^{7}x_3^{9}x_4^{7}x_5^{18} + x_1^{3}x_2^{7}x_3^{9}x_4^{14}x_5^{11}\\ &\quad + x_1^{3}x_2^{7}x_3^{13}x_4^{10}x_5^{11} + x_1^{3}x_2^{7}x_3^{14}x_4^{7}x_5^{13} + x_1^{3}x_2^{9}x_3^{7}x_4^{7}x_5^{18} + x_1^{3}x_2^{9}x_3^{7}x_4^{14}x_5^{11}\\ &\quad + x_1^{3}x_2^{9}x_3^{14}x_4^{7}x_5^{11} + x_1^{3}x_2^{14}x_3^{7}x_4^{7}x_5^{13} + x_1^{3}x_2^{14}x_3^{7}x_4^{9}x_5^{11} + x_1^{3}x_2^{14}x_3^{9}x_4^{7}x_5^{11}\\ &\quad + x_1^{4}x_2^{7}x_3^{7}x_4^{13}x_5^{13} + x_1^{4}x_2^{7}x_3^{9}x_4^{13}x_5^{11} + x_1^{4}x_2^{7}x_3^{13}x_4^{7}x_5^{13} + x_1^{4}x_2^{7}x_3^{13}x_4^{9}x_5^{11}\\ &\quad + x_1^{4}x_2^{9}x_3^{7}x_4^{13}x_5^{11} + x_1^{4}x_2^{9}x_3^{13}x_4^{7}x_5^{11}) +  Sq^2(x_1^{2}x_2^{7}x_3^{7}x_4^{14}x_5^{13} + x_1^{2}x_2^{7}x_3^{9}x_4^{14}x_5^{11}\\ &\quad + x_1^{2}x_2^{7}x_3^{13}x_4^{10}x_5^{11} + x_1^{2}x_2^{7}x_3^{14}x_4^{7}x_5^{13} + x_1^{2}x_2^{9}x_3^{7}x_4^{14}x_5^{11} + x_1^{2}x_2^{9}x_3^{14}x_4^{7}x_5^{11}\\ &\quad + x_1^{2}x_2^{13}x_3^{7}x_4^{7}x_5^{14} + x_1^{2}x_2^{13}x_3^{7}x_4^{10}x_5^{11} + x_1^{2}x_2^{13}x_3^{10}x_4^{7}x_5^{11} + x_1^{5}x_2^{7}x_3^{7}x_4^{7}x_5^{17}\\ &\quad + x_1^{5}x_2^{7}x_3^{7}x_4^{11}x_5^{13} + x_1^{5}x_2^{7}x_3^{9}x_4^{11}x_5^{11} + x_1^{5}x_2^{7}x_3^{11}x_4^{7}x_5^{13} + x_1^{5}x_2^{7}x_3^{11}x_4^{9}x_5^{11}\\ &\quad + x_1^{5}x_2^{9}x_3^{7}x_4^{11}x_5^{11} + x_1^{5}x_2^{9}x_3^{11}x_4^{7}x_5^{11} + x_1^{5}x_2^{11}x_3^{7}x_4^{7}x_5^{13} + x_1^{5}x_2^{11}x_3^{7}x_4^{9}x_5^{11}\\ &\quad + x_1^{5}x_2^{11}x_3^{9}x_4^{7}x_5^{11}) + Sq^4(x_1^{3}x_2^{7}x_3^{7}x_4^{7}x_5^{17} + x_1^{3}x_2^{7}x_3^{7}x_4^{11}x_5^{13} + x_1^{3}x_2^{7}x_3^{9}x_4^{11}x_5^{11}\\ &\quad + x_1^{3}x_2^{7}x_3^{11}x_4^{7}x_5^{13} + x_1^{3}x_2^{7}x_3^{11}x_4^{9}x_5^{11} + x_1^{3}x_2^{9}x_3^{7}x_4^{11}x_5^{11} + x_1^{3}x_2^{9}x_3^{11}x_4^{7}x_5^{11}\\ &\quad + x_1^{3}x_2^{11}x_3^{7}x_4^{7}x_5^{13} + x_1^{3}x_2^{11}x_3^{7}x_4^{9}x_5^{11} + x_1^{3}x_2^{11}x_3^{9}x_4^{7}x_5^{11}\\ &\quad + x_1^{3}x_2^{13}x_3^{4}x_4^{7}x_5^{14} + x_1^{3}x_2^{13}x_3^{6}x_4^{7}x_5^{12})  \  \mbox{ mod}(P_5^-((3)|^4)),
\end{align*}
where
\begin{align*}
g&= x_1^{3}x_2^{7}x_3^{9}x_4^{9}x_5^{17} + x_1^{3}x_2^{9}x_3^{7}x_4^{9}x_5^{17} + x_1^{3}x_2^{9}x_3^{9}x_4^{7}x_5^{17}\\ &\quad + x_1^{3}x_2^{9}x_3^{9}x_4^{11}x_5^{13} + x_1^{3}x_2^{9}x_3^{11}x_4^{9}x_5^{13} + x_1^{3}x_2^{11}x_3^{9}x_4^{9}x_5^{13} \in \mathcal P_{(5,45)}.
\end{align*}
Hence, $w$ is strongly inadmissible.
\begin{align*}
x_1^{3}x_2^{13}x_3^{14}x_4^{3}x_5^{12} &= x_1^{3}x_2^{7}x_3^{14}x_4^{8}x_5^{13} + x_1^{3}x_2^{7}x_3^{14}x_4^{12}x_5^{9} + x_1^{3}x_2^{9}x_3^{14}x_4^{6}x_5^{13}\\ &\quad + x_1^{3}x_2^{9}x_3^{14}x_4^{12}x_5^{7} + x_1^{3}x_2^{11}x_3^{14}x_4^{4}x_5^{13} + x_1^{3}x_2^{11}x_3^{14}x_4^{5}x_5^{12}\\ &\quad + x_1^{3}x_2^{13}x_3^{13}x_4^{6}x_5^{10} + x_1^{3}x_2^{13}x_3^{14}x_4^{2}x_5^{13} + x_1^{5}x_2^{13}x_3^{13}x_4^{5}x_5^{9}\\ &\quad + Sq^1(x_1^{3}x_2^{7}x_3^{17}x_4^{6}x_5^{11} + x_1^{3}x_2^{7}x_3^{17}x_4^{10}x_5^{7} + x_1^{3}x_2^{11}x_3^{17}x_4^{2}x_5^{11}\\ &\quad + x_1^{3}x_2^{11}x_3^{17}x_4^{3}x_5^{10} + x_1^{3}x_2^{11}x_3^{17}x_4^{6}x_5^{7} + x_1^{3}x_2^{14}x_3^{13}x_4^{5}x_5^{9})\\ &\quad +  Sq^2(x_1^{3}x_2^{13}x_3^{13}x_4^{5}x_5^{9} + x_1^{5}x_2^{7}x_3^{14}x_4^{6}x_5^{11} + x_1^{5}x_2^{7}x_3^{14}x_4^{10}x_5^{7}\\ &\quad + x_1^{5}x_2^{11}x_3^{14}x_4^{2}x_5^{11} + x_1^{5}x_2^{11}x_3^{14}x_4^{3}x_5^{10} + x_1^{5}x_2^{11}x_3^{14}x_4^{6}x_5^{7}) \\ &\quad+ Sq^4(x_1^{3}x_2^{7}x_3^{14}x_4^{6}x_5^{11} + x_1^{3}x_2^{7}x_3^{14}x_4^{10}x_5^{7} + x_1^{3}x_2^{11}x_3^{14}x_4^{2}x_5^{11}\\ &\quad + x_1^{3}x_2^{11}x_3^{14}x_4^{3}x_5^{10} + x_1^{3}x_2^{11}x_3^{14}x_4^{6}x_5^{7}\\ &\quad + x_1^{3}x_2^{13}x_3^{14}x_4^{4}x_5^{7})  \  \mbox{ mod}(P_5^-((3)|^4)).
\end{align*}
Since $x_1^{5}x_2^{13}x_3^{13}x_4^{5}x_5^{9} \in \mathcal P_{(5,45)}$, the monomial $x_1^{3}x_2^{13}x_3^{14}x_4^{3}x_5^{12}$  is strongly inadmissible. The lemma is completely proved.
\end{proof}

\begin{proof}[Proof of Proposition \ref{mdd41}.] 
Let $x \in P_5^+((3)|^4)$ be an admissible monomial, then $x = X_{i,j}y^2$ with $1 \leqslant i < j \leqslant k$. Since $x$ is admissible, by Theorem \ref{dlcb1}, $y$ is admissible. We can see that if $z \in A(3)\cup C(3)$ such that $X_{i,j}z^2 \in P_5^+((3)|^4)$ and  $X_{i,j}z^2 \ne a_{4,t}$ for all $t,\, 1 \leqslant t \leqslant 94$, then either $X_{i,j}z^2$ is one of the monomials as given in Lemmas \ref{bdk1}(iv), \ref{bdd41}, \ref{bdd42}, or $X_{i,j}z^2$ is of the form $uv^{2^r}$, where $u$ is a monomial as given in one of Lemmas \ref{bdk1}, \ref{bdd30}, \ref{bdd31}, $v$ is a monomial in $P_5$ and $r$ is a suitable positive integer. Hence, by Proposition 
\ref{mdcb51}, $X_{i,j}z^2$ is inadmissible. Since $x = X_{i,j}y^2$ and $y$ is admissible, we have $x =  a_{4,t}$ for some $t,\, 1 \leqslant t \leqslant 94$. Hence, $B_5^+((3)|^4) \subset A(4)\cup C(4)$. The proposition follows.
\end{proof}

\subsection{Proofs of Theorems \ref{dl51} and \ref{dl20}}\

\medskip
By a similar computation as given in the previous lemmas, one gets the following.
\begin{lems}\label{bdd51}\  
	
\medskip
{\rm i)} The following monomials are strictly inadmissible:
\[ x_1^{3}x_2^{7}x_3^{24}x_4^{29}x_5^{30}\ \ x_1^{7}x_2^{3}x_3^{24}x_4^{29}x_5^{30}\ \ x_1^{7}x_2^{7}x_3^{25}x_4^{26}x_5^{28}\ \ x_1^{15}x_2^{15}x_3^{17}x_4^{18}x_5^{28}.\]  
	
{\rm ii)} The following monomials are strongly inadmissible:
\[ x_1^{3}x_2^{15}x_3^{21}x_4^{26}x_5^{28}\ \ x_1^{15}x_2^{3}x_3^{21}x_4^{26}x_5^{28}.\]
\end{lems}
\begin{proof} By using the Cartan formula, we obtain
\begin{align*}
x_1^{3}x_2^{7}x_3^{24}x_4^{29}x_5^{30} &= x_1^{2}x_2^{7}x_3^{25}x_4^{29}x_5^{30} + x_1^{2}x_2^{9}x_3^{23}x_4^{29}x_5^{30} + x_1^{3}x_2^{4}x_3^{27}x_4^{29}x_5^{30}\\ &\quad  + x_1^{3}x_2^{5}x_3^{25}x_4^{30}x_5^{30} +  Sq^1(x_1^{3}x_2^{7}x_3^{15}x_4^{29}x_5^{38} + x_1^{3}x_2^{7}x_3^{15}x_4^{33}x_5^{34}\\ &\quad  + x_1^{3}x_2^{7}x_3^{15}x_4^{37}x_5^{30} + x_1^{3}x_2^{7}x_3^{19}x_4^{29}x_5^{34} + x_1^{3}x_2^{7}x_3^{19}x_4^{33}x_5^{30}\\ &\quad  + x_1^{3}x_2^{7}x_3^{23}x_4^{29}x_5^{30}) +  Sq^2(x_1^{2}x_2^{7}x_3^{15}x_4^{29}x_5^{38} + x_1^{2}x_2^{7}x_3^{15}x_4^{37}x_5^{30}\\ &\quad  + x_1^{2}x_2^{7}x_3^{23}x_4^{29}x_5^{30} + x_1^{5}x_2^{7}x_3^{15}x_4^{30}x_5^{34} + x_1^{5}x_2^{7}x_3^{15}x_4^{34}x_5^{30}\\ &\quad  + x_1^{5}x_2^{7}x_3^{19}x_4^{30}x_5^{30}) +  Sq^4(x_1^{3}x_2^{4}x_3^{23}x_4^{29}x_5^{30} + x_1^{3}x_2^{5}x_3^{21}x_4^{30}x_5^{30} \\ &\quad + x_1^{3}x_2^{11}x_3^{15}x_4^{30}x_5^{30}) +  Sq^8(x_1^{3}x_2^{7}x_3^{15}x_4^{30}x_5^{30}) \  \mbox{ mod}(P_5^-((3)|^5)),\\
x_1^{7}x_2^{3}x_3^{24}x_4^{29}x_5^{30} &= x_1^{4}x_2^{3}x_3^{27}x_4^{29}x_5^{30} + x_1^{5}x_2^{2}x_3^{27}x_4^{29}x_5^{30} + x_1^{5}x_2^{3}x_3^{25}x_4^{30}x_5^{30}\\ &\quad  + x_1^{7}x_2^{2}x_3^{25}x_4^{29}x_5^{30} +  Sq^1(x_1^{7}x_2^{3}x_3^{15}x_4^{29}x_5^{38} + x_1^{7}x_2^{3}x_3^{15}x_4^{33}x_5^{34}\\ &\quad  + x_1^{7}x_2^{3}x_3^{15}x_4^{37}x_5^{30} + x_1^{7}x_2^{3}x_3^{19}x_4^{29}x_5^{34} + x_1^{7}x_2^{3}x_3^{19}x_4^{33}x_5^{30}\\ &\quad  + x_1^{7}x_2^{3}x_3^{23}x_4^{29}x_5^{30}) +  Sq^2(x_1^{7}x_2^{2}x_3^{15}x_4^{29}x_5^{38} + x_1^{7}x_2^{2}x_3^{15}x_4^{37}x_5^{30}\\ &\quad  + x_1^{7}x_2^{2}x_3^{23}x_4^{29}x_5^{30} + x_1^{7}x_2^{5}x_3^{15}x_4^{30}x_5^{34} + x_1^{7}x_2^{5}x_3^{15}x_4^{34}x_5^{30}\\ &\quad  + x_1^{7}x_2^{5}x_3^{19}x_4^{30}x_5^{30}) +  Sq^4(x_1^{4}x_2^{3}x_3^{23}x_4^{29}x_5^{30} + x_1^{5}x_2^{2}x_3^{23}x_4^{29}x_5^{30}\\ &\quad  + x_1^{5}x_2^{3}x_3^{21}x_4^{30}x_5^{30} + x_1^{11}x_2^{3}x_3^{15}x_4^{30}x_5^{30})\\ &\quad  +  Sq^8(x_1^{7}x_2^{3}x_3^{15}x_4^{30}x_5^{30}) \  \mbox{ mod}(P_5^-((3)|^5)),\\
x_1^{7}x_2^{7}x_3^{25}x_4^{26}x_5^{28} &= x_1^{4}x_2^{7}x_3^{28}x_4^{27}x_5^{27} + x_1^{4}x_2^{11}x_3^{28}x_4^{23}x_5^{27} + x_1^{5}x_2^{3}x_3^{27}x_4^{30}x_5^{28}\\ &\quad + x_1^{5}x_2^{3}x_3^{28}x_4^{27}x_5^{30} + x_1^{5}x_2^{6}x_3^{28}x_4^{27}x_5^{27} + x_1^{5}x_2^{10}x_3^{28}x_4^{23}x_5^{27}\\ &\quad + x_1^{5}x_2^{11}x_3^{19}x_4^{30}x_5^{28} + x_1^{5}x_2^{11}x_3^{22}x_4^{27}x_5^{28} + x_1^{5}x_2^{11}x_3^{26}x_4^{23}x_5^{28}\\ &\quad + x_1^{7}x_2^{3}x_3^{25}x_4^{30}x_5^{28} + x_1^{7}x_2^{3}x_3^{28}x_4^{25}x_5^{30} + x_1^{7}x_2^{5}x_3^{26}x_4^{27}x_5^{28}\\ &\quad + x_1^{7}x_2^{6}x_3^{28}x_4^{25}x_5^{27} + x_1^{7}x_2^{7}x_3^{24}x_4^{27}x_5^{28} + x_1^{7}x_2^{7}x_3^{25}x_4^{24}x_5^{30}\\ &\quad +  Sq^1(x_1^{7}x_2^{5}x_3^{28}x_4^{23}x_5^{29} + x_1^{7}x_2^{7}x_3^{21}x_4^{29}x_5^{28} + x_1^{7}x_2^{7}x_3^{25}x_4^{24}x_5^{29}\\ &\quad + x_1^{7}x_2^{7}x_3^{25}x_4^{25}x_5^{28} + x_1^{7}x_2^{7}x_3^{28}x_4^{23}x_5^{27}) +  Sq^2(x_1^{7}x_2^{3}x_3^{22}x_4^{23}x_5^{36}\\ &\quad + x_1^{7}x_2^{3}x_3^{23}x_4^{30}x_5^{28} + x_1^{7}x_2^{3}x_3^{28}x_4^{23}x_5^{30} + x_1^{7}x_2^{6}x_3^{28}x_4^{23}x_5^{27} \\ &\quad+ x_1^{7}x_2^{7}x_3^{19}x_4^{30}x_5^{28} + x_1^{7}x_2^{7}x_3^{22}x_4^{27}x_5^{28} + x_1^{7}x_2^{7}x_3^{26}x_4^{23}x_5^{28}\\ &\quad + x_1^{7}x_2^{7}x_3^{26}x_4^{24}x_5^{27}) + Sq^4(x_1^{4}x_2^{7}x_3^{28}x_4^{23}x_5^{27} + x_1^{5}x_2^{3}x_3^{23}x_4^{30}x_5^{28}\\ &\quad + x_1^{5}x_2^{3}x_3^{28}x_4^{23}x_5^{30} + x_1^{5}x_2^{6}x_3^{28}x_4^{23}x_5^{27} + x_1^{5}x_2^{7}x_3^{19}x_4^{30}x_5^{28}\\ &\quad + x_1^{5}x_2^{7}x_3^{22}x_4^{27}x_5^{28} + x_1^{5}x_2^{7}x_3^{26}x_4^{23}x_5^{28}\\ &\quad + x_1^{11}x_2^{5}x_3^{22}x_4^{23}x_5^{28} + x_1^{11}x_2^{5}x_3^{25}x_4^{30}x_5^{28})\\ &\quad +  Sq^8(x_1^{7}x_2^{5}x_3^{22}x_4^{23}x_5^{28} + x_1^{7}x_2^{5}x_3^{25}x_4^{30}x_5^{28}) \  \mbox{ mod}(P_5^-((3)|^5)),\\ 
x_1^{15}x_2^{15}x_3^{17}x_4^{18}x_5^{28} &= x_1^{8}x_2^{15}x_3^{23}x_4^{19}x_5^{28} + x_1^{8}x_2^{23}x_3^{15}x_4^{19}x_5^{28} + x_1^{9}x_2^{15}x_3^{23}x_4^{18}x_5^{28}\\ &\quad  + x_1^{9}x_2^{23}x_3^{15}x_4^{18}x_5^{28} + x_1^{11}x_2^{12}x_3^{23}x_4^{19}x_5^{28} + x_1^{11}x_2^{13}x_3^{23}x_4^{18}x_5^{28}\\ &\quad  + x_1^{11}x_2^{20}x_3^{15}x_4^{19}x_5^{28} + x_1^{11}x_2^{21}x_3^{15}x_4^{18}x_5^{28} + x_1^{15}x_2^{12}x_3^{19}x_4^{19}x_5^{28}\\ &\quad  + x_1^{15}x_2^{13}x_3^{19}x_4^{18}x_5^{28} + x_1^{15}x_2^{15}x_3^{16}x_4^{19}x_5^{28} + Sq^1(x_1^{15}x_2^{15}x_3^{15}x_4^{19}x_5^{28})\\ &\quad  + Sq^2(x_1^{15}x_2^{15}x_3^{15}x_4^{18}x_5^{28}) +  Sq^4(x_1^{15}x_2^{12}x_3^{15}x_4^{19}x_5^{28}\\ &\quad + x_1^{15}x_2^{13}x_3^{15}x_4^{18}x_5^{28})  + Sq^8(x_1^{8}x_2^{15}x_3^{15}x_4^{19}x_5^{28} + x_1^{9}x_2^{15}x_3^{15}x_4^{18}x_5^{28}\\ &\quad  + x_1^{11}x_2^{12}x_3^{15}x_4^{19}x_5^{28} + x_1^{11}x_2^{13}x_3^{15}x_4^{18}x_5^{28}) \  \mbox{ mod}(P_5^-((3)|^5)).
\end{align*}

Part i) follows from the above equalities. 

\medskip
We prove Part ii). We have
\begin{align*}
x_1^{3}x_2^{15}x_3^{21}x_4^{26}x_5^{28} &= x_1^{2}x_2^{15}x_3^{21}x_4^{25}x_5^{30} + x_1^{2}x_2^{15}x_3^{21}x_4^{26}x_5^{29} + x_1^{2}x_2^{15}x_3^{21}x_4^{28}x_5^{27}\\ &\quad + x_1^{2}x_2^{21}x_3^{15}x_4^{25}x_5^{30} + x_1^{2}x_2^{21}x_3^{15}x_4^{26}x_5^{29} + x_1^{2}x_2^{21}x_3^{15}x_4^{28}x_5^{27}\\ &\quad + x_1^{3}x_2^{8}x_3^{23}x_4^{29}x_5^{30} + x_1^{3}x_2^{10}x_3^{23}x_4^{29}x_5^{28} + x_1^{3}x_2^{13}x_3^{23}x_4^{24}x_5^{30}\\ &\quad + x_1^{3}x_2^{13}x_3^{23}x_4^{26}x_5^{28} + x_1^{3}x_2^{15}x_3^{16}x_4^{29}x_5^{30} + x_1^{3}x_2^{15}x_3^{18}x_4^{29}x_5^{28}\\ &\quad + x_1^{3}x_2^{15}x_3^{21}x_4^{24}x_5^{30} + g_1 + Sq^1(x_1^{3}x_2^{15}x_3^{15}x_4^{25}x_5^{34}\\ &\quad + x_1^{3}x_2^{15}x_3^{15}x_4^{30}x_5^{29} + x_1^{3}x_2^{15}x_3^{17}x_4^{23}x_5^{34} + x_1^{3}x_2^{15}x_3^{17}x_4^{30}x_5^{27}\\ &\quad + x_1^{3}x_2^{15}x_3^{22}x_4^{23}x_5^{29} + x_1^{3}x_2^{15}x_3^{22}x_4^{25}x_5^{27} + x_1^{3}x_2^{17}x_3^{15}x_4^{23}x_5^{34}\\ &\quad + x_1^{3}x_2^{17}x_3^{15}x_4^{30}x_5^{27} + x_1^{3}x_2^{17}x_3^{19}x_4^{23}x_5^{30} + x_1^{3}x_2^{17}x_3^{19}x_4^{26}x_5^{27}\\ &\quad + x_1^{3}x_2^{19}x_3^{17}x_4^{23}x_5^{30} + x_1^{3}x_2^{19}x_3^{17}x_4^{26}x_5^{27} + x_1^{3}x_2^{22}x_3^{15}x_4^{23}x_5^{29}\\ &\quad + x_1^{3}x_2^{22}x_3^{15}x_4^{25}x_5^{27} + x_1^{4}x_2^{15}x_3^{15}x_4^{29}x_5^{29} + x_1^{4}x_2^{15}x_3^{17}x_4^{29}x_5^{27}\\ &\quad + x_1^{4}x_2^{17}x_3^{15}x_4^{29}x_5^{27}) +  Sq^2(x_1^{5}x_2^{15}x_3^{19}x_4^{25}x_5^{27} + x_1^{5}x_2^{15}x_3^{19}x_4^{23}x_5^{29}\\ &\quad + x_1^{5}x_2^{15}x_3^{15}x_4^{27}x_5^{29} + x_1^{5}x_2^{15}x_3^{15}x_4^{23}x_5^{33} + x_1^{5}x_2^{17}x_3^{15}x_4^{27}x_5^{27}\\ &\quad + x_1^{5}x_2^{15}x_3^{17}x_4^{27}x_5^{27} + x_1^{5}x_2^{19}x_3^{15}x_4^{23}x_5^{29} + x_1^{5}x_2^{19}x_3^{15}x_4^{25}x_5^{27}\\ &\quad + x_1^{2}x_2^{15}x_3^{17}x_4^{30}x_5^{27} + x_1^{2}x_2^{15}x_3^{21}x_4^{26}x_5^{27} + x_1^{2}x_2^{17}x_3^{15}x_4^{30}x_5^{27}\\ &\quad + x_1^{2}x_2^{21}x_3^{15}x_4^{26}x_5^{27} + x_1^{2}x_2^{15}x_3^{15}x_4^{30}x_5^{29} + x_1^{2}x_2^{15}x_3^{21}x_4^{23}x_5^{30}\\ &\quad + x_1^{2}x_2^{21}x_3^{15}x_4^{23}x_5^{30}) +  Sq^4(x_1^{3}x_2^{15}x_3^{19}x_4^{25}x_5^{27} + x_1^{3}x_2^{15}x_3^{19}x_4^{23}x_5^{29}\\ &\quad + x_1^{3}x_2^{15}x_3^{15}x_4^{27}x_5^{29} + x_1^{3}x_2^{15}x_3^{15}x_4^{23}x_5^{33} + x_1^{3}x_2^{17}x_3^{15}x_4^{27}x_5^{27}\\ &\quad + x_1^{3}x_2^{15}x_3^{17}x_4^{27}x_5^{27} + x_1^{3}x_2^{19}x_3^{15}x_4^{23}x_5^{29} + x_1^{3}x_2^{19}x_3^{15}x_4^{25}x_5^{27})\\ &\quad + Sq^8(x_1^{3}x_2^{8}x_3^{15}x_4^{29}x_5^{30} + x_1^{3}x_2^{10}x_3^{15}x_4^{29}x_5^{28} + x_1^{3}x_2^{13}x_3^{15}x_4^{24}x_5^{30}\\ &\quad + x_1^{3}x_2^{13}x_3^{15}x_4^{26}x_5^{28}) \  \mbox{ mod}(P_5^-((3)|^5)),
\end{align*}
where 
\begin{align*}
g_1 &= x_1^{3}x_2^{17}x_3^{19}x_4^{27}x_5^{27} + x_1^{3}x_2^{19}x_3^{17}x_4^{27}x_5^{27} + x_1^{3}x_2^{15}x_3^{17}x_4^{25}x_5^{33} + x_1^{3}x_2^{17}x_3^{15}x_4^{25}x_5^{33}\\ &\quad + x_1^{3}x_2^{17}x_3^{17}x_4^{23}x_5^{33} + x_1^{3}x_2^{17}x_3^{17}x_4^{27}x_5^{29} + x_1^{3}x_2^{17}x_3^{21}x_4^{25}x_5^{27} + x_1^{3}x_2^{21}x_3^{17}x_4^{25}x_5^{27}\\ &\quad + x_1^{3}x_2^{17}x_3^{21}x_4^{23}x_5^{29} + x_1^{3}x_2^{21}x_3^{17}x_4^{23}x_5^{29} \in \mathcal P_{(5,93)}.
\end{align*}

Hence, the monomial $x_1^{3}x_2^{15}x_3^{21}x_4^{26}x_5^{28}$ is strongly inadmissible. 

By a similar computation, we get
\begin{align*}
x_1^{15}x_2^{3}x_3^{21}x_4^{26}x_5^{28} &= x_1^{8}x_2^{3}x_3^{23}x_4^{29}x_5^{30} + x_1^{10}x_2^{3}x_3^{23}x_4^{29}x_5^{28} + x_1^{13}x_2^{2}x_3^{23}x_4^{25}x_5^{30}\\ &\quad + x_1^{13}x_2^{2}x_3^{23}x_4^{26}x_5^{29} + x_1^{13}x_2^{2}x_3^{23}x_4^{28}x_5^{27} + x_1^{13}x_2^{3}x_3^{23}x_4^{24}x_5^{30}\\ &\quad + x_1^{13}x_2^{3}x_3^{23}x_4^{26}x_5^{28} + x_1^{15}x_2^{2}x_3^{21}x_4^{25}x_5^{30} + x_1^{15}x_2^{2}x_3^{21}x_4^{26}x_5^{29}\\ &\quad + x_1^{15}x_2^{2}x_3^{21}x_4^{28}x_5^{27} + x_1^{15}x_2^{3}x_3^{16}x_4^{29}x_5^{30} + x_1^{15}x_2^{3}x_3^{18}x_4^{29}x_5^{28}\\ &\quad + x_1^{15}x_2^{3}x_3^{21}x_4^{24}x_5^{30} + g_2 +  Sq^1(x_1^{15}x_2^{3}x_3^{15}x_4^{25}x_5^{34}\\ &\quad + x_1^{15}x_2^{3}x_3^{15}x_4^{30}x_5^{29} + x_1^{15}x_2^{3}x_3^{17}x_4^{23}x_5^{34} + x_1^{15}x_2^{3}x_3^{17}x_4^{30}x_5^{27}\\ &\quad + x_1^{15}x_2^{3}x_3^{22}x_4^{23}x_5^{29} + x_1^{15}x_2^{3}x_3^{22}x_4^{25}x_5^{27} + x_1^{15}x_2^{4}x_3^{15}x_4^{29}x_5^{29}\\ &\quad + x_1^{15}x_2^{4}x_3^{17}x_4^{29}x_5^{27} + x_1^{17}x_2^{3}x_3^{15}x_4^{23}x_5^{34} + x_1^{17}x_2^{3}x_3^{15}x_4^{30}x_5^{27}\\ &\quad + x_1^{17}x_2^{3}x_3^{19}x_4^{23}x_5^{30} + x_1^{17}x_2^{3}x_3^{19}x_4^{26}x_5^{27} + x_1^{17}x_2^{4}x_3^{15}x_4^{29}x_5^{27}\\ &\quad + x_1^{19}x_2^{3}x_3^{17}x_4^{23}x_5^{30} + x_1^{19}x_2^{3}x_3^{17}x_4^{26}x_5^{27} + x_1^{22}x_2^{3}x_3^{15}x_4^{23}x_5^{29}\\ &\quad + x_1^{22}x_2^{3}x_3^{15}x_4^{25}x_5^{27}) +  Sq^2(x_1^{15}x_2^{2}x_3^{15}x_4^{30}x_5^{29} + x_1^{15}x_2^{2}x_3^{17}x_4^{30}x_5^{27}\\ &\quad + x_1^{15}x_2^{2}x_3^{21}x_4^{23}x_5^{30} + x_1^{15}x_2^{2}x_3^{21}x_4^{26}x_5^{27} + x_1^{15}x_2^{5}x_3^{15}x_4^{23}x_5^{33}\\ &\quad + x_1^{15}x_2^{5}x_3^{15}x_4^{27}x_5^{29} + x_1^{15}x_2^{5}x_3^{17}x_4^{27}x_5^{27} + x_1^{15}x_2^{5}x_3^{19}x_4^{23}x_5^{29}\\ &\quad + x_1^{15}x_2^{5}x_3^{19}x_4^{25}x_5^{27} + x_1^{17}x_2^{2}x_3^{15}x_4^{30}x_5^{27} + x_1^{17}x_2^{5}x_3^{15}x_4^{27}x_5^{27}\\ &\quad + x_1^{19}x_2^{5}x_3^{15}x_4^{23}x_5^{29} + x_1^{19}x_2^{5}x_3^{15}x_4^{25}x_5^{27} + x_1^{21}x_2^{2}x_3^{15}x_4^{23}x_5^{30}\\ &\quad + x_1^{21}x_2^{2}x_3^{15}x_4^{26}x_5^{27}) + Sq^4(x_1^{15}x_2^{3}x_3^{15}x_4^{23}x_5^{33} + x_1^{15}x_2^{3}x_3^{15}x_4^{27}x_5^{29}\\ &\quad + x_1^{15}x_2^{3}x_3^{17}x_4^{27}x_5^{27} + x_1^{15}x_2^{3}x_3^{19}x_4^{23}x_5^{29} + x_1^{15}x_2^{3}x_3^{19}x_4^{25}x_5^{27}\\ &\quad + x_1^{17}x_2^{3}x_3^{15}x_4^{27}x_5^{27} + x_1^{19}x_2^{3}x_3^{15}x_4^{23}x_5^{29} + x_1^{19}x_2^{3}x_3^{15}x_4^{25}x_5^{27})\\ &\quad +  Sq^8(x_1^{8}x_2^{3}x_3^{15}x_4^{29}x_5^{30} + x_1^{10}x_2^{3}x_3^{15}x_4^{29}x_5^{28} + x_1^{13}x_2^{2}x_3^{15}x_4^{25}x_5^{30}\\ &\quad + x_1^{13}x_2^{2}x_3^{15}x_4^{26}x_5^{29} + x_1^{13}x_2^{2}x_3^{15}x_4^{28}x_5^{27} + x_1^{13}x_2^{3}x_3^{15}x_4^{24}x_5^{30}\\ &\quad + x_1^{13}x_2^{3}x_3^{15}x_4^{26}x_5^{28}) \  \mbox{ mod}(P_5^-((3)|^5)), 
\end{align*}
where
\begin{align*}
g_2 &= x_1^{15}x_2^{3}x_3^{17}x_4^{25}x_5^{33} + x_1^{17}x_2^{3}x_3^{15}x_4^{25}x_5^{33} + x_1^{17}x_2^{3}x_3^{17}x_4^{23}x_5^{33} + x_1^{17}x_2^{3}x_3^{17}x_4^{27}x_5^{29}\\ &\quad + x_1^{17}x_2^{3}x_3^{19}x_4^{27}x_5^{27} + x_1^{17}x_2^{3}x_3^{21}x_4^{23}x_5^{29} + x_1^{17}x_2^{3}x_3^{21}x_4^{25}x_5^{27} + x_1^{19}x_2^{3}x_3^{17}x_4^{27}x_5^{27}\\ &\quad + x_1^{21}x_2^{3}x_3^{17}x_4^{23}x_5^{29} + x_1^{21}x_2^{3}x_3^{17}x_4^{25}x_5^{27} \in \mathcal P_{(5,93)}.
\end{align*}
The above equalities show that the monomial $x_1^{15}x_2^{3}x_3^{21}x_4^{26}x_5^{28}$ is strongly inadmissible. The lemma is completely proved.
\end{proof}
\begin{proof}[Proof of Theorem \ref{dl51}.] 
Denote $A(d) = \{a_{d,t}: 1 \leqslant t \leqslant 55\}$ and $C(d) = \{a_{d,t}: 56 \leqslant t \leqslant 90\}$, where  $a_{d,t}$, $1 \leqslant t \leqslant 90$, are determined as in Section \ref{s5}. We prove that  $B_5^+((3)|^d) \subset A(d)\cup C(d)$ by induction on $d \geqslant 5$.
	
Let $x \in P_5^+((3)|^d)$ be an admissible monomial. Then, $\omega(x) = (3)|^d$ and $x = X_{i,j}y^2$ with $y$ a monomial in $P_5((3)|^{d-1})$ and $1 \leqslant i < j \leqslant 5$. Since $x$ is admissible, by Theorem \ref{dlcb1}, $y$ is also admissible. 
	
Let $d = 5$ and $z \in A(4)\cup C(4)\cup B_5^0((3)|^4)$. Based on Theorem \ref{dlcb1} we can check that if $X_{i,j}z^2 \in P_5^+((3)|^5)$ and  $X_{i,j}z^2 \ne a_{5,t}$ for all $t,\, 1 \leqslant t \leqslant 90$, then either $X_{i,j}z^2$ is one of the monomials as given in Lemma \ref{bdd51}, or $X_{i,j}z^2$ is of the form $uv^{2^r}$, where $u$ is a monomial as given in one of Lemmas \ref{bdk1}, \ref{bdd30}, \ref{bdd31}, \ref{bdd41}, \ref{bdd42}, $v$ is a monomial in $P_5$ and $r$ is a suitable integer. Hence, by Proposition \ref{mdcb51}, $X_{i,j}z^2$ is inadmissible. Since $x = X_{i,j}y^2$ is admissible and $y \in B_5((3)|^4) \subset A(4)\cup C(4)\cup B_5^0((3)|^4)$, we have $x =  a_{5,t}$ for some $t,\, 1 \leqslant t \leqslant 90$. Hence, $B_5^+((3)|^5) \subset A(5)\cup C(5).$
	
Suppose $d > 5$ and $B_5^+((3)|^{d-1}) \subset A(d-1)\cup C(d-1)$. Let $z \in A(d-1)\cup C(d-1)\cup B_5^0((3)|^{d-1})$. It is not difficult to check that if $X_{i,j}z^2 \in P_5^+((3)|^d)$ and  $X_{i,j}z^2 \ne a_{d,t}$ for all $t,\, 1 \leqslant t \leqslant 90$, then $X_{i,j}z^2$ is of the form $uw^{2^s}$, where $u$ is a monomial as given in one of Lemmas \ref{bdk1}, \ref{bdd30}, \ref{bdd31}, \ref{bdd41}, \ref{bdd42}, \ref{bdd51}, $w$ is a monomial in $P_5$ and $s$ is a suitable integer. By Proposition \ref{mdcb51}, $X_{i,j}z^2$ is inadmissible. Since $x = X_{i,j}y^2$ and $y$ is admissible, we have $x =  a_{d,t}$ for some $t,\, 1 \leqslant t \leqslant 90$. That means $B_5^+((3)|^d) \subset A(d)\cup C(d).$
	
Now we prove that the set $[A(d)\cup C(d)]_{(3)|^d}$ is linearly independent in $QP_5((3)|^d)$.
Consider $\langle [A(d)]_{(3)|^d}\rangle \subset QP_5((3)|^d)$ and $\langle [C(d)]_{(3)|^d}\rangle \subset QP_5((3)|^d)$. It is easy to see that for $1 \leqslant t \leqslant 55$,  $a_{d,t} = x_i^{2^d-1}f_i(b_{d,t})$ with $b_{d,t}$ an admissible monomial of degree $2(2^d-1)$ in $P_4$ and $1 \leqslant i \leqslant 5$. By Proposition \ref{mdmo}, $a_{d,t}$ is admissible. This implies that $\dim \langle [A(d)]_{(3)|^d}\rangle = 55$. Since $\nu(a_{d,t}) = 2^d-1$ for $1\leqslant t \leqslant 55$ and $\nu(a_{d,t}) < 2^d-1$ for $56\leqslant t \leqslant 90$, we obtain $\langle [A(d)]_{(3)|^d}\rangle \cap \langle [C(d)]_{(3)|^d}\rangle = \{0\}$. Hence, we need only to prove the set $[C(d)]_{(3)|^d}$ is linearly independent in $QP_5((3)|^d)$.
Suppose there is a linear relation
\begin{equation}\label{ctd51}
S:= \sum_{56\leqslant t \leqslant 90}\gamma_ta_{d,t} \equiv_{(3)|^d} 0,
\end{equation}
where $\gamma_t \in \mathbb F_2$. We denote $\gamma_{\mathbb J} = \sum_{t \in \mathbb J}\gamma_t$ for any $\mathbb J \subset \{t\in \mathbb N:56\leqslant t \leqslant 90\}$.
	
Let $w_{d,u},\, 1\leqslant u \leqslant 11$, be as in Section \ref{s5} and the homomorphism $p_{(i;I)}:P_5\to P_4$ which is defined by \eqref{ct23} for $k=5$. From our work \cite[Lemma 3.5]{sp}, we see that $p_{(i;I)}$ passes to a homomorphism from $QP_5((3)|^d)$ to $QP_4((3)|^d)$. By applying $p_{(i;j)}$, $1\leqslant i < j \leqslant 5,$ to (\ref{ctd51}), we obtain
\begin{align*}
&p_{(1;2)}(S) \equiv_{(3)|^d}  \gamma_{\{58,68\}}w_{d,7} + \gamma_{62}w_{d,10} \equiv_{(3)|^d} 0,\\
&p_{(1;3)}(S) \equiv_{(3)|^d} \gamma_{\{57,65\}}w_{d,6} + \gamma_{59}w_{d,7}   \equiv_{(3)|^d} 0,\\
&p_{(1;4)}(S) \equiv_{(3)|^d}  \gamma_{\{56,61,66,70,72,77\}}w_{d,5} + \gamma_{60}w_{d,7}  \equiv_{(3)|^d} 0,\\
&p_{(1;5)}(S) \equiv_{(3)|^d}  \gamma_{\{56,57,58,59,60,62,67,71,73,74,78\}}w_{d,4} + \gamma_{61}w_{d,7}  \equiv_{(3)|^d} 0,\\
&p_{(2;3)}(S) \equiv_{(3)|^d} \gamma_{\{64,65,68,69,76\}}w_{d,6} + \gamma_{80}w_{d,7}  \equiv_{(3)|^d} 0,\\
&p_{(2;4)}(S) \equiv_{(3)|^d} \gamma_{\{63,66,82,83\}}w_{d,5} + \gamma_{81}w_{d,7}   \equiv_{(3)|^d} 0,\\
&p_{(2;5)}(S) \equiv_{(3)|^d} \gamma_{\{63,64,67,80,81,84,85\}}w_{d,4} + \gamma_{82}w_{d,7}  \equiv_{(3)|^d} 0,\\
&p_{(3;4)}(S) \equiv_{(3)|^d} \gamma_{\{75,76,77,79,87,88,89\}}w_{d,5} + \gamma_{86}w_{d,7}  \equiv_{(3)|^d} 0,\\
&p_{(3;5)}(S) \equiv_{(3)|^d} \gamma_{\{75,78,86\}}w_{d,4} + \gamma_{87}w_{d,7} \equiv_{(3)|^d} 0,\\
&p_{(4;5)}(S) \equiv_{(3)|^d}  \gamma_{79}w_{d,4} + \gamma_{89}w_{d,7}  \equiv_{(3)|^d} 0.
\end{align*}
From these equalities, we get $\gamma_{59} = \gamma_{60} = \gamma_{61} = \gamma_{62} = \gamma_{79} = \gamma_{80} = \gamma_{81} = \gamma_{82} = \gamma_{86} = \gamma_{87} = \gamma_{89} = 0$, $\gamma_{65} = \gamma_{57}$,  $\gamma_{68} = \gamma_{58}$,  $\gamma_{78} = \gamma_{75}$. Then, by applying the homomorphism $p_{(1;(u,v))}$, $2\leqslant u < v \leqslant 4,$ to (\ref{ctd51}), we get
\begin{align*}
p_{(1;(2,3))}(S) &\equiv_{(3)|^d} \gamma_{64}w_{d,6} + \gamma_{69}w_{d,7} + \gamma_{76}w_{d,10}  \equiv_{(3)|^d} 0,\\
p_{(1;(2,4))}(S) &\equiv_{(3)|^d} \gamma_{\{56,58,63,66,70,71,72,77,83\}}w_{d,5} + \gamma_{70}w_{d,7} + \gamma_{77}w_{d,10} \equiv_{(3)|^d} 0,\\
p_{(1;(3,4))}(S) &\equiv_{(3)|^d} \gamma_{\{63,64,67,84,85\}}w_{d,3} + \gamma_{\{56,57,66,69,70,72,73,75,76,77,83,88,90\}}w_{d,5}\\ &\hskip3cm + \gamma_{\{66,74,83\}}w_{d,6} + \gamma_{72}w_{d,7}  \equiv_{(3)|^d} 0.
\end{align*}
Computing from the above equalities gives $\gamma_{64} = \gamma_{69} = \gamma_{70} = \gamma_{72} = \gamma_{76} = \gamma_{77}  = 0$ and $\gamma_{58} = \gamma_{57}$, $\gamma_{66} = \gamma_{56}$, $\gamma_{88} = \gamma_{75}$. Now, by applying $p_{(1;(u,5))}$, $2\leqslant u \leqslant 4,$ to (\ref{ctd51}), we obtain
\begin{align*}
p_{(1;(2,5))}(S) &\equiv_{(3)|^d} \gamma_{\{56,57,63,67,71,73,74,75,84,85,90\}}w_{d,4} + \gamma_{71}w_{d,7} + \gamma_{75}w_{d,10}  \equiv_{(3)|^d} 0,\\
p_{(1;(3,5))}(S) &\equiv_{(3)|^d} \gamma_{\{56,63,83\}}w_{d,2} + \gamma_{\{56,57,67,71,73,74,84\}}w_{d,4}\\ &\hskip3cm + \gamma_{\{67,74,84\}}w_{d,6} + \gamma_{73}w_{d,7} \equiv_{(3)|^d} 0,\\
p_{(1;(4,5))}(S) &\equiv_{(3)|^d} \gamma_{\{67,71,73,74,85,90\}}w_{d,4} + \gamma_{\{67,71,73,85,90\}}w_{d,5} + \gamma_{74}w_{d,7} \equiv_{(3)|^d} 0.
\end{align*}
By computing from the above equalities we get $\gamma_t = 0$ for all $t, \, 56 \leqslant t \leqslant 90$. The theorem is proved. 
\end{proof}

We need the following for the proof of Theorem \ref{dl20}.

\begin{props}\label{mdt5} The set 
$$\mathcal B_5 = \{(\mathcal I, \mathcal J) \in {\sf PSeq}_5^6: X_{(\mathcal I, \mathcal J)} \in B_5((3)|^6)\}\subset {\sf PInc}_5^6$$  
is compatible with $(3)|^{6}$.
\end{props}
\begin{proof} Let $(\mathcal H,\mathcal K)\in {\sf PSeq}_5^6$. From the monomials as given in Section \ref{s5}, we can easily check that if $(\mathcal H,\mathcal K)\in \mathcal B_5$, then $\mbox{rl}(\mathcal H)\leqslant 4$ and $\mbox{rl}(\mathcal K)\leqslant 4$. We prove $X_{(\mathcal H,\mathcal K)}$ is of the form \eqref{ctbd} for $k = 5$ and $\mathcal B = \mathcal B_5$. We prove the claim by induction on $X_{(\mathcal H,\mathcal K)}$ with respect to the order as given in Definition \ref{defn3}. Obviously, this claim is true if $X_{(\mathcal H,\mathcal K)}$ is admissible. Suppose $X_{(\mathcal H,\mathcal K)}$ is inadmissible and the claim is true for all $(\mathcal U,\mathcal V)\in {\sf PSeq}_5^6$ such that $X_{(\mathcal U,\mathcal V)}< X_{(\mathcal H,\mathcal K)}$. From the proofs of Propositions \ref{md52}, \ref{mdd532}, \ref{mdd41} and the proof of Theorem \ref{dl51} we see that $X_{(\mathcal H,\mathcal K)} = uw^{2^c}y^{2^{c+d}}$, where $u$ is a monomial of weight vector $(3)|^c$, $y$ is a monomial of weight vector $(3)|^e$ and $w$ is a monomial of weight vector $(3)|^d$ as given in one of Lemmas \ref{bdk1}, \ref{bdd30}, \ref{bdd31}, \ref{bdd41}, \ref{bdd42}, \ref{bdd51}. Here $c,\, e \geqslant 0$, $2 \leqslant d \leqslant 5$ and $c+d+e = 6$. Hence, there are $(\mathcal H_1,\mathcal K_1)\in {\sf PSeq}_5^c$, $(\mathcal H_2,\mathcal K_2)\in {\sf PSeq}_5^d$, $(\mathcal H_3,\mathcal K_3)\in {\sf PSeq}_5^e$ such that $u = X_{(\mathcal H_1,\mathcal K_1)}$, $w = X_{(\mathcal H_2,\mathcal K_2)}$, $y = X_{(\mathcal H_3,\mathcal K_3)}$ and $\mathcal H = \mathcal H_1|\mathcal H_2|\mathcal H_3$, $\mathcal K = \mathcal K_1|\mathcal K_2|\mathcal K_3$.  From the proofs of Lemmas \ref{bdk1}, \ref{bdd30}, \ref{bdd31}, \ref{bdd41}, \ref{bdd42}, \ref{bdd51} we see that 
$$w = X_{(\mathcal H_2,\mathcal K_2)} = \sum_{(\mathcal S,\mathcal T)\in \mathcal B_w} X_{(\mathcal S,\mathcal T)} + g \ \mbox{mod}(P_5^-((3)|^d)+\mathcal A(d-1)^+P_5), $$ 
where $\mathcal B_w$ is a set of suitable pairs $(\mathcal S,\mathcal T) \in {\sf PSeq}_5^d$ such that $\min \mathcal S = \min \mathcal H_2$, $\min \mathcal T \leqslant \min \mathcal K_2$,  $X_{(\mathcal S,\mathcal T)} < w$ and $g \in \mathcal P_{(5,n_d)}$ with $n_d = 3(2^d-1)$. Using the proof of Proposition \ref{mdcb51} we obtain
$$X_{(\mathcal H,\mathcal K)} = \sum_{(\mathcal S,\mathcal T)\in \mathcal B_w} uX_{(\mathcal S,\mathcal T)}^{2^c} y^{2^{c+d}}+ ug^{2^c}y^{2^{c+d}} \ \mbox{mod}(P_5^-((3)|^6)+\mathcal A(5)^+P_5),$$
where $ug^{2^c}y^{2^{c+d}} \in \mathcal P_{(5,n_6)}$, $uX_{(\mathcal S,\mathcal T)}^{2^c} y^{2^{c+d}} = X_{(\mathcal U,\mathcal V)} < X_{(\mathcal H,\mathcal K)}$ with $\mathcal U = \mathcal H_1|\mathcal S|\mathcal H_3$, $\mathcal V = \mathcal K_1|\mathcal T|\mathcal K_3$. Since $\min \mathcal S = \min \mathcal H_2$, $\min \mathcal T \leqslant \min \mathcal K_2$, we have $\min \mathcal U = \min\mathcal H$, $\min\mathcal V \leqslant \min\mathcal K$. The proposition now follows from the inductive hypothesis.
\end{proof}
\begin{proof}[Proof of Theorem \ref{dl20}] Let $n = 2^{d+s+t} + 2^{d + s} + 2^d -3$ and $m = 2^{s+t} + 2^s -2$. We have $\frac{n-5}2 = 2^{d-1+s+t} + 2^{d-1 + s} + 2^{d-2} + 2^{d-2} -4 $. By Theorems 1.3 and 1.4 in \cite{su2}, if $d \geqslant 6$, $s \geqslant 4$ and $t \geqslant 4$, then 
\[\dim(QP_5)_{\frac{n-5}2} = (2^5-1)\dim(QP_4)_{2^{s+t+1}+2^{s+1}-2}= 3(2^3-1)(2^4-1)(2^5-1).\]
Kameko's squaring operation
$(\widetilde {Sq}^0_*)_{(5,n)}: (QP_5)_n \longrightarrow (QP_5)_{\frac{n-5}2}$
is an epimorphism, hence by using Theorem \ref{dlwa}, we get
\begin{align*}4(2^3-1)(2^4-1)(2^5-1) &\leqslant \dim(QP_5)_n = \dim \mbox{\rm Ker}(\widetilde {Sq}^0_*)_{(5,n)} + \dim(QP_5)_{\frac{n-5}2}\\ &= \dim \mbox{\rm Ker}(\widetilde {Sq}^0_*)_{(5,n)} + 3(2^3-1)(2^4-1)(2^5-1).
\end{align*}
This implies that $\dim \mbox{\rm Ker}(\widetilde {Sq}^0_*)_{(5,n)}  \geqslant (2^3-1)(2^4-1)(2^5-1)$. 
	
By Proposition \ref{mdt5}, the set $\mathcal B_5 \subset {\sf PInc}_5^6$ 
is compatible with $(3)|^{6}$ and $|\mathcal B_5| = |B_5((3)|^6)| = 155$. By applying Theorem \ref{dlck2}, we obtain
\[\dim \mbox{\rm Ker}(\widetilde {Sq}^0_*)_{(5,n)} \leqslant |\mathcal B_{5}|\dim (QP_{3})_m = 155\dim (QP_{3})_m.\]
From Kameko \cite[Theorem 8.1]{ka}, we have $\dim (QP_{3})_m = 21$ for any $s,\, t \geqslant 2$. Hence, we get
\begin{align*}
\dim \mbox{\rm Ker}(\widetilde {Sq}^0_*)_{(5,n)} &\leqslant 155\dim (QP_{3})_m\\ 
&= 155\times 21 = (2^3-1)(2^4-1)(2^5-1).
\end{align*} 
Thus, $\dim \mbox{\rm Ker}(\widetilde {Sq}^0_*)_{(5,n)} = (2^3-1)(2^4-1)(2^5-1)$, for any $d \geqslant 6$ and $s,\, t \geqslant 4$. The theorem is proved.
\end{proof}

Combining this result and Theorem 1.6 in \cite{su} one gets the following, which is numbered as Corollary \ref{cor4240} in the introduction.

\begin{corls}\label{cor424} Let $n$ be as in Theorem $\ref{dl20}$. If $d \geqslant 6$ and $s,\, t \geqslant 4$, then
$$\dim (QP_5)_n = 4(2^3-1)(2^4-1)(2^5-1) = 13020.$$
Consequently, the inequality $\eqref{ct12}$ is an equality for $k = 5$ and $d\geqslant 6$.
\end{corls} 

\section{Appendix}\label{s5}

In this section, we list the admissible monomials of weight vector $(3)^d$ in $P_k$ with $k \leqslant 5$.

From the results of Kameko \cite[Theorem 8.1]{ka} and our work \cite[Proposition 5.4.2]{su2} we see that if $d \geqslant 4$, then $B_3^+((3)|^d) = \{(x_1x_2x_3)^{2^d-1}\}$ and 
$B_4^+((3)|^d) = \{w_u=w_{d,u} : 1\leqslant u \leqslant 11\},$ 
where

\medskip
\centerline{\begin{tabular}{lll}
$w_{1} = x_1x_2^{2^d-2}x_3^{2^d-1}x_4^{2^d-1}$ &\ \ \ &$w_{2} = x_1x_2^{2^d-1}x_3^{2^d-2}x_4^{2^d-1}$\cr  
$w_{3} = x_1x_2^{2^d-1}x_3^{2^d-1}x_4^{2^d-2}$ &\ &$w_{4} = x_1^{3}x_2^{2^d-3}x_3^{2^d-2}x_4^{2^d-1}$\cr     
$w_{5} = x_1^{3}x_2^{2^d-3}x_3^{2^d-1}x_4^{2^d-2}$ &\ &$w_{6} = x_1^{3}x_2^{2^d-1}x_3^{2^d-3}x_4^{2^d-2}$\cr   
$w_{7} = x_1^{7}x_2^{2^d-5}x_3^{2^d-3}x_4^{2^d-2}$ &&$w_{8} = x_1^{2^d-1}x_2x_3^{2^d-2}x_4^{2^d-1}$\cr  
$w_{9} = x_1^{2^d-1}x_2x_3^{2^d-1}x_4^{2^d-2}$ &&$w_{10} = x_1^{2^d-1}x_2^{3}x_3^{2^d-3}x_4^{2^d-2}$\cr  
$w_{11} = x_1^{2^d-1}x_2^{2^d-1}x_3x_4^{2^d-2}$ &&\cr
\end{tabular}} 

\medskip
The sets $B_4((3)|^d)$ and $B_5^0((3)|^d)$ are determined by using Proposition \ref{mdbs}.

\medskip
For any $d \geqslant 5$, $B_5^+((3)|^d) = \{a_t = a_{d,t} : 1\leqslant t \leqslant 90\},$ 
where

\medskip
\centerline{\begin{tabular}{lll}
$a_{1} = x_1x_2x_3^{2^d-2}x_4^{2^d-2}x_5^{2^d-1}$ &\ \ &$a_{2} = x_1x_2x_3^{2^d-2}x_4^{2^d-1}x_5^{2^d-2}$\cr  
$a_{3} = x_1x_2x_3^{2^d-1}x_4^{2^d-2}x_5^{2^d-2}$ &&$a_{4} = x_1x_2^{2}x_3^{2^d-4}x_4^{2^d-1}x_5^{2^d-1}$\cr  
$a_{5} = x_1x_2^{2}x_3^{2^d-3}x_4^{2^d-2}x_5^{2^d-1}$ &&$a_{6} = x_1x_2^{2}x_3^{2^d-3}x_4^{2^d-1}x_5^{2^d-2}$\cr  
$a_{7} = x_1x_2^{2}x_3^{2^d-1}x_4^{2^d-4}x_5^{2^d-1}$ &&$a_{8} = x_1x_2^{2}x_3^{2^d-1}x_4^{2^d-3}x_5^{2^d-2}$\cr  
$a_{9} = x_1x_2^{2}x_3^{2^d-1}x_4^{2^d-1}x_5^{2^d-4}$ &&$a_{10} = x_1x_2^{3}x_3^{2^d-4}x_4^{2^d-2}x_5^{2^d-1}$\cr  
$a_{11} = x_1x_2^{3}x_3^{2^d-4}x_4^{2^d-1}x_5^{2^d-2}$& &$a_{12} = x_1x_2^{3}x_3^{2^d-2}x_4^{2^d-4}x_5^{2^d-1}$\cr  
$a_{13} = x_1x_2^{3}x_3^{2^d-2}x_4^{2^d-1}x_5^{2^d-4}$ &&$a_{14} = x_1x_2^{3}x_3^{2^d-1}x_4^{2^d-4}x_5^{2^d-2}$\cr  
$a_{15} = x_1x_2^{3}x_3^{2^d-1}x_4^{2^d-2}x_5^{2^d-4}$ &&$a_{16} = x_1x_2^{2^d-2}x_3x_4^{2^d-2}x_5^{2^d-1}$\cr  
$a_{17} = x_1x_2^{2^d-2}x_3x_4^{2^d-1}x_5^{2^d-2}$ &&$a_{18} = x_1x_2^{2^d-2}x_3^{2^d-1}x_4x_5^{2^d-2}$\cr   
$a_{19} = x_1x_2^{2^d-1}x_3x_4^{2^d-2}x_5^{2^d-2}$ &&$a_{20} = x_1x_2^{2^d-1}x_3^{2}x_4^{2^d-4}x_5^{2^d-1}$\cr  
$a_{21} = x_1x_2^{2^d-1}x_3^{2}x_4^{2^d-3}x_5^{2^d-2}$ &&$a_{22} = x_1x_2^{2^d-1}x_3^{2}x_4^{2^d-1}x_5^{2^d-4}$\cr  
$a_{23} = x_1x_2^{2^d-1}x_3^{3}x_4^{2^d-4}x_5^{2^d-2}$ &&$a_{24} = x_1x_2^{2^d-1}x_3^{3}x_4^{2^d-2}x_5^{2^d-4}$\cr  
$a_{25} = x_1x_2^{2^d-1}x_3^{2^d-2}x_4x_5^{2^d-2}$ &&$a_{26} = x_1x_2^{2^d-1}x_3^{2^d-1}x_4^{2}x_5^{2^d-4}$\cr  
$a_{27} = x_1^{3}x_2x_3^{2^d-4}x_4^{2^d-2}x_5^{2^d-1}$ &&$a_{28} = x_1^{3}x_2x_3^{2^d-4}x_4^{2^d-1}x_5^{2^d-2}$\cr  
$a_{29} = x_1^{3}x_2x_3^{2^d-2}x_4^{2^d-4}x_5^{2^d-1}$ &&$a_{30} = x_1^{3}x_2x_3^{2^d-2}x_4^{2^d-1}x_5^{2^d-4}$\cr  
$a_{31} = x_1^{3}x_2x_3^{2^d-1}x_4^{2^d-4}x_5^{2^d-2}$ &&$a_{32} = x_1^{3}x_2x_3^{2^d-1}x_4^{2^d-2}x_5^{2^d-4}$\cr  
$a_{33} = x_1^{3}x_2^{2^d-3}x_3^{2}x_4^{2^d-4}x_5^{2^d-1}$ &&$a_{34} = x_1^{3}x_2^{2^d-3}x_3^{2}x_4^{2^d-1}x_5^{2^d-4}$\cr 
$a_{35} = x_1^{3}x_2^{2^d-3}x_3^{2^d-1}x_4^{2}x_5^{2^d-4}$ &&$a_{36} = x_1^{3}x_2^{2^d-1}x_3x_4^{2^d-4}x_5^{2^d-2}$\cr  
$a_{37} = x_1^{3}x_2^{2^d-1}x_3x_4^{2^d-2}x_5^{2^d-4}$ &&$a_{38} = x_1^{3}x_2^{2^d-1}x_3^{2^d-3}x_4^{2}x_5^{2^d-4}$\cr   
$a_{39} = x_1^{2^d-1}x_2x_3x_4^{2^d-2}x_5^{2^d-2}$ &&$a_{40} = x_1^{2^d-1}x_2x_3^{2}x_4^{2^d-4}x_5^{2^d-1}$\cr   
$a_{41} = x_1^{2^d-1}x_2x_3^{2}x_4^{2^d-3}x_5^{2^d-2}$ &&$a_{42} = x_1^{2^d-1}x_2x_3^{2}x_4^{2^d-1}x_5^{2^d-4}$\cr  
$a_{43} = x_1^{2^d-1}x_2x_3^{3}x_4^{2^d-4}x_5^{2^d-2}$ &&$a_{44} = x_1^{2^d-1}x_2x_3^{3}x_4^{2^d-2}x_5^{2^d-4}$\cr 
\end{tabular}}   
\centerline{\begin{tabular}{lll} 
$a_{45} = x_1^{2^d-1}x_2x_3^{2^d-2}x_4x_5^{2^d-2}$ &&$a_{46} = x_1^{2^d-1}x_2x_3^{2^d-1}x_4^{2}x_5^{2^d-4}$\cr 
$a_{47} = x_1^{2^d-1}x_2^{3}x_3x_4^{2^d-4}x_5^{2^d-2}$ &&$a_{48} = x_1^{2^d-1}x_2^{3}x_3x_4^{2^d-2}x_5^{2^d-4}$\cr 
$a_{49} = x_1^{2^d-1}x_2^{3}x_3^{2^d-3}x_4^{2}x_5^{2^d-4}$ &&$a_{50} = x_1^{2^d-1}x_2^{2^d-1}x_3x_4^{2}x_5^{2^d-4}$\cr  
$a_{51} = x_1^{3}x_2^{5}x_3^{2^d-6}x_4^{2^d-4}x_5^{2^d-1}$ &&$a_{52} = x_1^{3}x_2^{5}x_3^{2^d-6}x_4^{2^d-1}x_5^{2^d-4}$\cr   
$a_{53} = x_1^{3}x_2^{5}x_3^{2^d-1}x_4^{2^d-6}x_5^{2^d-4}$ &&$a_{54} = x_1^{3}x_2^{2^d-1}x_3^{5}x_4^{2^d-6}x_5^{2^d-4}$\cr   
$a_{55} = x_1^{2^d-1}x_2^{3}x_3^{5}x_4^{2^d-6}x_5^{2^d-4}$ &\ \ &$a_{56} = x_1x_2^{3}x_3^{2^d-3}x_4^{2^d-2}x_5^{2^d-2}$\cr    
$a_{57} = x_1x_2^{3}x_3^{2^d-2}x_4^{2^d-3}x_5^{2^d-2}$ &&$a_{58} = x_1x_2^{6}x_3^{2^d-5}x_4^{2^d-3}x_5^{2^d-2}$\cr  
$a_{59} = x_1x_2^{7}x_3^{2^d-6}x_4^{2^d-3}x_5^{2^d-2}$ &&$a_{60} = x_1x_2^{7}x_3^{2^d-5}x_4^{2^d-4}x_5^{2^d-2}$\cr  
$a_{61} = x_1x_2^{7}x_3^{2^d-5}x_4^{2^d-2}x_5^{2^d-4}$ &&$a_{62} = x_1x_2^{2^d-2}x_3^{3}x_4^{2^d-3}x_5^{2^d-2}$\cr 
$a_{63} = x_1^{3}x_2x_3^{2^d-3}x_4^{2^d-2}x_5^{2^d-2}$ &&$a_{64} = x_1^{3}x_2x_3^{2^d-2}x_4^{2^d-3}x_5^{2^d-2}$\cr 
$a_{65} = x_1^{3}x_2^{3}x_3^{2^d-4}x_4^{2^d-3}x_5^{2^d-2}$ &&$a_{66} = x_1^{3}x_2^{3}x_3^{2^d-3}x_4^{2^d-4}x_5^{2^d-2}$\cr  
$a_{67} = x_1^{3}x_2^{3}x_3^{2^d-3}x_4^{2^d-2}x_5^{2^d-4}$ &&$a_{68} = x_1^{3}x_2^{4}x_3^{2^d-5}x_4^{2^d-3}x_5^{2^d-2}$\cr  
$a_{69} = x_1^{3}x_2^{5}x_3^{2^d-6}x_4^{2^d-3}x_5^{2^d-2}$ &&$a_{70} = x_1^{3}x_2^{5}x_3^{2^d-5}x_4^{2^d-4}x_5^{2^d-2}$\cr 
$a_{71} = x_1^{3}x_2^{5}x_3^{2^d-5}x_4^{2^d-2}x_5^{2^d-4}$ &&$a_{72} = x_1^{3}x_2^{7}x_3^{2^d-7}x_4^{2^d-4}x_5^{2^d-2}$\cr  
$a_{73} = x_1^{3}x_2^{7}x_3^{2^d-7}x_4^{2^d-2}x_5^{2^d-4}$ &&$a_{74} = x_1^{3}x_2^{7}x_3^{2^d-3}x_4^{2^d-6}x_5^{2^d-4}$\cr  
$a_{75} = x_1^{3}x_2^{2^d-3}x_3x_4^{2^d-2}x_5^{2^d-2}$ && $a_{76} = x_1^{3}x_2^{2^d-3}x_3^{2}x_4^{2^d-3}x_5^{2^d-2}$\cr  
$a_{77} = x_1^{3}x_2^{2^d-3}x_3^{3}x_4^{2^d-4}x_5^{2^d-2}$ &&$a_{78} = x_1^{3}x_2^{2^d-3}x_3^{3}x_4^{2^d-2}x_5^{2^d-4}$\cr  
$a_{79} = x_1^{3}x_2^{2^d-3}x_3^{2^d-2}x_4x_5^{2^d-2}$ &&$a_{80} = x_1^{7}x_2x_3^{2^d-6}x_4^{2^d-3}x_5^{2^d-2}$\cr  
$a_{81} = x_1^{7}x_2x_3^{2^d-5}x_4^{2^d-4}x_5^{2^d-2}$ &&$a_{82} = x_1^{7}x_2x_3^{2^d-5}x_4^{2^d-2}x_5^{2^d-4}$\cr  
$a_{83} = x_1^{7}x_2^{3}x_3^{2^d-7}x_4^{2^d-4}x_5^{2^d-2}$ &&$a_{84} = x_1^{7}x_2^{3}x_3^{2^d-7}x_4^{2^d-2}x_5^{2^d-4}$\cr  
$a_{85} = x_1^{7}x_2^{3}x_3^{2^d-3}x_4^{2^d-6}x_5^{2^d-4}$ &&$a_{86} = x_1^{7}x_2^{2^d-5}x_3x_4^{2^d-4}x_5^{2^d-2}$\cr  
$a_{87} = x_1^{7}x_2^{2^d-5}x_3x_4^{2^d-2}x_5^{2^d-4}$ &&$a_{88} = x_1^{7}x_2^{2^d-5}x_3^{5}x_4^{2^d-6}x_5^{2^d-4}$\cr  
$a_{89} = x_1^{7}x_2^{2^d-5}x_3^{2^d-3}x_4^{2}x_5^{2^d-4}$ &&$a_{90} = x_1^{7}x_2^{11}x_3^{2^d-11}x_4^{2^d-6}x_5^{2^d-4}$\cr  
\end{tabular}} 

\section*{Acknowledgment} 

The first version of this work was written while the author was visiting the Viet Nam Institute for Advanced Study in Mathematics (VIASM) in November, 2019. He would like to thank the VIASM for the wonderful working condition and for the hospitality.

The author is very grateful to the referee for his valuable comments and suggestions which helped to improve the quality of the paper.

{}

\end{document}